\documentclass[preprint,12pt,nonatbib]{elsarticle}

\usepackage[bibliography=common]{apxproof}
\usepackage[maxbibnames=50]{biblatex}
\usepackage{color}
\usepackage{bbm}
\usepackage{amsmath}
\usepackage{amssymb}
\usepackage{algorithm,algorithmic}
\usepackage{soul}
\usepackage{amsthm}
\usepackage{caption}
\usepackage{mathrsfs}
\usepackage{graphicx}
\usepackage{mathtools}
\usepackage[mathcal]{eucal}
\usepackage{enumerate}
\usepackage[shortlabels]{enumitem}
\usepackage{verbatim}
\usepackage{fullpage} 
\usepackage{multicol}
\usepackage{mathdots}
\usepackage[table,xcdraw]{xcolor}
\usepackage{csquotes}
\usepackage[title]{appendix}
\usepackage{hyperref}
\usepackage{nicematrix}

\newcommand{\bbC}{\mathbb{C}}

\newcommand{\bbZ}{\mathbb{Z}}

\newcommand{\interior}[1]{
    {\kern0pt#1}^{\mathrm{o}}
}

\newcommand{\floor}[1]{\left\lfloor #1 \right\rfloor}
\newcommand{\ceil}[1]{\left\lceil #1 \right\rceil}
\newcommand{\bb}[1]{\left(#1\right)}

\DeclarePairedDelimiterX{\infdivx}[2]{(}{)}{%
  #1\;\delimsize\|\;#2%
}

\newcommand{\Z}{\mathbf{Z}}
\newcommand{\z}{\mathbf{z}}

\renewcommand{\phi}{\varphi}

\newtheorem{theorem}{Theorem}[section]
\newtheorem*{theorem*}{Theorem}
\newtheorem{proposition}[theorem]{Proposition}
\newtheorem{lemma}[theorem]{Lemma}
\newtheorem{corollary}[theorem]{Corollary}
\newtheorem*{corollary*}{Corollary}

\newtheorem{example}[theorem]{Example}

\theoremstyle{definition}
\newtheorem{definition}[theorem]{Definition}
\newtheorem*{definition*}{Definition}

\newtheorem{remark}[theorem]{Remark}
\newtheorem{assumption}[]{Assumption}
\newtheorem{conjecture}[]{Conjecture}
\newtheorem{claim}[theorem]{Claim}

\newtheorem{condition}[theorem]{Condition}

\begin{document}

\begin{frontmatter}



\title{Efficient Tensor Decomposition via Moment Matrix Extension}


\author{Bobby Shi\footnote[1]{Corresponding author: \href{mailto:bhshi@utexas.edu}{bhshi@utexas.edu}}}
\author{Julia Lindberg} 
\author{Joe Kileel}

\affiliation{organization={The University of Texas at Austin},
            city={Austin},
            state={TX},
            country={USA}}

\begin{abstract}
Motivated by a flurry of recent work on efficient tensor decomposition algorithms, 
we show that the celebrated moment matrix extension algorithm of 
Brachat, Comon, Mourrain, and Tsigaridas
\cite{brachat2010symmetric} 
for symmetric tensor canonical polyadic (CP) decomposition can be made efficient under the right conditions.  
We first show that the crucial property determining the complexity of the algorithm is the regularity of a target decomposition. This allows us to reduce the complexity of the vanilla algorithm, while also unifying results from previous works.  
We then show that for tensors in $S^d\bbC^{n+1}$ with $d$ even, low enough regularity can reduce finding a symmetric tensor decomposition to solving a system of linear equations.  For order-$4$ tensors we prove that generic tensors of rank up to $r=2n+1$ can be decomposed efficiently via moment matrix extension, exceeding the rank threshold allowed by simultaneous diagonalization.  We then formulate a conjecture that states for generic order-$4$ tensors of rank $r=O(n^2)$ the induced linear system is sufficient for efficient tensor decomposition, matching the asymptotics of existing algorithms and in fact improving the leading coefficient.  Towards this conjecture we give computer assisted proofs that the statement holds for $n=2, \dots, 17$. Next we demonstrate that classes of nonidentifiable tensors can be decomposed efficiently via the moment matrix extension algorithm, bypassing the usual need for uniqueness of decomposition.  Of particular interest is the class of monomials, for which the extension algorithm is not only efficient but also improves on existing theory by explicitly parameterizing the space of decompositions.  Code for implementations of the efficient algorithm for generic tensors and monomials are provided,  along with several numerical examples.
\end{abstract}



\end{frontmatter}



\section{Introduction}

Tensors can be thought of as multidimensional arrays.  Of particular interest in data science applications are symmetric tensors: tensors unchanged by permutation of indices \cite{comon2008symmetric}.  In many of these applications a central task is to compute the \textit{canonical polyadic (CP)}, or \textit{Waring}, decomposition of a symmetric tensor $\phi \in S^d \mathbb{C}^{n+1}$:
\begin{equation}
    \phi=\sum_{i=1}^r {\z}_i^{\otimes d},
\end{equation}
where $\z_i$ lie in $\bbC^{n+1}$. Rank-1 elements $\z^{\otimes d}$ span the vector space of symmetric tensors, so such a decomposition always exists \cite{Landsberg2011}.  Such a decomposition is useful in various applications, including blind source separation \cite{cardoso1998blind, cardoso1991super, deLathauwer2007fourth}, independent component analysis \cite{pmlr-v89-podosinnikova19a, HyvarinenKarhunenOja2001}, psychometrics \cite{CarrollChang1970}, magnetic resonance imaging \cite{BASSER2007220}, telecommunications \cite{CHEVALIER199927, DELATHAUWER2007322, veen1996algorithm, Sidiropoulos2000multiaccess}, signal processing \cite{donoho2001uncertainty}, latent variable estimation for Gaussian mixture models \cite{ge2015learning, pereira2022tensormomentsgaussianmixture, hsu2013learning}, moment estimation \cite{zhang2023moment}, and hidden Markov models \cite{Allman2008IdentifiabilityOP, JMLR:v15:anandkumar14b}.

Though finding such a decomposition (with $r$ minimal) is known to be NP-hard \cite{hillar2013most}, there has been a flurry of recent research that aims to develop \textit{efficient} algorithms for specific classes of tensors \cite{ma2016polynomial, koiran2025efficient, kothari2024overcompletetensordecompositionkoszulyoung, telen2022normal, kolda2015numerical, johnston2023computing}.  In these contexts \textit{efficient} often means only using polynomially many (in the number of tensor entries) linear algebra operations, and consequently are polynomial-time in the $n^d$.  
These algorithms largely build on a simultaneous diagonalization procedure that realizes tensor components as common eigenvectors of some set of matrices.

\subsection{Our contributions}
In this paper we refine and improve the analysis of the algebraic algorithm of \cite{brachat2010symmetric}. 
At a high level, said algorithm reduces symmetric tensor decomposition over $\bbC$ to a \textit{moment matrix extension} problem. The extension procedure involves solving a system of polynomial equations in a set of \textit{moment variables}.  In Section \ref{sec:mmextension} our first contribution is to show that the key measure of complexity in the extension algorithm is the \textit{regularity} \cite{Eisenbud2004} of a target decomposition.  When the regularity of a target decomposition is sufficiently small relative to the order of the tensor, we demonstrate that the polynomial system can reduce to a linear system. Under additional conditions solving this linear system is sufficient and the task of symmetric tensor decomposition can be performed efficiently. Therefore, low regularity is the crucial property for our perspective on efficient tensor decomposition -- more so than low rankness or uniqueness of the decomposition, in contrast to existing literature.

Low regularity of decompositions holds for almost all tensors (i.e., generically) of low enough rank.  
On the other hand, showing that for such tensors the task of tensor decomposition reduces to solving a linear system of equations is more challenging.
In Section \ref{sec:lowrankeven} we prove that order-$4$ tensor decomposition generically reduces to solving a linear system for ranks up to $2n+1$, exceeding the ranks provably allowed by simultaneous diagonalization. 
Further, we formulate a conjecture giving sufficient conditions under which symmetric tensor decomposition generically reduces to a linear system at higher ranks.
We provide evidence for the conjecture’s validity through a computer-assisted proof verifying its correctness up to $n = 17$.
An affirmative answer to the conjecture would imply that the moment matrix extension method can decompose almost all order-$4$ tensors of rank up to $r=O(n^2)$ efficiently. The precise leading coefficient in this bound, which we work out explicitly, improves existing work concerning different tensor decomposition algorithms.  Moreover, the sufficiency of the linear system provides a new certificate for uniqueness of decomposition for specific tensors.

In addition, we show that classes of nonidentifiable tensors can also be decomposed efficiently by the extension algorithm. This is in contrast to previous work on efficient tensor decomposition, which requires that the underlying decomposition be unique. 
In particular, we show that two classes of nonidentifiable tensors, tensors corresponding to three collinear points (Section \ref{sec:collinear}) and arbitrary monomials (Section \ref{sec:monomials}), can be decomposed efficiently. 

We provide supplementary code with an implementation of the algorithm using only linear algebra (with no symbolic algebra packages needed) at:
\vspace{-0.5em}
\begin{center}
\url{https://github.com/rhshi/TensorDecomposition.jl}
\end{center}

\subsection{Related work}
Symmetric tensor decomposition has been classically studied from the perspective of algebra and geometry as the problem of finding a Waring decomposition of a homogeneous polynomial \cite{iarrobino1999power, Landsberg2011}.  Using tools from projective algebraic geometry, this approach has led to algebraic algorithms via apolarity theory by casting tensor decomposition as a root-finding problem \cite{Sylvester1886, oeding2013eigenvectors} (e.g., Sylvester's method, the catalecticant method, and more sophisticated eigenvector-based methods using Young flattenings \cite{landsberg2013equations}). Closely related to  apolarity based methods are normal form algorithms \cite{telen2022normal, beltran2019pencil, bernardi2020waring, brachat2010symmetric}, which consist of more systematic procedures for solving relevant polynomial systems, as well as algorithms using generating polynomials \cite{nie2017generating}. The method of generating polynomials reduces tensor decomposition to solving a different, but related, polynomial system of degree two.

The mentioned decomposition algorithms are often not efficient, as they reduce tensor decomposition to solving nonlinear polynomial systems of equations where the number of variables (and possibly the degrees) increase with the number of tensor entries.  
On the other hand, efficient tensor decomposition is known to be possible for certain classes of tensors. 
This has been well-studied in the case of order-$3$ and $4$ tensors.  Indeed for (asymmetric) order-$3$ tensors of size $n \times n \times n$, generic tensors of rank $r\leq n$ are efficiently decomposable via the classic simultaneous diagonalization algorithm \cite{Hars1970, leurgans1993decomposition}.  For order-$4$ tensors, the FOOBI algorithm for certain partially symmetric $n\times n\times n\times n$ tensors \cite{deLathauwer2007fourth, cardoso1991super} can efficiently decompose overcomplete tensors of rank $r>n$; see \cite{johnston2023computing} for a recent treatment, including the symmetric setting.  
Recently, more sophisticated algorithms have focused on improving rank bounds for efficient tensor decomposition of asymmetric tensors \cite{kothari2024overcompletetensordecompositionkoszulyoung, koiran2024uniqueness, koiran2025efficient, telen2022normal}.  Related is the question of identifiability, in particular the search for specific criteria for tensor identifiability \cite{kruskal1977three, chiantini2017effective, LovitzPetrov2023, DOMANOV2017342, chiantini2017generic, chiantini2015algorithm, Angelini2018}.

Alternative approaches in the study of efficient tensor decomposition include iterative algorithms \cite{kolda2009tensor} such as alternating least squares \cite{kolda2015numerical}, methods based on power iterations \cite{kolda2011shifted}, methods based on approximations \cite{ribot2024decomposingtensorsrankoneapproximations}, methods combining the catalecticant method with power iterations \cite{kileel2024subspacepowermethodsymmetric}, methods based on homotopy continuation \cite{BERNARDI201778}, as well as algorithms dependent on random models \cite{ge2015decomposing, ma2016polynomial, hopkins2016fast, hopkins2019robust, cao2020analysisstochasticalternatingsquares}. 
These numerical methods rely on nonconvex optimization and often do not have global guarantees, although analysis of their nonconvex optimization landscapes has been achieved in some cases (e.g., \cite{kileel2021landscape}).  
There also exist numerical methods involving truncations/relaxations of the generalized moment problem \cite{nie2014tmp, gamertsfelder2025effectivegeneralizedmomentproblem}, but these require extra real and/or positivity constraints on the tensor.

Part of the present paper deals with efficient decomposition of nonidentifiable tensors.  Nonidentifiability has historically been considered from the algebraic perspective \cite{ranestad2000varieties, CARLINI2017630, BUCZYNSKA201345waring}. A case-by-case analysis for tensors in low dimensions and with low rank is given in \cite{MOURRAIN2020347}.

\subsection{Paper outline}
 Section \ref{sec:preliminaries} contains notation and basic definitions.  We discuss dehomogenization -- viewing a homogeneous polynomial of degree $d$ in $n+1$ variables as an inhomogeneous polynomial of degree at most $d$ in $n$ variables -- as useful for moment matrix extension.  We also describe the simultaneous diagonalization procedure and connect its performance to the regularity of sets of points.

 Section \ref{sec:mmextension} describes the moment matrix extension algorithm.  We refine the algorithm to connect its complexity explicitly to the regularities of target decompositions, generalizing previous improvements to the algorithm of \cite{brachat2010symmetric}.  We then begin our study of efficient decomposition; as examples we describe tensors with decompositions of extremely low regularity and binary tensors.

 Section \ref{sec:lowrankeven} contains most of the main results of the paper.  We specify the conditions needed for the moment matrix extension algorithm to reduce to solving a linear system of equations.  For $d=4$ we provide counts of the number of linear equations and moment variables, and prove generic sufficiency of the linear system for $r \leq 2n+1$.
 We then present a sharper conjecture for efficient decomposition up to $r=O(n^2)$; computational evidence towards this conjecture is provided. 
 We also touch on nonidentifiable tensors in this setting.

 Section \ref{sec:monomials} presents an algorithm for efficient decomposition of monomials via the moment matrix extension method.  
 The method 
gives a new explicit  parameterization for the space of decompositions of monomials. 
 As another main result of the paper, it is not only able to algorithmically decompose monomials but also lets us provide and reprove theory.

 Throughout the paper we defer some proofs and examples to the \hyperref[app:a]{Appendix}; many of these proofs are largely calculations.

\section{Preliminaries}\label{sec:preliminaries}
We use bold lowercase $\z$ to denote vectors and uppercase $\Z$ to denote matrices.  With some abuse of notation we also use $\Z$ to denote a {set} of points $\{\z_1, \dots, \z_s\}$; just take $\Z$ to be the matrix with rows the points.

We will always use $\phi$ to denote a symmetric tensor, with varying format.  The notation $S^d\bbC^{n+1}$ denotes the $d$th symmetric power of the vector space $\bbC^{n+1}$, i.e.,  the vector space of symmetric tensors of dimension $n+1$ and order $d$.  The dimension of $S^d\bbC^{n+1}$ is $\binom{n+d}{n}$.  The following identifications are well known:
\begin{align*}
    & S^d\bbC^{n+1}=\left\{\begin{array}{c}\textrm{symmetric tensors}\\ \textrm{of order $d$ in $n+1$ variables}\end{array}\right\}\\
    \leftrightsquigarrow \quad & \bbC[x_0, \dots, x_n]_d=\left\{\begin{array}{c}\textrm{homogeneous polynomials}\\ \textrm{of degree $d$ in $n+1$ variables}\end{array}\right\}\\
    \leftrightsquigarrow \quad & \bbC[x_1, \dots, x_n]_{\leq d}=\left\{\begin{array}{c} \textrm{inhomogeneous polynomials of degree at most $d$ in $n$ variables}\end{array}\right\}.
\end{align*}
To index $\phi$ we will use the third identification.  
When a tensor $\phi\in S^d\bbC^{n+1}$ is  viewed as a multiway array, it is oftened indexed by a $d$-tuple $(i_1, \dots, i_d)$ where $i_k\in \{0, \dots, n\}$.
We index $\phi$ by a monomial of degree at most $d$ in $n$ variables $x^\alpha=x_1^{\alpha_1}\dots x_n^{\alpha_n}$, where $\alpha_j$ counts the number of times $j$ appears in the $d$-tuple ($0$ is ignored).  Equivalently, we index $\phi$ directly with the exponent vector $\alpha\in \bbZ_{\geq 0}^n$.  

For an exponent vector $\alpha\in \bbZ^n_{\geq 0}$ we call $|\alpha|:=\sum_{i=1}^n \alpha_i$ the \textit{size} of $\alpha$.  This is the same as the \textit{degree} of $x^\alpha$.  For a finite index set $S\subset \bbZ_{\geq 0}^n$ we write $|S|$ to be the cardinality of $S$, $S_k:=\{\alpha\in S\mid |\alpha|=k\}$, and $\operatorname{deg}(S):=\max_{\alpha\in S} |\alpha|$. We will also take $S$ to be the set of corresponding monomials $x^\alpha$.  For $\gamma$ in a monomial set $S$ we write $\operatorname{pos}(\gamma)$ to denote the position of $\gamma$ in $S$ in the graded lexicographic order (and similarly if we write $x^\gamma$ in $S$).

\begin{definition}\label{def:catalecticant}
    For $\phi\in S^d\bbC^{n+1}$ and $k\leq d$, the $k$th-\textit{catalecticant matrix} of $\phi$ is: 
    \[\operatorname{Cat}_k(\phi):=\bb{\phi_{\alpha+\beta}}_{\substack{|\alpha|\leq d-k\\|\beta|\leq k}} \in \bbC^{\binom{n+d-k}{n}\times \binom{n+k}{n}}.\]
    $\operatorname{Cat}_k(\phi)$ can equivalently be viewed as a linear operator $S^{k}\bbC^{n+1}\to S^{d-k}\bbC^{n+1}$.
\end{definition}

We note that catalecticant matrices are equivalent to the usual $k$-mode symmetric flattenings of a symmetric tensor.  
\begin{definition}
    Let $\Z\in \bbC^{s\times (n+1)}$ be the matrix with rows $\z_1, \dots, \z_s \in \mathbb{C}^{n+1}$.  For an index set $S\subset \bbZ_{\geq 0}^n$ and an integer $k\geq \operatorname{deg}(S)$, define the $(S, k)$-\textit{Vandermonde matrix} $\Z_{S, k}$ as:
    \[\Z_{S, k}:=\bb{\z_{i, 0}^{k-|\alpha|}\z_{i, 1}^{\alpha_1}\dots \z_{i, n}^{\alpha_n}}_{\substack{i=1, \dots, s\\ \alpha\in S}} \in \bbC^{s \times |S|}.\]
    When $S$ consists of all monomials of degree at most $k$, then $\Z_{S,k}$ is the usual $k^{\text{th}}$ Vandermonde matrix of $\Z$ and is denoted $\Z_k$.
\end{definition}
Note that $\Z_k$ can also be seen as the $k^{\text{th}}$ row-wise Khatri-Rao product of $\Z$.

\begin{definition}
    The \textit{Hilbert function} of a graded ideal $I \subseteq \bbC[x_0, \dots, x_n]$ is the function
    \[k\mapsto \operatorname{dim}_{\bbC}(\bbC[x_0, \dots, x_n]_k/I_k),\]
    where $I_k = I \cap \bbC[x_0,\ldots,x_n]_k$.
\end{definition}

\begin{definition}\label{Apolarity}
    Let $\mathbf{v}_0, \dots, \mathbf{v}_n$ a basis for $\bbC^{n+1}$ and dual basis $\mathbf{u}_0, \dots, \mathbf{u}_n$ for $\bb{\bbC^{n+1}}^*$.  A basis for $S^k\bbC^{n+1}$ is given by $\mathbf{v}_{\alpha}, |\alpha|=k$, and similarly for $\bb{S^k\bbC^{n+1}}^*$.  For $|\alpha|=k, |\alpha'|=k'$ the action is given via $\mathbf{u}_{\alpha'}\circ \mathbf{v}_{\alpha}=\mathbf{v}_{\alpha-\alpha'}$ if $\alpha-\alpha'\geq 0$ entrywise; otherwise it is zero.  Extend this by linearity to define an action of $\bb{S\bbC^{n+1}}^*$ on $S\bbC^{n+1}$.
    
    An element $g\in \bb{S \bbC^{n+1}}^*$ is \textit{apolar} to $\phi\in S^d\bbC^{n+1}$ if $g \circ \phi=0$.  Define the \textit{apolar ideal}
    \[\phi^\perp=\{g\in \bb{S \bbC^{n+1}}^*\mid g\circ \phi=0\}.\]
\end{definition}

We write $I_{\Z}$ to be the graded ideal of forms vanishing on $\Z=\{\z_1, \dots, \z_s\}\subset \mathbb{P}^n$.

\begin{lemma}\label{lem: hilbert function is ranks}
    The Hilbert function $h_{\phi}$ of $\phi^\perp$ is $k\mapsto \operatorname{rank}\bb{\operatorname{Cat}_k(\phi)}$.  Similarly, the Hilbert function $h_{\Z}$ of $I_{\Z}$ is $k\mapsto \operatorname{rank}\bb{\Z_k}$. 
\end{lemma}
\begin{proof}
    It is known that $\phi^\perp_k=\operatorname{ker}\bb{\operatorname{Cat}_k(\phi)}$ when $\operatorname{Cat}_k(\phi)$ is viewed as a linear operator \cite{Landsberg2011}. Similarly, a $k$-form vanishes on $\z_1, \dots, \z_s$ if and only if it vanishes via evaluation on $\z_1^{\otimes k}, \dots \z_s^{\otimes k}$; then the degree-$k$ piece is exactly $\operatorname{ker}\bb{\Z_k}$.  The conclusions follow.
\end{proof}
It is clear that for $k>d$, $h_\phi(k)=0$ and that $h_\phi(k)=h_\phi(d-k)$; these are both well-known properties \cite{iarrobino1999power}.  For a set of points $\Z$, $h_\Z$ is nondecreasing in $k$.  Moreover, the function stabilizes at $|\Z|$, i.e., there exists $k$ such that $h_{\Z}(k)=|\Z|$ and for all $k'\geq k$, $h_{\Z}(k')=|\Z|$ \cite{iarrobino1999power}.  An important property of sets of points that we will use is the \textit{regularity} \cite{MOURRAIN2020347, Eisenbud2004}.
\begin{definition}
    For a set of points $\Z=\{\z_1, \dots, \z_s\} \subset \bbC^{n+1}$ the \textit{regularity} of $\Z$, $\rho(\Z)$, is the smallest integer $k$ such that $\z_1^{\otimes k}, \dots, \z_s^{\otimes k}$ are linearly independent as elements of $S^k\bbC^{n+1}$.
\end{definition}

    Equivalently, the regularity of $\Z$ is the smallest integer $k$ such that $\Z_{k}$ is rank $s$. By Lemma \ref{lem: hilbert function is ranks}, this is the integer $k$ at which $h_{\Z}$ stabilizes.  This is also equivalent to the \textit{Castelnuovo-Mumford regularity} of $\mathbb{C}[x_0, \dots, x_n]/I_{\Z}$ \cite{Eisenbud2004}.

    \begin{definition}
        A set of points $\Z=\{\z_1, \dots, \z_s\}\subset \mathbb{C}^{n+1}$ \textit{decomposes} a tensor $\phi\in S^d\bbC^{n+1}$ if $\phi=\sum_{i=1}^s \z_i^{\otimes d}$.  The smallest $r$ such that there exists $\Z=\{\z_1, \dots, \z_r\}$ decomposing $\phi$ is the \textit{symmetric rank} (or \textit{Waring rank}) of $\phi$.  The points $\z_1, \dots, \z_r$ give a \textit{minimal rank decomposition} of $\phi$.
    \end{definition}

If a set of points $\Z=\{\z_1, \dots, \z_s\} \subset \mathbb{P}^{n+1}$ decomposes $\phi$ then $\operatorname{Cat}_k(\phi)=\Z_{d-k}^\top \Z_k$ for $0\leq k\leq d$.  
Then for such $\Z$ and $\phi$ it holds that $h_{\phi}(k)\leq h_{\Z}(k)$ by submultiplicativity of the matrix rank.  Although the minimal rank decomposition for a given tensor $\phi$ may not be unique, the rank of $\phi$ is invariant under the $\textrm{GL}(n+1)$ action that acts via $\mathbf{M}\cdot \phi=\sum_{i=1}^s (\mathbf{M}\z_i)^{\otimes d}$.  This action can be defined and performed independent of any given decomposition.

\begin{definition}
    A tensor is \textit{identifiable} if it has a unique minimum rank decomposition (up to permutation of factors).
\end{definition}

Identifiability is closely related to the \textit{variety of sums of powers} $\operatorname{VSP}(\phi, s)$, the closure of the set $s$ points decomposing $\phi$ \cite{ranestad2000varieties}  (usually defined with respect to $s$ points in the Hilbert scheme $\operatorname{Hilb}_s \bb{\mathbb{P}^n}$).  An identifiable tensor is one such that $\operatorname{VSP}(\phi, \operatorname{rank}(\phi))$ is a point.

\subsubsection*{Dehomogenization}
The method of indexing we have chosen is especially useful when we \textit{dehomogenize} with respect to the $0$th index. In this case, the index $0$ -- corresponding to variable $x_0$ -- is special.  If $\Z=\{\z_1, \dots, \z_s\} \subset \bbC^{n+1}$ is a set of points giving a size-$s$ decomposition for $\phi$ with the property that $\z_{i, 0}\neq 0$ for $1 \leq i \leq s$, then we can without loss of generality assume that $\z_{i, 0}=1$.  
In this case, for appropriate $\lambda_1, \dots, \lambda_s \in \bbC$, we can write $\phi=\sum_{i=1}^s \lambda_i\z_i^{\otimes d}$.  

At the level of points, for an index set $S\subset \bbZ_{\geq 0}^n$ by dehomogenizing we can drop the dependency on $k$ to write any Vandermonde matrix. In this case the $S$-Vandermonde matrix of $\Z$ is 
\[\Z_S:=\bb{\z_i^\alpha}_{\substack{i=1, \dots, s\\ \alpha\in S}} \in \bbC^{s\times |S|},\] 
where $\z_i^\alpha:= \z_{i, 1}^{\alpha_1}\dots  \z_{i, n}^{\alpha_n}$ as $\z_{i, 0}=1$, and $\Z_k$ is the usual $k^{\text{th}}$ Vandermonde matrix of $\Z$. For $k'\leq k$, $\Z_{k'}$ is a submatrix of $\Z_k$ via restricting indices to size at most $k'$.

At the level of tensors, a tensor $\phi$ can be viewed as an object not only in $S^d\bbC^{n+1}$ but also in a larger vector space of objects that restrict to $\phi$ at specific degrees.  This is the same as viewing a inhomogeneous polynomial of degree $d$ in $n$ variables as an object in the space of all polynomials, where polynomials of larger degree can restrict to the given polynomial by keeping only terms of degree at most $d$.

If $\z_1, \dots, \z_s$ gives a size-$s$ decomposition of $\phi$ with coefficients $\lambda_1, \dots, \lambda_s$ then a family of tensors can be produced via
\[\sum_{i=1}^s \lambda_i \z_i^{\otimes d'}, \quad d'\geq 0.\]
If $d'=d$ then this is obviously $\phi$ and for $d'<d$ these are certain slices of $\phi$, i.e., fixing $d-d'$ indices of $\phi$.  More interesting is if $d'>d$; in this case $\phi$ is a slice of these higher order tensors, i.e., the restriction of these higher order tensors to multiindices of size at most $d$ is precisely $\phi$.

\subsubsection*{Model of computation}
It is known that tensor decomposition is NP-hard in the worst case \cite{hillar2013most}.  In this work we are mainly focused on the possibility of \textit{efficient} algorithms; we will be using linear algebraic operations as subroutines (on polynomially sized inputs in the size of the tensor), including eigendecompositions and solving linear systems.  Following \cite{koiran2025efficient} and other work on efficient tensor decomposition we assume these operations can be performed in exact arithmetic.  

\subsection{Warm-up: simultaneous diagonalization}\label{sec:simultaneous_diagonalization}
As a warm-up we describe the classic simultaneous diagonalization algorithm for symmetric tensor decomposition (also known as \textit{Jennrich's algorithm}), see e.g., \cite{leurgans1993decomposition}.  

Write $D:=\floor{(d-1)/2}$.  Let $\Z=\{\z_1, \dots, \z_s\} \subset \bbC^{n+1}$ be a size-$s$ decomposition of $\phi\in S^d\bbC^{n+1}$, and suppose $\z_{i, 0}=1$ for all $i$. Thus, we write $\phi=\sum_{i=1}^s \lambda_i \z_i^{\otimes d}$.  Write $\Lambda=\operatorname{diag}\{\lambda_1, \dots, \lambda_s\}$.

Assume that $D\geq \rho(\Z)$.  Let $\hat{\phi}_0, \dots, \hat{\phi}_n$ be the $n+1$ slices of $\phi$ matricized as elements of $S^D\bbC^{n+1}\otimes \bb{S^{d-1-D}{\bbC^{n+1}}}^*$.  Explicitly we have
\[\hat{\phi}_j=\sum_{i=1}^s \lambda_i \z_{i, j}\operatorname{vec}\bb{\z_i^{\otimes D}}\operatorname{vec}\bb{\z_i^{\otimes d-1-D}}^\top=\Z_D^\top \Lambda_j \Z_{d-1-D},\]
where $\Lambda_j:=\operatorname{diag}\{\lambda_1 \z_{1,j}, \dots, \lambda_s \z_{s, j}\}$.  Then letting $\dagger$ denote the pseudoinverse, we have
\begin{equation}\label{eq: simult diag mats}
\begin{aligned}
    \hat{\phi}_j \hat{\phi}_0^{\dagger}
    &=\Z_D^\top \Lambda_j \Z_{d-1-D} \bb{\Z_D^\top \Lambda_0 \Z_{d-1-D}}^\dagger \\
    &=\Z_D^\top \Lambda_j \Z_{d-1-D} \Z_{d-1-D}^\dagger \Lambda_0^{-1} \bb{\Z_D^\top}^\dagger\\
    &=\Z_D^\top \operatorname{diag}\{\z_{1, j}, \dots, \z_{s, j}\} \bb{\Z_D^\top}^\dagger,
\end{aligned}
\end{equation}
where the second line follows because $\Z_D$ is full row-rank and as $d-1-D\geq D$, $\Z_{d-1-D}$ is also full row-rank.  Now this last line is an eigendecomposition, where $\z_i^{\otimes D}, \z_{i, j}$ form an eigenpair.  In the case where the eigenvalues are simple the decomposition is readily obtained as eigenvectors. Otherwise, we can pick $a_1, \dots, a_n$ randomly from some continuous distribution (e.g., Gaussian) and form the linear combination
\[\sum_{j=1}^n a_j \hat{\phi}_j \hat{\phi}_0^\dagger,\]
for which $(\z_i^{\otimes D}, \sum_{j=1}^n a_j \z_{i, j})$ form an eigenpair (as the $\hat{\phi}_j\hat{\phi}_0^\dagger$ commute).  With probability one the eigenvalues are simple.

The implications of this procedure succeeding are threefold.  First, $s$ must equal the $\operatorname{rank}(\phi)$ from $D\geq \rho(\Z)$.  Second, the decomposition $\Z$ must be unique. Though we described the procedure in terms of $\Z$ it can be seen that each step left no room for another decomposition (as we assumed that $\Z$ is a decomposition in the first place).  This is in fact an algorithmic proof of Theorem 2.17 of \cite{MOURRAIN2020347}, which states that if $\Z$ gives a decomposition for $\phi$ and $d\geq 2\rho(\Z)+1$ then $\Z$ is a unique decomposition. Indeed, $D\geq \rho(\Z)$ and $d\geq 2D+1$.  Third, as an algorithm it is efficient: each step can be performed via polynomially many linear algebra operations.  

\subsubsection{Meta-algorithm for tensor decomposition}
The above can be simplified slightly by choosing a set $B, |B|=s$ indexing columns of $\Z_D$ such that $\Z_B$ is an $s\times s$ invertible submatrix of $\Z_D$.  Taking now $\mathbf{H}_j=\Z_B^\top \Lambda_j \Z_B$ the above procedure returns
\[\mathbf{H}_j\mathbf{H}_0^{-1}=\Z_B^\top \operatorname{diag}\{\z_{1, j}, \dots, \z_{s, h}\}\Z_B^{-\top}.\]
As long as $B$ contains $\{1, x_1, \dots, x_n\}$ we can obtain the tensor decomposition from the eigendecomposition.  This is not a strong assumption, nor is the assumption that we can find such a $B$ -- we will elaborate on both in Section \ref{sec:mmextension}.

A limitation to the simultaneous diagonalization algorithm is that $\z_1^{\otimes D}, \dots, \z_s^{\otimes D}$ have to be linearly independent.  In particular, this forces $s\leq \binom{n+D}{n}=O\bb{n^{\floor{(d-1)}/2}}$ -- and even if $s$ is small enough, it is not always the case that this linear independence condition is satisfied.  

Conceptually this can be addressed by taking a slightly inverted perspective -- the issue is not that $s$ is too large but $d$ is too small.  Indeed, let $\Z=\{\z_1, \dots, \z_s\}, \z_{i, 0}=1$ be a set of points (with corresponding scalars $\lambda_i$).  If we are given $\phi=\sum_{i=1}^s \lambda_i \z_i^{\otimes d}$ with $2\rho(\Z)+1>d$ then we would not be able to apply the simultaneous diagonalization algorithm; however, if we somehow had access to the tensor $\phi^{\uparrow}=\sum_{i=1}^s \lambda_i \z_i^{\otimes 2\rho(\Z)+1}$ then we can.  Of course, we do not have access to $\Z$ nor $\phi^{\uparrow}$ -- but we would like to solve for, in some principled way, this \textit{extension} $\phi^{\uparrow}$, motivating a meta-algorithm for tensor decomposition.

\begin{algorithm}[h!]
   \hspace*{\algorithmicindent} \textbf{Input: }$\phi\in S^d \bbC^{n+1}$; desired decomposition size $s$ \\
    \hspace*{\algorithmicindent} \textbf{Output: }A size-$s$ decomposition of $\phi$
\begin{algorithmic}[1]
\STATE Solve for an \textit{extension} $\phi^\uparrow$ such that $\phi^\uparrow$ has a unique size-$s$ decomposition that also decomposes $\phi$ \label{alg1:extension}
\STATE Apply a \textit{core} algorithm to recover the unique decomposition from $\phi^\uparrow$
\end{algorithmic}
\caption{Meta-algorithm for tensor decomposition}
\label{alg:meta}
\end{algorithm}

The central idea of \cite{brachat2010symmetric} is to introduce an algorithm for the first step of the meta-algorithm and apply simultaneous diagonalization as the second step.  Given the ``right'' extension $\phi^{\uparrow}$, simultaneous diagonalization, and therefore decomposition, is efficient.  However, we know that tensors are not always identifiable nor efficiently decomposable; thus, both properties must manifest in the extension step. 

We also note that in Step \ref{alg1:extension} it often suffices to obtain a \textit{partial} extension; we do not necessarily need a full degree $2\rho(\Z)+1$ extension but rather enough of a sufficiently high degree extension to apply the core algorithm (which will often be simultaneous diagonalization).  

\begin{remark}
    This meta-algorithm is implicit in the work on average case tensor decomposition \cite{ma2016polynomial, hopkins2019robust}.  The main idea is to argue that some $\phi^{\uparrow}$ can be well-approximated with high probability over the randomness of the underlying components efficiently, and then show that simultaneous diagonalization on this approximated $\phi^{\uparrow}$ outputs a decomposition close enough to the random decomposition.  Though we are instead focused on efficient \textit{exact} decomposition and do not assume a random model, our exposition unifies some ideas from different fields.
\end{remark}

\section{Moment matrix extension}\label{sec:mmextension}
We discuss the algorithm first described by \cite{brachat2010symmetric}. The crucial contribution in their work was to introduce the idea of extending a tensor $\phi$ to a higher order tensor; this is Step~\ref{alg1:extension} of Algorithm \ref{alg:meta}.  We make a novel observation connecting possible decompositions output by this algorithm and the Hilbert functions of the corresponding sets of points, allowing for a refinement of the algorithm.
\subsection{Key constructions}
\begin{definition}\label{def:hankel}
    Given a tensor $\phi\in S^d\bbC^{n+1}$ the corresponding $\binom{n+d}{n}\times \binom{n+d+1}{n}$ \textit{Hankel matrix} $\mathcal{H}_{\phi}$ is:
\[
    \bb{\mathcal{H}_{\phi}}_{\alpha, \beta}=\begin{cases}
        \phi_{\alpha+\beta}, & |\alpha+\beta|\leq d,\\
        y_{\alpha+\beta}, & \textrm{otherwise}.
    \end{cases}
\]
Here $|\alpha|\leq d$, $|\beta|\leq d+1$, and the $y$ are formal \textit{moment variables}.  When $\phi$ is clear from context we simply write $\mathcal{H}$.  For index sets $B, B'$, $\mathcal{H}_{B, B'}$ denotes the corresponding submatrix.
\end{definition}
Note the similarities of this definition with Definition \ref{def:catalecticant}.  We can view the Hankel matrix as the catalecticant matrix of an order $2d+1$ ``extension'' of $\phi$; however, we only have access to entries indexed by multiindices of size at most $d$ so the moment variables of larger size are left unknown.

\begin{definition}
    A set of monomials $B$ is \textit{connected to $1$} if for all $x^\alpha\in B$ either $x^\alpha=1$ or there exists $x^{\alpha'}\in B, i\in [n]$ such that $x^\alpha=x^{\alpha'}x_i$.  
\end{definition}

Connected to $1$ will be required for Theorem \ref{thm:main} below which justifies the extension algorithm.

\begin{definition}\label{def:determinal_relations}
    Let $B$ be a set of monomials.  Define the system of \textit{determinantal relations} with respect to $B$ to be the polynomial system in the moment variables
    \[\operatorname{det}\bb{\mathcal{H}_{B\cup \{x^\alpha x_i\}, B\cup\{x^\beta x_j\}}}-\operatorname{det}\bb{\mathcal{H}_{B\cup \{x^\alpha x_j\}, B\cup\{x^\beta x_i\}}}=0\]
    for $x^\alpha, x^\beta\in B, i\neq j$.  
\end{definition}

Let $Y$ be the set of moment variables involved in these relations; then $Y$ can be explicitly given by 
\begin{equation}\label{eq:moment_variables}
    Y=\{y_{x^{\alpha+\alpha'}x_i}\mid |\alpha+\alpha'|\geq d, x^\alpha, x^{\alpha'}\in B, 1\leq i\leq n\}.
\end{equation}
We will also denote elements of $Y$ by just the indices $x^{\alpha+\alpha'}x_i$.  Note that $Y$ indeed contains all the variables appearing in the determinantal relations as in each equation, the variable $x^{\alpha+\beta}x_ix_j$ cancels and the rest of the variables are of the form $x^{\alpha+\alpha'}x_i, x^{\alpha+\alpha'}x_j$ for $\alpha'\in B$ such that $|\alpha+\alpha'|\geq d$.

\begin{definition}
    For a set of monomials $B$ connected to $1$ define the $n$ \textit{shifted matrices} $\mathcal{H}_{B, B_i}$, where
    \[B_i=\{x^\alpha x_i\mid x^\alpha\in B\}.\]
    If $\mathcal{H}_{B, B}$ is invertible define the $i$th \textit{multiplication matrix} as
    \[\mathbf{M}^B_i(Y):=\mathcal{H}_{B, B_i}\mathcal{H}_{B, B}^{-1}.\]
\end{definition}

The entries of each multiplication matrix are polynomials in the moment variables $Y$. More commonly used in the literature are the \textit{commuting relations} \cite{brachat2010symmetric, bernardi2020waring}.  
 
\begin{definition}\label{def:commuting_relations}
    Define the system of \textit{commuting relations} to be the polynomial system in the moment variables
    \[\mathbf{M}_i^B(Y)\mathbf{M}_j^B(Y)-\mathbf{M}_j^B(Y)\mathbf{M}_i^B(Y)=0, \quad \forall i\neq j.\]
\end{definition}
Sometimes $\mathcal{H}_{B, B}$ is not fully specified, i.e., there are moment variables appearing in some entries, and thus invertibility of $\mathcal{H}_{B, B}$ is not well-defined.
In particular, this occurs when $\operatorname{deg}(B)>d/2$.  Using the identity with the adjugate:
\[\mathcal{H}_{B, B}^{-1}=\operatorname{det}\bb{\mathcal{H}_{B, B}}^{-1}\operatorname{adj}\bb{\mathcal{H}_{B, B}},\]
we can clear denominators and, as $\operatorname{adj}\bb{\mathcal{H}_{B, B}}$ is entry-wise polynomial in $Y$, the commuting relations are polynomials in $Y$. 

The following lemma shows that the commuting and determinantal relations are equivalent; this is essentially Theorem 6.3 of \cite{brachat2010symmetric} but we derive the result directly.
\begin{lemma}\label{lem:determinant_equals_commuting}
    The determinantal and commuting relations are equivalent (up to a global factor).
\end{lemma}
\begin{proof}
    First, the commuting relations are explicitly
    \[\mathcal{H}_{B, B_i}\mathcal{H}_{B, B}^{-1}\mathcal{H}_{B, B_j}\mathcal{H}_{B, B}^{-1}-\mathcal{H}_{B, B_j}\mathcal{H}_{B, B}^{-1}\mathcal{H}_{B, B_i}\mathcal{H}_{B, B}^{-1}=0.\]
    Up to multiplication by $\operatorname{det}\bb{\mathcal{H}_{B, B}}$ this is equivalent to 
    \[\mathcal{H}_{B, B_i}\operatorname{adj}\bb{\mathcal{H}_{B, B}}\mathcal{H}_{B, B_j}-\mathcal{H}_{B, B_j}\operatorname{adj}\bb{\mathcal{H}_{B, B}}\mathcal{H}_{B, B_i}=0.\]
    We can index the matrix on the left-hand side by $(\alpha, \beta)$ for $\alpha, \beta\in B$.  The corresponding entry is 
    \begin{equation}\label{eq:entry_commuting}
        \mathcal{H}_{B, x^\alpha x_i}^\top \operatorname{adj}\bb{\mathcal{H}_{B, B}}\mathcal{H}_{B, x^\beta x_j}-\mathcal{H}_{B, x^\alpha x_j}^\top\operatorname{adj}\bb{\mathcal{H}_{B, B}}\mathcal{H}_{x^\beta x_i}.
    \end{equation}
    It is straightforward to see that
    \[\operatorname{det}\bb{\mathcal{H}_{B\cup\{x^\alpha x_i\}, B\cup\{x^\beta x_j\}}}=\mathcal{H}_{x^{\alpha+\beta}x_ix_j}\operatorname{det}\bb{\mathcal{H}_{B, B}}-\mathcal{H}_{B, x^\alpha x_i}^\top \operatorname{adj}\bb{\mathcal{H}_{B, B}}\mathcal{H}_{B, x^\beta x_j}\]
    via the Laplace expansion of the determinant.  Then setting equation \ref{eq:entry_commuting} equal to zero is equivalent to 
    \[\operatorname{det}\bb{\mathcal{H}_{B\cup \{x^\alpha x_i\}, B\cup\{x^\beta x_j\}}}-\operatorname{det}\bb{\mathcal{H}_{B\cup \{x^\alpha x_j\}, B\cup\{x^\beta x_i\}}}=0.\qedhere \]
\end{proof}

We observe that $Y$ is also the set of moment variables appearing in the matrices $\mathcal{H}_{B, B}$ and $\mathcal{H}_{B, B_i}, i=1, \dots, n$.  In other words, once a solution to the determinantal relations is found, substituting this solution into the matrices $\mathcal{H}_{B, B}, \mathcal{H}_{B, B_i}, i=1, \dots, n$ fully specifies these matrices.

\subsection{Algorithm for size-\texorpdfstring{$s$}{s} decomposition}
Motivated by the discussion in Section \ref{sec:simultaneous_diagonalization} we seek to ``extend'' $\phi$ to a higher-order tensor.  We have the pieces: the Hankel matrix $\mathcal{H}$ represents symbolically this unknown higher-order tensor, and the determinantal/commuting relations represent relations that the entries of this higher-order tensor must satisfy.  Indeed, reconsidering the base algorithm the multiplication matrices $\mathbf{M}_j^B(Y)$ are generalizations of the matrices $\hat{\phi}_j\hat{\phi}_0^\dagger$ in \eqref{eq: simult diag mats}, which must commute.

In this section we do not necessarily look for a minimal rank decomposition and we let $s$ denote the size of a target decomposition; of course, when $s$ is less than the rank of $\phi$ any algorithm for decomposition will fail.  The observation that the proposed method can produce decompositions of size larger than the minimal rank does not seem to have been
emphasized previously.

First we make the following weak assumption:
\begin{assumption}\label{ass:dehomogenization}
    The size-$s$ decompositions $\Z=\{\z_1, \dots, \z_s\}$ of $\phi$ satisfy $\z_{i, 0}\neq 0$ on a dense set in the space of decompositions.
\end{assumption}
This assumption is standard, though sometimes implicit; see Equation 2.9 of \cite{nie2017generating} and the first few paragraphs of Section 7 of \cite{brachat2010symmetric}; we note that this assumptions is satisfied with probability one after a generic change of basis.  However, sometimes it may be useful to apply our algorithms to certain ``canonical forms''; in these cases we should justify Assumption \ref{ass:dehomogenization}.

We now give the theoretical pillar of the moment matrix extension algorithm in \cite{brachat2010symmetric}.
\begin{theorem}[\cite{brachat2010symmetric}, Theorem 6.2]\label{thm:main}
    A tensor $\phi$ satisfying Assumption \ref{ass:dehomogenization} admits a size-$s$ decomposition $\phi=\sum_{i=1}^s \lambda_i \z_i^{\otimes d}$, $\z_{i, 0}=1$, if and only if there exists a monomial set $B$ connected to $1$, $|B|=s$, $\operatorname{det}\bb{\mathcal{H}_{B, B}}\neq 0$, and a solution $Y^*$ satisfying the determinantal relations \eqref{def:determinal_relations} (equivalently, the commuting relations \eqref{def:commuting_relations}) and each $\mathbf{M}_i^B(Y^*)$ has $s$ distinct eigenvectors; the common eigenvectors are $\z_i^B$ and the corresponding eigenvalues of $\mathbf{M}_j^B(Y^*)$ are $\z_{i, j}$.  
\end{theorem}
\begin{proof}[Proof sketch]
    We sketch an alternative proof compared to the proof in \cite{brachat2010symmetric}.  The only if direction is obvious.  Consider the if direction; given a solution $Y^*$ such that the multiplication matrices $\mathbf{M}_{i}^B(Y^*), i=1, \dots, n$ are fully specified, commute, and have $s$ distinct eigenvectors, the idea is to define $\phi^\uparrow$ a power series (so not necessarily degree bounded) such that $\phi_\gamma^\uparrow = \mathbf{e}_0^\top \bb{\mathbf{M}^B(Y^*)}^\gamma \mathcal{H}_{B, 1}$ for $\gamma\in \bbZ^n_{\geq 0}$; here $\bb{\mathbf{M}^B(Y^*)}^\gamma=\prod_{i=1}^n \bb{\mathbf{M}_i^B(Y^*)}^{\gamma_i}$, well-defined as the multiplication matrices commute.  By Theorem 7.34 of \cite{ElkadiMourrain2007}, eigenpairs of $\mathbf{M}_i^B(Y^*)$ can be found to be of the form $\z_{j, i}, \z_j^B, j=1, \dots, s$ and the $\z_1, \dots, \z_s$ decompose $\phi^\uparrow$; we then can show that $\phi_\gamma^\uparrow=\phi_\gamma$ for $\gamma\in \bbZ_{\geq 0}^n, |\gamma|\leq d$.
\end{proof}

Theorem \ref{thm:main} reduces tensor decomposition to solving a polynomial system.  Each solution to the determinantal relations, with nondefective multiplication matrices (i.e., $s$ linearly independent eigenvectors), corresponds to a unique decomposition via an eigendecomposition. 
Thus, the space of decompositions of size $s$ correspond to the solutions of the determinantal relations (up to nondefectivity of the multiplication matrices). The work in \cite{nie2017generating} introduces a similar reduction of tensor decomposition to polynomial system solving using generating polynomials. This method directly treats unknown entries of the multiplication matrices as variables and employs the commuting relations as defining equations. In contrast we treat the entries of unknown extensions as variables and employ the determinantal relations as defining equations.

We present the algorithm corresponding to the discussion so far as Algorithm \ref{alg:size_s_decomp}.  We have discussed all steps except the collection of monomial sets $\mathcal{B}_s$ in Step~\ref{for b in Bs}. A natural choice implied by Theorem \ref{thm:main} is to take $\mathcal{B}_s$ to be the collection of all monomial sets of size $s$ connected to $1$. However, the size of $\mathcal{B}_s$ in this case can be quite large and it is desirable to search over a smaller set.  In the next section we discuss how the size of $\mathcal{B}_s$ can be reduced, substantially decreasing the combinatorial cost of the extension algorithm.

\begin{algorithm}[!ht]
   \hspace*{\algorithmicindent} \textbf{Input: }$\phi\in S^d \bbC^{n+1}$; desired decomposition size $s$; collection of monomial sets $\mathcal{B}_s$ \\
    \hspace*{\algorithmicindent} \textbf{Output: }A size-$s$ decomposition of $\phi$ or FAIL, with no decompositions existing with respect to any $B\in \mathcal{B}_s$ 

\begin{algorithmic}[1]
\STATE Construct the Hankel matrix of $\phi$
\FOR{$B\in \mathcal{B}_s$} \label{for b in Bs}
\STATE Find parameters $Y^*$ such that
\begin{itemize}
    \item $\operatorname{det}\bb{\mathcal{H}_{B, B}}\neq 0$,
    \item the determinantal relations \ref{def:determinal_relations} are satisfied,
    \item there are $s$ linearly independent eigenvectors common to the $\mathbf{M}_i^B(Y^*)$.
\end{itemize}
\STATE  If parameters $Y^*$ exist, obtain a decomposition $\z_1, \dots, \z_s$ via an eigendecomposition and solve a linear system for $\lambda_1, \dots, \lambda_s$; otherwise, continue
\ENDFOR
\STATE Output FAIL; there are no solutions to the commuting relations with respect to any $B\in \mathcal{B}_s$

\end{algorithmic}
\caption{Size-$s$ tensor decomposition}
\label{alg:size_s_decomp}
\end{algorithm}

\begin{corollary}\label{cor:alg1correct}
    Algorithm \ref{alg:size_s_decomp} with $\mathcal{B}_s$ taken to be the collection of size $s$ monomial sets connected to $1$ outputs all decompositions of $\phi$ of size $s$.
\end{corollary}

A procedure for minimal-rank tensor decomposition is a wrapper around this algorithm. Simply initialize $s=\max_k\operatorname{rank}\bb{\operatorname{Cat}_k(\phi)}$ and if the algorithm fails increment $s$ by one.  Since 
\[\max_k\operatorname{rank}\bb{\operatorname{Cat}_k(\phi)}\leq \operatorname{rank}(\phi), \] 
as long as the collection of monomial sets $\mathcal{B}_s$ is selected correctly (e.g., all sets of size $s$ connected to $1$) a minimal-rank decomposition will be found.

\begin{remark}
    In this work we focus on decompositions over $\mathbb{C}$.  Note that a real decomposition corresponds to the additional condition that all multiplication matrices $\mathbf{M}_i^B(Y^*)$ have real eigenvalues, which is not a closed condition.  This merits further consideration.
\end{remark}

\subsubsection{Choice of monomial set \texorpdfstring{$B$}{B}}\label{sec:B_choice}
We now discuss how to more efficiently find a basis $B$. To streamline our discussion, we first give a few definitions.

\begin{definition}
    Let $\Z=\{\z_1, \dots, \z_s\}$ be a set of $s$ points.  We say that a size $s$ monomial set $B$ connected to $1$ is \textit{valid with respect to $\Z$} if $\Z_B$ is an invertible $s\times s$ matrix.  Correspondingly, let $\phi\in S^d\bbC^{n+1}$.  We say a collection $\mathcal{B}_s$ of size $s$, connected to $1$ monomial sets is\textit{ valid with respect to $\phi$} if for every size $s$ decomposition $\Z$ of $\phi$ there exists $B\in \mathcal{B}_s$ such that $B$ is valid with respect to $\Z$.
\end{definition}

These definitions are made to describe the fact that given a tensor $\phi$, if $\mathcal{B}_s$ is valid, then all decompositions of size $s$ of $\phi$ are output by Algorithm \ref{alg:size_s_decomp}. Thus, as the algorithm loops over $\mathcal{B}_s$ we would prefer a valid $\mathcal{B}_s$ of both minimal size.  Additionally, in Section \ref{sec:lowrankeven} we will see that minimizing $\operatorname{deg}(B)$ will be helpful.

We first consider the maximum degree needed of monomials in $B$ on the level of points. 
\begin{proposition}\label{prop:connectedto1}
    Let $\Z=\{\z_1, \dots, \z_s\}$ be a set of points with $\z_{i, 0}=1$ for all $i$.  Then there exists a set of monomials $B$ connected to $1$ with $|B_k|=\operatorname{rank}(\Z_k)-\operatorname{rank}(\Z_{k-1})$ such that $B$ is valid with respect to $\Z$.
\end{proposition}
\begin{proof}
    Let $D:=\rho(\Z)$; then we know that $\operatorname{rank}(\Z_{D})=s$ and  $\operatorname{rank}(\Z_{j})<s$ for $j<D$.  

    Recall that $h_{\Z}$ is the Hilbert function of the ideal of polynomials vanishing on $\Z$, and that $h_{\Z}(k)=\operatorname{rank}(\Z_k)$.  Since $\Z_{k-1}$ embeds as a submatrix of $\Z_k$ there are exactly $h_{\Z}(k)-h_{\Z}(k-1)$ extra columns of $\Z_k$ not in the column space of $\Z_{k-1}$.  These extra columns are indexed by $\alpha$ such that $|\alpha|=k$.

    These added indices can be chosen so that the whole monomial set is connected to $1$.  Indeed, we can proceed by induction.  Write $B_{\leq k}=B_0\cup \dots \cup B_{k}$.  The case $k=1$ is trivial.  Now suppose for all $k'<k$ that $B_{\leq k'}$ is connected to $1$ and $B_{k'}=\operatorname{rank}(\Z_{k'})-\operatorname{rank}(\Z_{k'-1})$.  Let $\gamma$ be some index, $|\gamma|=k$, $x^\gamma$ not connected to $1$, so that $x^{\gamma_{\hat{i}}}:=x^{\gamma}/x_i\not\in B_{k-1}$ for any $\gamma_i\neq 0$.  Now fix $i$ such that $\gamma_i\neq 0$, with some abuse of notation; $|\gamma_{\hat{i}}|=k-1$, so by induction $\Z_{\gamma_{\hat{i}}}$ must be some linear combination of the columns in $\Z_{B_{\leq k-1}}$, i.e., $\Z_{\gamma_{\hat{i}}}=\Z_{B_{\leq k-1}}\mathbf{c}$ for some nonzero vector $\mathbf{c}$, which can also be indexed by indices $|\alpha|\leq k-1$.  If $\mathbf{c}_{\alpha}=0$ for all $|\alpha|=k-1$ then $\Z_{\gamma_{\hat{i}}}$ can be written as a linear combination of columns in $\Z_{B_{\leq k-2}}$, so that $\Z_{\gamma}$ is already a linear combination of the columns in $\Z_{B_{\leq k-1}}$.  So assume that $\mathbf{c}_\alpha\neq 0$ for some $|\alpha|=k-1$.  But then $\Z_\gamma=\Z_{x_i B_{\leq k-1}}\mathbf{c}$, a linear combination of columns corresponding to indices connected to $1$.  So, we can reduce the construction of $B_k$ to sets connected to $1$.
\end{proof}

Proposition \ref{prop:connectedto1} implies a method to form a  set $\mathcal{B}_s$ valid  with respect to $\phi$ using information about the Hilbert functions of possible decompositions $\Z$. We say that a sequence $\ell\in \bbZ_{\geq 0}^{\infty}$ is \textit{admissible} with respect to $\phi$ if it is possible for $\phi$ to have a decomposition $\Z$ with $\ell$ as the range of its Hilbert function, i.e., $\ell_k=h_{\Z}(k)$ and $\sup \ell=s$. For example, we cannot include any sequence such that $\ell_k<h_\phi(k)$ for any $k$. 
\begin{enumerate}
    \item Let $\mathcal{S}$ be the set of sequences $\ell\in \bbZ_{\geq 0}^{\infty}, \ell_k:=h_{\Z}(k)$ admissible with respect to $\phi$.
    \item Set
    \[\mathcal{B}_s:=\{\text{monomial sets }B \mid |B_k|=\ell_k-\ell_{k-1},\ell \in\mathcal{S} \}.\]
\end{enumerate}

In other words, for a tensor $\phi$ we first establish all the possible Hilbert functions $h_{\Z}(k)$ of size $s$ decompositions $\Z$ of $\phi$; we then let $\mathcal{B}_s$ be the sequences with range the Hilbert functions.  In some sense, this $\mathcal{B}_s$ is the smallest collection valid with respect $\phi$; the more knowledge one has on the possible $h_{\Z}(k)$ the smaller we can make $\mathcal{B}_s$.  Though this collection of monomial bases can still be quite large, we show that this idea already encompasses some existing improvements in the literature.

\begin{definition}[ \cite{Landsberg2011}]
    Let $\phi\in S^d\bbC^{n+1}$. We say $\phi$ is \textit{concise} if there does not exist a proper linear subspace $V\subset \bbC^{n+1}$ such that $\phi\in S^d V$. The \textit{essential variables} of $\phi$ is a basis of $V\subseteq \bbC^{n+1}$ for which $\phi$ is concise as an element of $S^d V$.  The number of essential variables of $\phi$, $N_{\text{ess}}(\phi)$, is the dimension of such $V$.
\end{definition}

Note that the number of essential variables is well-defined as the number of essential variables is invariant under the $\operatorname{GL}(n+1)$ action.  Moreover, making $\phi$ concise is an easy task via the following proposition.
\begin{proposition}[\cite{carlini2006reducing}, Proposition 1]\label{prop: concise}
    Let $\phi\in S^d\bbC^{n+1}$.  Then
    \[N_{\text{ess}}(\phi)=\operatorname{rank}\bb{\operatorname{Cat}_1(\phi)}.\]
    Moreover, a basis for the row space of $\operatorname{Cat}_1(\phi)$ serves as essential variables for $\phi$.
\end{proposition}

From Proposition \ref{prop: concise} we can now without loss of generality assume our tensor is concise. We formalize this as an assumption.

\begin{assumption}\label{ass:concise}
    The tensor $\phi\in S^d\bbC^{n+1}$ is concise.
\end{assumption}

Under Assumption \ref{ass:concise} we have the following two guarantees on the functions $h_{\Z}(k)$ for any size-$s$ decomposition $\Z$ of $\phi\in S^d\bbC^{n+1}$ and thus conditions for a valid $\mathcal{B}_s$.
\begin{enumerate}
    \item $h_{\Z}(1)=n+1$, so that $|B_1|=n$ for all $B\in \mathcal{B}_s$,
    \item $h_{\Z}(k)=s$ for $k\geq d$, so that $\operatorname{deg}(B)\leq d$ for all $B\in \mathcal{B}_s$.
\end{enumerate}
The first guarantee is due to submultiplicativity of the matrix rank.  The second is due to the fact that for $\phi$ to admit a size-$s$ decomposition, $\phi=\sum_{i=1}^s \lambda_i \z_i^{\otimes d}$ and the elements $\z_1^{\otimes d}, \dots, \z_s^{\otimes d}$ must be linearly independent. Then $\operatorname{rank}\bb{\Z_d}=s$. We have thus recovered Theorem 3.9 of \cite{bernardi2020waring}. We now show we can do better in the following theorem; it is similar to Lemma 2.15 of \cite{MOURRAIN2020347}.

\begin{theorem}\label{thm:Bk_refine}
    Instate Assumption \ref{ass:dehomogenization}.  Suppose we include $B\in \mathcal{B}_s$ such that $\operatorname{deg}(B)=D$.   Then we can also impose $|B_{\leq k}|=\operatorname{rank}\bb{\operatorname{Cat}(\phi, k)}$ for $k\leq d-D$.
\end{theorem}
\begin{proof}
    Suppose we can choose $B$ such that $\operatorname{deg}(B)=D$.  Let $\Z=\{\z_1, \dots, \z_s\}$ be a decomposition output by  Algorithm \ref{alg:size_s_decomp} corresponding to $B$.  Then $\operatorname{deg}(B)=D$ implies that $\operatorname{rank}(\Z_D)=s$.  

    Writing $\Lambda=\operatorname{diag}\{\lambda_1, \dots, \lambda_s\}$, since $\operatorname{Cat}_k(\phi)=\Z_{d-k}^\top \Lambda \Z_k$, if $k\leq d-D$ then $d-k\geq D$, so that $\Z_{d-k}^\top$ is full column rank, i.e., it is an injective linear map.  Therefore, as $\Lambda$ is also rank $s$, we must have $\operatorname{rank}\bb{\operatorname{Cat}_k(\phi)}=\operatorname{rank}(\Z_k)$.  Then $|B_{\leq k}|=\operatorname{rank}(\Z_k)=\operatorname{rank}\bb{\operatorname{Cat}_k(\phi)}$.
\end{proof}
This is a ``data-dependent'' way of forming a valid $\mathcal{B}_s$ and serves as a refinement to the vanilla algorithm of \cite{brachat2010symmetric}.  For example, we can determine the following:
\begin{enumerate}
    \item Suppose we desire a size-$s$ decompositions, but $s>\max_k\operatorname{rank}\bb{\operatorname{Cat}_k(\phi)}$.  Then by submultiplicativity of the matrix rank, any valid $B\in \mathcal{B}_s$ must be such that $\operatorname{deg}(B)>d/2$.
    \item On the other hand, suppose we include $B\in \mathcal{B}_s$ with $\operatorname{deg}(B)=D\leq d$.  Then Theorem \ref{thm:Bk_refine} leaves free the choice of size of $B_k$ for $k>d-D$.  In particular, the sizes of possible decompositions are not fully specified by $\operatorname{deg}(B)$.
\end{enumerate}

\begin{remark}
    Our discussion has an interesting connection to Question 2.14 of \cite{MOURRAIN2020347}, which asks if all minimal rank decompositions of a tensor $\phi$ have the same Hilbert functions.  This was shown to be false in \cite{angelini2023waring}. The counterexamples involve sets of points with regularity larger than $d/2$.  In light of Theorem \ref{thm:Bk_refine} this makes sense.  From the perspective of Algorithm \ref{alg:size_s_decomp}, when we choose $B$ with $\operatorname{deg}(B)>d/2$ the matrix $\mathcal{H}_{B, B}$ is not fully specified by the tensor, and so the degree sequences of possible $B$ specify the differing regularities of these decompositions.
\end{remark}

\subsection{Efficient decomposition}
 Algorithm \ref{alg:size_s_decomp} is not a priori efficient: $\mathcal{B}_s$ can be of size exponential in the number of tensor entries (if $s$ is large) and the relations can be polynomials of large degree in the moment variables $Y$.  

However, if we can ensure that for a tensor $\phi$ and a desired decomposition size $s$, both a valid $\mathcal{B}_s$ can be made to be a collection of size polynomial in the number of tensor entries (or even better, of constant size), and  for each $B\in \mathcal{B}_s$ the solution to the relations is fully specified by a (sequence of) \textit{linear} system(s) of size polynomial in the number of tensor entries, then Algorithm \ref{alg:size_s_decomp} is efficient.

We now describe classes of tensors for which this is possible. Each of these classes satisfy certain conditions that make this Algorithm \ref{alg:size_s_decomp} efficient.  
In Section \ref{sec:lowrankeven} we use our discussion in Section \ref{sec:B_choice} to detail a method for efficient decomposition for even order tensors with decompositions of low regularity -- giving efficient decomposition for generic tensors of rank larger than allowed by simultaneous diagonalization -- and expand on the theory for $d=4$.  In Section \ref{sec:monomials} we demonstrate that if the given tensor to be decomposed is a monomial in a canonical format, the algorithm parametrizes the space of decompositions (monomials are nonidentifiable) to output \textit{arbitrary decompositions} efficiently.

As a lead-in, we outline how efficient decomposition is possible for two simple classes of tensors: tensors decomposable via simultaneous diagonalization (to be expected from the discussion in Section \ref{sec:simultaneous_diagonalization}) and binary tensors.

\subsubsection{Tensors decomposable via simultaneous diagonalization}
As described in Section \ref{sec:simultaneous_diagonalization}, simultaneous diagonalization is a specific instance of the moment matrix extension algorithm where no polynomial system needs to be solved. 

\begin{condition}\label{cond:really_low_rank}
    Let $D=\floor{\frac{d-1}{2}}$.  $\phi$ satisfies 
    \[\operatorname{rank}\bb{\operatorname{Cat}_D(\phi)}=\operatorname{rank}(\phi).\]
\end{condition}

This is a condition on the regularity of target decompositions: by submultiplicativity of the matrix rank, for any minimal rank decomposition $\Z$ of $\phi$, $D\geq \rho(\Z)$.  In particular, by Theorem \ref{thm:Bk_refine} there exists a monomial set $B, |B|=\operatorname{rank}(\phi)$ connected to $1$ with $\operatorname{deg}(B)\leq D$, and the corresponding Hankel matrix $\mathcal{H}_{B, B}$ is invertible and fully specified (i.e., no moment variables are present) as $d\geq 2D+1$.  

\begin{claim}\label{claim:really_low_rank_B}
    The monomial set $B$ described above can be found efficiently.
\end{claim}

Therefore, $\mathcal{B}_s=\{B\}$ is valid with respect to $\phi$.  This claim is a special case of a more general proposition we prove in Section \ref{sec:lowrankeven}.  As $d\geq 2D+1$ the set of moment variables $Y$ is empty, and under Condition \ref{cond:really_low_rank} we assume existence of a decomposition of size $\operatorname{rank}(\phi)$; thus the determinantal relations must be trivially satisfied and an eigendecomposition recovers the unique decomposition.  This discussion implies the following result.
\begin{proposition}
    Let $\phi\in S^d\bbC^{n+1}$ satisfy Assumption \ref{ass:dehomogenization} and Condition \ref{cond:really_low_rank}.  Then Algorithm \ref{alg:size_s_decomp} is efficient.  In particular, generic tensors of rank at most $\binom{n+D}{n}$ are efficiently decomposable via moment matrix extension.
\end{proposition}

\subsubsection{Binary tensors}
We now consider the class of binary tensors, i.e., tensors $\phi\in S^d\bbC^{2}$.  Classically, Sylvester's algorithm \cite{Sylvester1886} can be used for decomposing binary tensors.  We show that an application of Algorithm \ref{alg:size_s_decomp} similarly decomposes binary tensors, efficiently.

Recall that the determinantal relations require $1\leq i<j\leq n$.  However, since $n=1$ there are no relations.  Moreover, a monomial set $B$ connected to $1$ is simply a set $B=\{1, x, \dots, x^k\}$ for some integer $k$.  Thus, a desired decomposition size $s$ forces $\mathcal{B}_s=\{B\}$ with $B=\{1, x, \dots, x^{s-1}\}$; it is valid with respect to $\phi$.

This implies an efficient algorithm for size $s$ decomposition of binary tensors satisfying Assumption \ref{ass:dehomogenization}: simply take $B=\{1, x, \dots, x^{s-1}\}$ and form the matrix $\mathcal{H}_{B, B}$ and shifted matrices $\mathcal{H}_{B, B_i}$.  There are no moment variables in any of the matrices when $2s-1\leq d$, there are moment variables in the shifted matrices when $2s-1\geq d+1$, and $\mathcal{H}_{B, B}$ has moment variables when $2s-2\geq d+1$.  When there are moment variables, because there are no determinantal relations, the practitioner can treat the moment variables as free parameters. The only constraint is that the choice of values for the moment variables must satisfy that $\mathcal{H}_{B, B}$ is invertible and the resulting multiplication matrices $\mathbf{M}_i^B$ are nondefective.  

In the case where $s<\operatorname{rank}(\phi)$ there cannot exist a choice of values for the moment variables that make the multiplication matrices nondefective (otherwise we would have a size $s<\operatorname{rank}(\phi)$ decomposition) by definition of the tensor rank.  For larger decompositions, we have the following restriction on the existence of size $s$ decompositions.

\begin{proposition}\label{prop: binary tensors size s}
    Let $\phi\in S^d\bbC^2$, and write $r:=\operatorname{rank}(\phi)$.  Then there are no decompositions of $\phi$ of size $s$ for $r+1\leq s\leq d-r+1$. 
\end{proposition}
\begin{proof}
    If $r+1\leq s\leq d-r+1$ then $s+r\leq d+1$.  Take $|B|=s$, which must be $\{1, x, \dots, x^{s-1}\}$, and consider $\mathcal{H}_{B, B}$.  In particular, take $S=\{1, x, \dots, x^r\}$; the $(r+1)\times s$ matrix $\mathcal{H}_{S, B}$ is the first $r+1$ rows of $\mathcal{H}_{B, B}$.  Since $s+r\leq d+1\implies r+s-1\leq d$, $\mathcal{H}_{S, B}$ is fully specified with no moment variables.  As the rank of $\phi$ is $r$ this means that $\mathcal{H}_{S, B}$ is rank at most $r$.  Thus, even though $\mathcal{H}_{B, B}$ may contain moment variables, no choice of values for the moment variables can make $\mathcal{H}_{B, B}$ invertible.
\end{proof}

This result is likely known, but our proof demonstrates that a simple analysis of the extension algorithm can give theoretical results.  We perform similar analysis in Section \ref{sec:monomials}.  Our discussion implies the following result.
\begin{proposition}
    Let $\phi\in S^d\bbC^{2}$ satisfy Assumption \ref{ass:dehomogenization}.  Then generic choices of values for the moment variables correspond to a decomposition of size $s$, if such a decomposition exists.  In particular, parameterized decompositions of binary tensors can be found efficiently.
\end{proposition}

\begin{remark}
    For the statement in Proposition \ref{prop: binary tensors size s} to be non-vacuous
    we need $r+1\leq d-r+1$, so that $2r\leq d$.  This means that $r$ must be of strictly subgeneric rank; the generic rank of a binary tensor is $\ceil{\frac{d+1}{2}}$ \cite{AlexanderHirschowitz1995, Sylvester1886}.  Moreover, the bounds given have an interesting connection to a result of Comas-Seiguer \cite{Comas2011rank} that states that the set of tensors of \textit{border rank} at most $r$ is given as the union of the sets of tensors of rank at most $r$ and the tensors of rank at least $d-r+2$.
\end{remark}

\begin{example}
    Consider the tensor $\phi=x_0^4x_1(x_0+x_1)\in S^6\bbC^2$.  We do not know Assumption \ref{ass:dehomogenization} holds a priori, but we proceed in the example as if it does.
    \[
    \mathcal{H}=\begin{bmatrix}
        0 & 1 & 1 & 0 & 0 & 0 & 0\\
        1 & 1 & 0 & 0 & 0 & 0 & y_{x^7} \\
        1 & 0 & 0 & 0 & 0 & y_{x^7} & y_{x^8}\\
        0 & 0 & 0 & 0 & y_{x^7} & y_{x^8} & y_{x^9}\\
        0 & 0 & 0 & y_{x^7} & y_{x^8} & y_{x^9} & y_{x^10}\\
        0 & 0 & y_{x^7} & y_{x^8} & y_{x^9} & y_{x^10} & y_{x^{11}}\\
        0 & y_{x^7} & y_{x^8} & y_{x^9} & y_{x^{10}} & y_{x^{11}} & y_{x^{12}}
    \end{bmatrix}.
    \]
    The sequence of ranks of the catalecticants are $(1, 2, 3, 3, 3, 2, 1)$.  If we guess the rank of $\phi$ is $3$, then $B=\{1, x, x^2\}$ and we have
    \[\mathcal{H}_{B, B}=\begin{bmatrix}
        0 & 1 & 1\\
        1 & 1 & 0\\
        1 & 0 & 0
    \end{bmatrix}\quad \text{and} \quad \mathcal{H}_{B, B_1}
    =\begin{bmatrix}
        1 & 1 & 0 \\
        1 & 0 & 0\\
        0 & 0 & 0
    \end{bmatrix}.\]
    While $\mathcal{H}_{B, B}$ is invertible, 
    \[\mathcal{H}_{B, B_1}\mathcal{H}_{B, B}=\begin{bmatrix}
        0 & 1 & 0\\
        0 & 0 & 1\\
        0 & 0 & 0
    \end{bmatrix}\]
    is defective, so we conclude that $\operatorname{rank}(\phi)>3$.  This scenario corresponds to the \textit{cactus rank} of $\phi$; for a more in depth algorithmic perspective see \cite{bernardi2020waring}.  If we next guess that $\operatorname{rank}(\phi) = 4$, then $\mathcal{H}_{B, B}$ is singular, so we conclude $\operatorname{rank}(\phi)>4$.  If we now guess  that $\operatorname{rank}(\phi) = 5$, we have
    \[\mathcal{H}_{B, B}=\begin{bmatrix}
        0 & 1 & 1 & 0 & 0\\
        1 & 1 & 0 & 0 & 0\\
        1 & 0 & 0 & 0 & 0\\
        0 & 0 & 0 & 0 & y_{x^7}\\
        0 & 0 & 0 & y_{x^7} & y_{x^8}
    \end{bmatrix},\quad \mathcal{H}_{B, B_1}
    =\begin{bmatrix}
        1 & 1 & 0 & 0 & 0 \\
        1 & 0 & 0 & 0 & 0\\
        0 & 0 & 0 & 0 & y_{x^7}\\
        0 & 0 & 0 & y_{x^7} & y_{x^8}\\
        0 & 0 & y_{x^7} & y_{x^8} & y_{x^9}
    \end{bmatrix}.\]
    Then
    \[
    \mathbf{M}_1^B=\begin{bmatrix}
        0 & 1 & 0 & 0 & 0 \\
        0 & 0 & 1 & 0 & 0 \\ 
        0 & 0 & 0 & 1 & 0 \\
        0 & 0 & 0 & 0 & 1 \\
        y_{x^7} & -y_{x^7} & y_{x^7} & \frac{y_{x^9}}{y_{x^7}}-\frac{y_{x^8}^2}{y_{x^7}^2} & \frac{y_{x^8}}{y_{x^7}}
    \end{bmatrix}.
    \]
    We know that the eigenvectors of $\mathbf{M}_1^B$ must satisfy
    \[
    \mathbf{M}_1^B\begin{bmatrix}
        1 \\
        t \\ 
        t^2 \\
        t^3 \\
        t^4
    \end{bmatrix}=t\begin{bmatrix}
        1 \\
        t \\ 
        t^2 \\
        t^3 \\
        t^4
    \end{bmatrix}
    \]
    by Theorem 7.34 of \cite{ElkadiMourrain2007}.  In particular, this reduces to the single quintic
    \[y_{x^7}^2t^5-y_{x^7}y_{x^8}t^4+\bb{y_{x^8}^2-y_{x^7}y_{x^9}}t^3-y_{x^7}^3t^2+y_{x^7}^3t+y_{x^7}^3=0.\]
    It can be checked through software that the discriminant of this polynomial is nonzero as a polynomial in $y_{x^7}, y_{x^8}, y_{x^9}$. Thus, for generic values of $y_{x^7}, y_{x^8}, y_{x^9}$, $\mathbf{M}_1^B$ admits five linearly independent eigenvectors -- and the space of decompositions of size $5$ of $\phi$ is $3$-dimensional. Using this matrix extension approach, we have shown that $\operatorname{rank}(\phi)=5$, reproving a result of \cite{TOKCAN2017250}.  Moreover, the decompositions of $\phi$ are efficiently obtained from the eigendecomposition of $\mathbf{M}_1^B$ and parameterized by choices of $y_{x^7}, y_{x^8}, y_{x^9}$.

    If we want a larger decomposition we can take $|B|=6$. The involved variables are $y_{x^7}, y_{x^8}, y_{x^9}, y_{x^{10}}, y_{x^{11}}$. From a similar computer check we see that generic values produce decompositions of size $6$.  Thus, the space of decompositions of size $6$ is $5$-dimensional and decompositions of size $6$ can be efficiently obtained.

    Now to justify Assumption \ref{ass:dehomogenization} we give an explicit example in our supplementary code where we act on this $\phi$ via a random change of basis, so that Assumption \ref{ass:dehomogenization} holds with probability $1$.  We then show numerically we obtain the same conclusions.  The approach can also be justified with an application of Sylvester's algorithm.
\end{example}

\section{Even order tensors with decompositions of low regularity}\label{sec:lowrankeven}
We now focus our analysis on even order tensors with decompositions of low regularity.
Let $D=\floor{(d-1)/2}$.  For odd order tensors $\operatorname{Cat}_D(\phi)$ is the most square catalecticant.  However, for even $d$, $D$ is strictly smaller than $d/2$.  Because $\operatorname{Cat}_{d/2}(\phi)$ can be accessed through linear algebra operations, we ask if an efficient moment matrix extension algorithm exists for even $d$ when $\phi$ does not satisfy Condition \ref{cond:really_low_rank}.

\begin{condition}\label{cond:low_rank}
    $\phi$ satisfies 
    \[\operatorname{rank}\bb{\operatorname{Cat}_{\floor{d/2}}(\phi)}=\operatorname{rank}(\phi).\]
\end{condition}
This implies that for any minimal rank decomposition $\Z$ of $\phi$ we have $\rho(\Z)\leq \floor{d/2}$.  In particular, by Theorem \ref{thm:Bk_refine} there exists a monomial set $B$ such that 
\begin{enumerate}
    \item $|B|=\operatorname{rank}(\phi)$,
    \item $B$ is connected to $1$,
    \item $\operatorname{deg}(B)\leq \floor{\frac{d}{2}}$.
\end{enumerate}
We first show that such a $B$ can be found efficiently.  This also proves Claim \ref{claim:really_low_rank_B}. Indeed, if $\operatorname{Cat}_{\floor{(d-1)/2}}(\phi)=\operatorname{rank}(\phi)$, then $\operatorname{Cat}_{\floor{d/2}}(\phi)=\operatorname{rank}(\phi)$ as well.

\begin{algorithm}[h!]
\hspace*{\algorithmicindent} \textbf{Input: }$\phi\in S^d\bbC^{n+1}$ satisfying Assumption \ref{ass:dehomogenization} and Condition \ref{cond:low_rank}\\
\hspace*{\algorithmicindent} \textbf{Output: }A valid monomial set $B$
\begin{algorithmic}[1]
    \STATE Record $h_\phi(k)$, $k=0, \dots, \floor{\frac{d}{2}}$; $r:=h_{\phi}\bb{\floor{\frac{d}{2}}}$
    \STATE $B\leftarrow \{1\}$
    \FOR{$k=1, \dots, \floor{\frac{d}{2}}$}
        \STATE Restrict the matrix $\operatorname{Cat}_{\floor{d/2}}(\phi)$ to columns indexed by $x_1B_{k-1}\cup \dots \cup x_nB_{k-1}$; denote this matrix by $\operatorname{Cat}_{\floor{d/2}}(\phi)_{x_i B_{k-1}}$
        \STATE Select $h_{k}(\phi)-h_{k-1}(\phi)$ linearly independent columns from $\operatorname{Cat}_{\floor{d/2}}(\phi)_{x_i B_{k-1}}$; let $B_k$ index these columns
        \STATE $B\leftarrow B\cup B_k$
    \ENDFOR
\end{algorithmic}
\caption{Finding a valid monomial set $B$}
\label{alg:finding_B}
\end{algorithm}

\begin{proposition}\label{prop:found_B}
    Let $\phi$ satisfy Assumption \ref{ass:dehomogenization} and Condition \ref{cond:low_rank}.  Then Algorithm \ref{alg:finding_B} finds a monomial set $B$ satisfying $|B|=\operatorname{rank}(\phi)$, $B$ is connected to $1$, and $\operatorname{deg}(B)\leq \floor{d/2}$ such that $\mathcal{H}_{B, B}$ is fully specified and invertible. Moreover, $\operatorname{deg}(B)$ is as small as possible.  The algorithm is also efficient.
\end{proposition}
\begin{proof}
    Under Condition \ref{cond:low_rank} $r=\operatorname{rank}(\phi)$.  Let $\Z$ be any minimal rank decomposition of $\phi$.  We know that $\operatorname{Cat}_{k}(\phi)=\Z_{d-k}^\top \Lambda \Z_k$; moreover, $\Z_{k}$ is rank $r$ for all $k\geq \floor{d/2}$.  Therefore, $\Z_k^\top$ is injective for all $k\geq \floor{d/2}$.  We then have $h_{\phi}(k)=h_{\Z}(k)$ for all $k\leq \floor{d/2}$.  

    Let $\bb{\operatorname{Cat}_{\floor{d/2}}(\phi)}_k$ be the submatrix of $\operatorname{Cat}_{\floor{d/2}}(\phi)$ restricted to columns corresponding to $|\gamma|\leq k$ for $k\leq \floor{d/2}$.  Then
    \[\bb{\operatorname{Cat}_{\floor{d/2}}(\phi)}_k=\Z_{d-\floor{d/2}}^\top \Lambda \Z_k.\]
    Because $\Z_{d-\floor{d/2}}^\top, \Lambda$ are injective, linearly independent columns of $\bb{\operatorname{Cat}_{\floor{d/2}}(\phi)}_j$ correspond exactly to linearly independent columns of $\Z_j$.  By the proof of Proposition \ref{prop:connectedto1} we need only to look for linearly independent columns in $\Z_{x_i B_{k-1}}$ for each $k$, which correspond exactly to linearly independent columns of $\operatorname{Cat}_{\floor{d/2}}(\phi)_{x_i B_{k-1}}$.  

    This procedure is efficient in our model of computation as we need only to compute the rank of matrices a polynomial (in the number of tensor entries) number of times.
\end{proof}
In the context of Algorithm \ref{alg:size_s_decomp}, a valid $\mathcal{B}_{\operatorname{rank}(\phi)}$ can be taken to be $\{B\}$ with $B$ found via the subroutine Algorithm \ref{alg:finding_B}.  Note that the $B$ found need not be the only connected to $1$, degree bounded monomial set such that $\Z_B$ is invertible, but any one of them suffices to construct and solve the determinantal relations.

\subsection{Determinantal equations}

When $d$ is odd, Condition \ref{cond:low_rank} is equivalent to  Condition \ref{cond:really_low_rank}.  On the other hand, for $d$ even, when Algorithm \ref{alg:finding_B} finds $B$ such that $\operatorname{deg}(B)=d/2$ we are outside the regime of the simultaneous diagonalization algorithm.  In this case the set of moment variables $Y$ will be nonempty and there will be nontrivial determinantal relations.

\begin{lemma}\label{lem:equation_deg}
    Let $d$ be even, $B$ be such that $\operatorname{deg}(B)= d/2$ and $B$ chosen by Algorithm \ref{alg:finding_B}.  Then the expression
    \[\operatorname{det}\bb{\mathcal{H}_{B\cup \{x^\alpha x_i\}, B\cup\{x^\beta x_j\}}}-\operatorname{det}\bb{\mathcal{H}_{B\cup \{x^\alpha x_j\}, B\cup\{x^\beta x_i\}}}\]
    is
    \begin{enumerate}[(1)]
        \item the zero polynomial if $|\alpha+\beta|\leq d-2$ or $\alpha=\beta$ or 
        \[\bb{x^\alpha x_i\in B\vee x^\beta x_j\in B}\wedge \bb{x^\alpha x_j\in B\vee x^\beta x_i\in B};\]
        \item linear if $|\beta|= d/2-1$ and $|\alpha|= d/2$, and $x^\beta x_i\notin B\vee x^\beta x_j\notin B$ (or vice-versa for $\alpha$);
        \item quadratic otherwise.
    \end{enumerate}
\end{lemma}
\begin{proof}
    \begin{enumerate}[(1)]
        \item If 
        \[\bb{x^\alpha x_i\in B\vee x^\beta x_j\in B}\wedge \bb{x^\alpha x_j\in B\vee x^\beta x_i\in B}\]
        then there is a repeated row or column in both matrices, so by definition the determinant is zero.  Next, suppose $|\alpha+\beta|\leq d-2$.  Then $\operatorname{deg}(x^{\alpha+\beta}x_ix_j)\leq d$.  Consider the submatrix of $\mathcal{H}_{B\cup \{x^{\alpha} x_i\}, B\cup\{x^\beta x_j\}}$ given by $\mathcal{H}_{B_{\leq |\alpha|+1}\cup \{x^\alpha x_i\}, B_{\leq |\beta|+1}\cup \{x^\beta x_j\}}$.  This matrix is fully specified, i.e., no moment variables present, and is a submatrix of $\bb{\operatorname{Cat}_{d/2}(\phi)}_{|\beta|+1}$.  But we know Algorithm 2 guarantees $|B|_{\leq |\beta|+1}=\operatorname{rank}\bb{\bb{\operatorname{Cat}_{d/2}(\phi)}_{|\beta|+1}}$, and so $\mathcal{H}_{B_{\leq |\alpha|+1}\cup \{x^\alpha x_i\}, B_{\leq |\beta|+1}\cup \{x^\beta x_j\}}$ must have linearly dependent columns.  Then \\ $\operatorname{det}\bb{\mathcal{H}_{B\cup \{x^{\alpha} x_i\}, B\cup\{x^\beta x_j\}}}=0$.  The same holds for $\operatorname{det}\bb{\mathcal{H}_{B\cup \{x^{\alpha} x_j\}, B\cup\{x^\beta x_i\}}}$.

        \item Without loss of generality suppose $|\beta|= d/2-1$ and $x^\beta x_j\notin B$ and $|\alpha|=d/2$, so that $x^\alpha x_i$ indexes the last row and $x^\beta x_j$ the last column of $\mathcal{H}_{B\cup\{x^\alpha x_i\}, B\cup\{x^\beta x_j\}}$.  The row corresponding to $x^\alpha x_i$ will contain moment variables $y_{x^{\alpha+\alpha'} x_i}$, where $x^{\alpha'}\in B_{d/2}$, and if $|\beta|=d/2-1$, $y_{x^{\alpha+\beta} x_i x_j}$.  Because $|\beta|= d/2-1$, the column corresponding to $x^\beta x_j$ will be fully specified outside of the last entry.  Therefore, by the Laplace expansion the determinant is linear in the moment variables.  The difference of two linear expressions will still be linear.
        \item If both $|\alpha|, |\beta|=d/2$ then there are moment variables in both the row and column corresponding to $x^\alpha x_i, x^\beta x_j$, respectively.  Again, by the Laplace expansion the determinant is quadratic in the moment variables.\qedhere
    \end{enumerate}
\end{proof}
Under Assumption \ref{ass:dehomogenization} and Condition \ref{cond:low_rank} the polynomial system of linear and quadratic equations will admit at least one solution, which corresponds to a minimal rank decomposition.  We are interested when \textit{only linear equations} suffice, that is, when all solutions to the moment matrix extension problem can be specified by linear equations (or a sequence of linear systems).  This motivates Algorithm \ref{alg:efficient_alg}.

\begin{algorithm}[h]
\hspace*{\algorithmicindent} \textbf{Input: }$\phi\in S^d\bbC^{n+1}$, $d$ even, satisfying  Assumption \ref{ass:dehomogenization} and  Condition \ref{cond:low_rank}\\
\hspace*{\algorithmicindent} \textbf{Output: }A minimal rank decomposition of $\phi$ or FAIL, quadratic relations remain
\begin{algorithmic}[1]
    \STATE Construct the Hankel matrix $\mathcal{H}$ of $\phi$
    \STATE Find $B$ via Algorithm \ref{alg:finding_B}
    \STATE $E\leftarrow$ linear determinantal relations
    \WHILE{$E\neq \varnothing$}
        \STATE Rewrite a subset of the moment variables as a linear combination of the other moment variables specified by $E$
        \STATE Substitute these into the quadratic relations
        \STATE $E\leftarrow$ new linear equations (if they exist)
    \ENDWHILE
    \IF{all of the moment variables are fully specified (maybe after choosing parameters)}
        \STATE Obtain a decomposition $\z_1, \dots, \z_r$ via an eigendecomposition and solve a linear system for $\lambda_1, \dots, \lambda_r$
    \ELSE
        \STATE Output FAIL; quadratic relations remain
    \ENDIF
\end{algorithmic}
\caption{Efficient minimal rank decomposition}
\label{alg:efficient_alg}
\end{algorithm}

Algorithm \ref{alg:efficient_alg} is efficient if there are a polynomial number of moment variables, each linear system in the while loop is of polynomial size, and there are a polynomial number of loops.  We remark that there is a possibility that at the end of the while loop the practitioner can choose parameters; these are moment variables left unspecified by the determinantal relations and can be freely chosen to have generic values.  If there are no moment variables left unspecified then as there is at most one solution to a linear system then the decomposition is also unique.  We will elaborate on the sufficiency of a linear system as a an effective criterion for specific identifiability \cite{chiantini2017effective}.

\subsubsection{Linear equations}
We now analyze the structure of the linear equations.  From Lemma \ref{lem:equation_deg}, consider $\operatorname{det}\bb{\mathcal{H}_{B\cup\{x^\alpha x_i\}, B\cup\{x^\beta x_j\}}}$ for $|\alpha|=d/2, |\beta|= d/2-1$ and $i\neq j$ with $x^\beta x_j\not\in B$.  Expanding the determinant with respect to the row $x^\alpha x_i$, we have $ \operatorname{det}\bb{\mathcal{H}_{B\cup\{x^\alpha x_i\}, B\cup\{x^\beta x_j\}}}$ is equal to 
\begin{align*}
    \operatorname{det}\bb{\mathcal{H}_{B, B}}\mathcal{H}_{x^\alpha x_i, x^\beta x_j}+\sum_{x^\gamma\in B}(-1)^{r+1+\operatorname{pos}(x^\gamma)}\operatorname{det}\bb{\mathcal{H}_{B, B\setminus \{x^\gamma\}\cup \{x^\beta x_j\}}}\mathcal{H}_{x^\alpha x_i, x^\gamma},
\end{align*}
where $\operatorname{pos}(x^\gamma)$ is equal to $k$ if $x^\gamma$ is the $k^{\text{th}}$ element in $B$ in some ordering we fix beforehand, e.g., graded lexicographic ordering. The entry $\mathcal{H}_{\alpha+e_i, x^\gamma}$ is a variable only if $|\alpha+\gamma|\geq d$, which is the case when $|\gamma|=d/2$.  Otherwise, $\mathcal{H}_{x^\alpha x_i, x^\gamma}=\phi_{x^{\alpha+\gamma}x_i}$.  Moreover, if $\Z$ is some minimal rank decomposition of $\phi$ and $S$ is any index set then
\[\mathcal{H}_{B, S}=\Z_B^\top \Z_S.\]
Since $|B|=\operatorname{rank}(\phi)$, $\Z_B$ is a square matrix.
If $|S|=\operatorname{rank}(\phi)$ as well, then $\Z_S$ is also square, and thus we can divide the expression $\operatorname{det}\bb{\mathcal{H}_{B\cup\{x^\alpha x_i\}, B\cup \{x^\beta x_j\}}}$ by the factor $\operatorname{det}\bb{\Z_B}$ (which is nonzero) to obtain the equivalent expression
\begin{equation}\label{eq:new_lin_eq}
\begin{aligned}
     &\operatorname{det}\bb{\Z_B}y_{x^{\alpha+\beta}x_ix_j}+\sum_{x^{\gamma}\in B_{d/2}}(-1)^{r+1+\operatorname{pos}(x^\gamma)}\operatorname{det}\bb{\Z_{B\setminus\{x^\gamma\}\cup \{x^\beta x_j\} }}y_{x^{\alpha+\gamma}x_i}\\
     &\quad +\sum_{x^\gamma\in B_{ d/2-1}} (-1)^{r+1+\operatorname{pos}(x^\gamma)}\operatorname{det}\bb{\Z_{B\setminus\{x^\gamma\}\cup \{x^\beta x_j\} }}\phi_{x^{\alpha+\gamma}x_i}.
\end{aligned}
\end{equation}
Notationally, we henceforth write
\begin{equation}
    \begin{aligned}
        m_{x^\gamma}^{x^\delta}&:=\operatorname{det}\bb{\Z_{B\setminus \{x^\gamma\}\cup \{x^\delta\}}},\\
        m&:=\operatorname{det}\bb{\Z_B}.
    \end{aligned}
\end{equation}
A linear relation is taken to be the difference of two of the expressions in Equation \ref{eq:new_lin_eq} and setting equal to $0$.  For some fixed ordering of the equations and variables, write the linear system as 
\begin{equation}
    \mathbf{A} Y=\mathbf{b},
\end{equation}
where the rows of $A$ index equations and the columns index variables and with some abuse of notation $Y$ is the vector of moment variables.

We would like to argue that when the number of equations exceeds the number of variables (i.e., when $\mathbf{A}$ is tall), then $\mathbf{A}$ is generically full column rank.  That is, the condition that $\mathbf{A}$ is not full rank is the vanishing of the $|Y|\times |Y|$ minors of $\mathbf{A}$; each entry of $\mathbf{A}$ is some  $r\times r$ minor of $\operatorname{Cat}_{d/2}(\phi)$.  By the discussion above these are polynomials in the entries of $\Z$; thus, we would like at least one of these $|Y|\times |Y|$ minors of $\mathbf{A}$ to be a nonzero polynomial in the entries of $\Z$.

This would imply that the task of tensor decomposition, having already been expressed as a moment matrix extension problem, would further reduce to solving the linear system $\mathbf{A}Y=\mathbf{b}$: if $\mathbf{A}$ is full column rank then values for moment variables can be efficiently found.  Moreover, this would imply the resulting decomposition is unique.

However, a precise analysis is complicated.  We do not know a priori what $B$ is output by Algorithm \ref{alg:finding_B}, and for each $B$ obtaining a precise count of the number of linear equations and moment variables is prohibitive.
Further, the counts alone would not prove that $\mathbf{A}$ is generically full column rank when it has more rows than columns.  For these reasons we focus on $d=4$ in the next section. With some simplifications we can greatly improve our analysis.

\begin{remark}
    Observe that the polynomials dictating the rank of $\mathbf{A}$ are certain maximal $r\times r$ minors of $\Z_{d/2}$.  This is related to a line of work studying equations defining the Hilbert scheme of points \cite{alonso2009hilbert, SKJELNES2022254}.  These include Pl\"ucker-like relations defining the Hilbert scheme as a subscheme of a single Grassmannian.  Our task is equivalent to showing that the polynomials derived from $\mathbf{A}$ being not full column rank are not in the radical of these defining equations of the Hilbert scheme of points.
\end{remark}

\subsection{Order four symmetric tensors}
For $d=4$, Condition \ref{cond:low_rank} reduces to $\operatorname{rank}\bb{\operatorname{Cat}_2(\phi)}=\operatorname{rank}(\phi)=:r$.  Taking both Assumption \ref{ass:dehomogenization} and Assumption \ref{ass:concise} to hold, we can assume that $B$ is of the form
\[B=\{1, x_1, \dots, x_n\}\cup B_2,\]
where $B_2$ is some set of degree two monomials such that $|B_2|=r-(n+1)$.  In this case, the moment variables are given by
\begin{equation*}
    Y:=\{x^{\alpha+\alpha'}x_i\mid x^\alpha, x^{\alpha'}\in B_2, 1\leq i\leq n\}.
\end{equation*}
Consider the sets:
\begin{equation*}
\begin{aligned}
        E_1 &:=\{(x^\alpha x_i, x^\beta x_j)\mid x^\alpha\in B_2, x^\beta\in B_1, i\neq j, x^\beta x_j\notin B_2, x^\beta x_i\in B_2\},\\
        E_2 &:=\{[(x^\alpha x_i, x^\beta x_j), (x^\alpha x_j, x^\beta x_i)]\mid x^\alpha\in B_2, x^\beta\in B_1, i<j, x^\beta x_i, x^\beta x_j\notin B_2\},\\
        E &:=E_1 \cup E_2.
    \end{aligned}
\end{equation*}
By Lemma \ref{lem:equation_deg}, $E$ indexes the linear equations. With some abuse of terminology we also refer the tuples in $E$ as linear equations.  Equations in $E_1$ index linear equations defined as
\[\operatorname{det}\bb{\mathcal{H}_{B\cup\{x^\alpha x_i\}, B\cup\{x^\beta x_j\}}}=0,\]
while equations in $E_2$ index linear equations defined as
\[\operatorname{det}\bb{\mathcal{H}_{B\cup\{x^\alpha x_i\}, B\cup\{x^\beta x_j\}}}-\operatorname{det}\bb{\mathcal{H}_{B\cup\{x^\alpha x_j\}, B\cup\{x^\beta x_i\}}}=0.\]

\subsubsection{Simplifications}\label{sec:simplifications}
To simplify analysis slightly we impose an additional condition.

\begin{condition}\label{cond:first_r}
    Take $\mathcal{H}_{B, B}$ to be invertible for $B$ equal to the first $r$ monomials in the graded lexicographic ordering.
\end{condition}

In other words, if $\phi$ satisfies this condition and Condition \ref{cond:low_rank} we can always take $|B|=r$ with
\[B=\{1, x_1, \dots, x_n, x_1^2, \dots, x_1 x_n, \dots, x_i^2, \dots, x_i x_j\}.\]
This condition is common in prior work. Indeed, this is assumed in Equation 2.13 of \cite{nie2017generating} without much further discussion.  At the very least, it is an open condition -- generic tensors satisfying Condition \ref{cond:low_rank} also satisfy this additional condition.  However, we note it does not always hold.

\begin{example}\label{ex:not_first_r}
    Let $n=3, d=4$.  Let $\Z=\{\z_1, \dots, \z_7\}$ be a set of points such that $\z_{i, 3}=0$ for $i=1, \dots, 6$, but the other coordinates of points in $\Z$ are generic.
    Let
    \[B=\{1, x_1, x_2, x_3, x_1^2, x_1x_2, x_1x_3\},\]
   be the set consisting of the first seven monomials in the graded lexicographic order.  Then $\operatorname{rank}(\Z_1)=4, \operatorname{rank}(\Z_2)=7$.  But under any change of basis, $\operatorname{det}(\Z_{B})=0$.  Thus, $\mathcal{H}_{B, B}=\Z_B^\top \Z_B$ is not full rank.  However, choosing 
    \[B'=\{1, x_1, x_2, x_3, x_1^2, x_1x_2, x_2^2\}\]
    gives a full rank $\Z_{B'}$ and $\mathcal{H}_{B', B'}$ generically.

    Note that the points in $\Z$ are not in general linear position, and moreover, we can see that $\z_1, \dots, \z_6$, if considered as elements of $\bbC^{3}$, are such that $\z_1^{\otimes 2}, \dots, \z_6^{\otimes 2}$ span all of $\bbC^6$.  If viewed this way, the set 
    \[B''=\{1, x_1, x_2, x_1^2, x_1x_2, x_2^2\}\]
    indexes an invertible submatrix for the second Vandermonde matrix of these $6$ points.  Adding $\z_7$ then corresponds to embedding $B''$ into $B'$ in the obvious way and then adding the extra monomial $x_3$.
\end{example}

 Condition \ref{cond:first_r} will make counting $|Y|$ and $|E|$ more feasible.  We give our first results for a coarsening of $r$, taking $r=\sum_{j=0}^c (n-j+1)$ for some integer $1\leq c\leq n$.  Then $B$ takes the form
\[B=\{1, x_1, \dots, x_n, x_1^2, \dots, x_1x_n, \dots, x_c^2, \dots, x_cx_n\}.\]
For $\eta\in \bbZ_{\geq 0}^n$ we write $\eta_{\mathcal{C}}:=\sum_{j=1}^c \eta_j$.  

\begin{lemma}\label{lem:E1}
    Let $r=\sum_{j=0}^c (n-j+1)$.  Under Condition \ref{cond:low_rank} and Condition \ref{cond:first_r} 
    \[E_1=\{(x^\eta, x^\theta)\mid |\eta|=3, |\theta|=2, \eta_{\mathcal{C}}\geq 2, \theta_{\mathcal{C}}=0\}.\]
\end{lemma}
\begin{proof}
    Recall equations in $E_1$ are of the form $(x^\alpha x_i, x^\beta x_j)$ for $|\alpha|=2, |\beta|=1, i\neq j$ with $x^\beta x_j\notin B, x^\beta x_i\in B$.  First, we know that $j>c$ because otherwise $x^\beta x_j\in B$.  We also must have $x^\beta =x_k$ for $k>c$ as well, for the same reason.  Next, if $\alpha_\mathcal{C}=2$ then we automatically have $\eta=\alpha+e_i$ is such that $\eta_{\mathcal{C}}\geq 2$.  So suppose $\alpha_{\mathcal{C}}=1$, and $i>c$.  But then $x^\beta x_i\notin B$, a contradiction.  So we must have $i\leq c$ in this case.

    Conversely, consider some $(x^\eta, x^\theta)$ with the prescribed conditions.  Then $x^\eta=x_{a_1}x_{a_2}x_i$, where we have $a_1, i\leq c$ and $x^\theta=x_bx_j$ for $b, j>c$.  Setting $x^\alpha=x_{a_1}x_{a_2}$, which is in $B$, and $x^\beta=x_b$ we have $x^\beta x_i\in B$.
\end{proof}

We can now give precise counts of $|Y|$ and $|E_1|$.
\begin{theorem}\label{thm:counts}
    Let $\phi\in S^4\bbC^{n+1}$ satisfy Condition \ref{cond:low_rank} and Condition \ref{cond:first_r} with $r=\sum_{j=0}^c (n-j+1)$.  Then
    \[
        |Y|=\binom{c+4}{5}+(n-c)\binom{c+3}{4}+\binom{n-c+1}{2}\binom{c+2}{3}+\binom{n-c+2}{3}\binom{c+1}{2}
    \]
    and 
    \[
        |E_1|=\binom{n-c+1}{2}\bb{\binom{c+2}{3}+(n-c)\binom{c+1}{2}}.
    \]
\end{theorem}
\begin{proof}
    If $r=\sum_{j=0}^c (n-j+1)$ then
    \[B=\{1, x_1, \dots, x_n, x_1^2, \dots, x_1x_n, \dots, x_c^2, \dots, x_cx_n\}.\]
    In particular, if $\gamma_i\neq 0$ for $i=1, \dots, c$ then $x^\gamma \in B$.  
    
    Elements of $Y$ are of the form $x^{\alpha+\alpha'}x_i$, $x^\alpha, x^{\alpha'}\in B$, so in particular, we must have $\alpha_{\mathcal{C}}, \alpha'_{\mathcal{C}}\geq 1$.  Then $\bb{\alpha+\alpha'+e_i}_{\mathcal{C}}\geq 2$.  If $\bb{\alpha+\alpha'+e_i}_{\mathcal{C}}=\sum_{j=1}^c \alpha_j+\alpha'_j+1_{i}= k$ then $\sum_{j=c+1}^n \alpha_j+\alpha'_j+1_{i}= 5-k$.  Thus, for $k=2, \dots 5$ we count the number of monomials of degree $k$ in $c$ variables multiplied by the number of monomials of degree $5-k$ in $n-c$ variables and then sum.  

    Similarly, for $E_1$ we count the possible $(x^\alpha x_i, x^\beta x_j)$ for $x^\alpha\in B_2, x^\beta\in B_1$, $i\neq j$, where $x^\beta x_i\in B, x^\beta x_j\notin B$.  In particular, this latter condition implies $i\leq c$.  $x^\beta x_j$ then ranges over all monomials of degree $2$ not in $B_2$; this is the number of monomials of degree $2$ on $n-c$ variables.  On the other hand, $x^\alpha x_i$ is such that $\bb{\alpha+e_i}_{\mathcal{C}}=2, 3$.  If it is $3$ then $x^\alpha x_i$ ranges over all monomials of degree $3$ on $c$ variables, and if it is $3$ then it ranges over all monomials of degree $3$ where component of degree $2$ only involves $x_1, \dots, x_c$ and one variable ranges over $x_{c+1}, \dots, x_n$.  The latter case gives the count $(n-c)\binom{c+1}{2}$ and the former case gives the count $\binom{c+2}{3}$; these are multiplied by $\binom{n-c+1}{2}$.
\end{proof}

\begin{proposition}\label{prop:counts_inductions}
    Assume Condition \ref{cond:low_rank} and Condition \ref{cond:first_r} hold and fix the format $(n, r)$.  Let $c$ be the smallest integer such that $r\leq \sum_{j=0}^c (n-j+1)$.  Write $E^{(n, r)}$, $Y^{(n, r)}$ to be the set of equations and variables for this format.  If $|E^{(n, r)}|\geq |Y^{(n, r)}|$ then $|E^{(n+1, r+c)}|\geq |Y^{(n+1, r+c)}|$.
\end{proposition}

The proof of this proposition is in the appendix.  In other words, as $n$ increases, for some fixed ``final element'' $x_cx_j$ of $B$, the number of equations increases faster than the number of variables.  It suffices to count the coarse ranks as these determine a fixed ``final element.''  When $|E_1|$ is larger than $|Y|$ we would like to show that the corresponding matrix $\mathbf{A}$ is generically full column rank.  This is the topic of the next section.

\subsubsection{Conjecture for efficient decomposition of generic tensors}\label{sec:conjecture}
In this section we first show that Algorithm \ref{alg:efficient_alg} can exceed the undercomplete setting for order-$4$ tensors by proving that generic tensors of rank up to $2n+1$ are efficiently decomposable.  We then conjecture and provide evidence, including computer assisted proofs for $n=2, \dots, 17$, that generic tensors of rank $O(n^2)$ are efficiently decomposable via Algorithm \ref{alg:efficient_alg}.  We first give some preliminary definitions and results that streamline notation and proofs.

\begin{definition}
    We say that $(n, r)$ is an \textit{efficient format} if for a tensor satisfying Assumption \ref{ass:dehomogenization}, Assumption \ref{ass:concise},  Condition \ref{cond:low_rank} and Condition \ref{cond:first_r}, the coefficient matrix $\mathbf{A}\in \bbC^{|E|\times |Y|}$ is generically full column rank (with the open condition being a nonvanishing of some $|Y|\times |Y|$ minor of $\mathbf{A}$).

    We say that $(n, r)$ is a \textit{very efficient format} if under the same set of assumptions and conditions, the submatrix $\mathbf{A}^1\in \bbC^{|E_1|\times |Y|}$ of $\mathbf{A}$ (using only the equations in $E_1$) is generically full column rank.
\end{definition}
Of course, if $(n, r)$ is a very efficient format it is also an efficient format.  It is also clear that if $|E|<|Y|$ (resp. $|E_1|<|Y|$) then $(n, r)$ cannot be an efficient (resp. very efficient) format.  

If $(n, r)$ is an efficient format generic tensors of rank $r$ in $S^4\bbC^{n+1}$ Algorithm \ref{alg:efficient_alg} gives a decomposition efficiently.  Moreover, the matrix $\mathbf{A}$ being full column rank serves as an \textit{effective criterion} as it then certifies identifiability on a dense subset of the set of rank $r$ tensors; see Definition 4.1 of \cite{chiantini2017effective}.

Let $\mathbf{A}^1$ be the submatrix of $\mathbf{A}$ given by only taking the equations in $E_1$.  Recall from Lemma \ref{lem:E1} that $E_1$ has the form
\[E_1=\{(x^\eta, x^\theta)\mid |\eta|=3, |\theta|=2, \eta_{\mathcal{C}}\geq 2, \theta_{\mathcal{C}}=0\}.\]
Moreover, from the proof of Theorem \ref{thm:counts} we have that
\begin{align*}
    Y&=Y_5\cup Y_4\cup Y_3\cup Y_2,\\
    Y_k &=\{x^\gamma\mid \gamma_{\mathcal{C}}=k\}.
\end{align*}
Then a row $\mathbf{A}^1_{(x^\eta, x^\theta)}\in \bbC^{|Y|}$ satisfies, for $x^\gamma\in Y$
\[
\bb{\mathbf{A}^1_{(x^\eta, x^\theta)}}_{x^\gamma} =
\begin{cases}
    (-1)^{r+1+\operatorname{pos}(\gamma-\eta)}m_{\gamma-\eta}^\theta, \quad &  \gamma\in B_2+\eta,\\
    m,\quad & \gamma=\eta+\theta,\\
    0,\quad & \text{otherwise}.
\end{cases}
\]
Note that the first two cases are indeed disjoint as $\theta\notin B_2$.  

We make a series of reductions from $\mathbf{A}^1$ to a smaller matrix $\mathbf{N}$ such that $\mathbf{A}^1$ is full column rank if and only if $\mathbf{N}$ is full column rank.  Define
\begin{equation}
    Y'=Y_5\cup Y_4\cup Y_3
\end{equation}
and 
\begin{equation}
    \begin{aligned}
    E_{1, 1}'&:=\{(x^\eta, x_j x_k)\mid |\eta|=3, \eta_{\mathcal{C}}, j\leq k\},\\
    E_{1, 2}'&:=\{((x^\delta x_k, x_kx_{\ell}), (x^\delta x_\ell, x_k^2))\mid |\delta|=2, \delta_{\mathcal{C}}=2, k\neq \ell\},\\
    &\quad \cup \{((x^\delta x_j, x_kx_\ell), (x^\delta x_k, x_j x_\ell))\mid |\delta|=2, \delta_{\mathcal{C}}=2, j<k, k\neq \ell\},\\
    E_{1}'&:=E_{1, 1}'\cup E_{1, 2}'.
\end{aligned}
\end{equation}
It is easy to see from the proof Theorem \ref{thm:counts} and the construction of the sets that
\begin{equation}
\begin{aligned}
    |Y'|&=\binom{c+4}{5}+(n-c)\binom{c+3}{4}+\binom{n-c+1}{2}\binom{c+2}{3},\\
    |E_1'|&=\binom{n-c+1}{2}\binom{c+2}{3}+\binom{c+1}{2}\bb{\binom{n-c+1}{2}(n-c)-\binom{n-c+2}{3}}.
\end{aligned}
\end{equation}
Define ${\mathbf{A}^1}'\in \bbC^{|E_1'|\times |Y'|}$ as the matrix with rows $E_1'$ and columns $Y'$.  

\begin{example}\label{ex:A_matrix}
    Let $n=4, c=1$.  Then 
    \[B=\{1, x_1, x_2, x_3, x_4, x_1^2, x_1x_2, x_1x_3, x_1x_4\}.\]
    The rows correspond to
    \begin{align*}
         & (x_1^3, x_2^2),
         (x_1^3, x_2x_3),
         (x_1^3, x_2x_4),
         (x_1^3, x_3^2),
         (x_1^3, x_3x_4),
         (x_1^3, x_4^2),\\
        &[(x_1^2x_2, x_2x_3), (x_1^2x_3, x_2^2)],
        [(x_1^2x_2, x_2x_4), (x_1^2x_4, x_2^2)],
        [(x_1^2x_2, x_3^2), (x_1^2x_3, x_2x_3)],\\
        & [(x_1^2x_2, x_3x_4), (x_1^2x_3, x_2x_4)]
        [(x_1^2x_2, x_3x_4), (x_1^2x_4, x_2x_3)],
        [(x_1^2x_2, x_4^2), (x_1^2x_4, x_2x_4)],\\
        &[(x_1^2x_3, x_3x_4), (x_1^2x_4, x_3^2)],
        [(x_1^2x_3, x_4^2), (x_1^2x_4, x_3x_4)]
    \end{align*}
    and columns to
    \[
         x_1^5,
         x_1^4x_2,
         x_1^4x_3,
         x_1^4x_4,
         x_1^3x_2^2,
         x_1^3x_2x_3,
         x_1^3x_2x_4,
         x_1^3x_3^2,
         x_1^3x_3x_4,
         x_1^3x_4^2.
    \]
        The matrix ${\mathbf{A}^1}'$ takes the form
    \[
    \bb{{\mathbf{A}^1}'}_{Y_5}= \begin{bmatrix}
        m_{x_1^2}^{x_2x_3} \\
        m_{x_1^2}^{x_2x_3)} \\
        m_{x_1^2}^{x_2x_4} \\
        m_{x_1^2}^{x_3^2} \\
        m_{x_1^2}^{x_3x_4} \\
        m_{x_1^2}^{x_4^2} \\ 
        0 \\
        0 \\
        0 \\
        0 \\
        0 \\
        0 \\
        0 \\
        0
    \end{bmatrix},\quad 
    \bb{{\mathbf{A}^1}'}_{Y_4}= \begin{bmatrix}
        -m_{x_1x_2}^{x_2^2} & m_{x_1x_3}^{x_2^2} & -m_{x_1x_4}^{x_2^2} \\
        -m_{x_1x_2}^{x_2x_3} & m_{x_1x_3}^{x_2x_3} & -m_{x_1x_4}^{x_2x_3} \\
        -m_{x_1x_2}^{x_2x_4} & m_{x_1x_3}^{x_2x_4} & -m_{x_1x_4}^{x_2x_4} \\
        -m_{x_1x_2}^{x_3^2} & m_{x_1x_3}^{x_3^2} & -m_{x_1x_4}^{x_3^2} \\
        -m_{x_1x_2}^{x_3x_4} & m_{x_1x_3}^{x_3x_4} & -m_{x_1x_4}^{x_3x_4} \\
        -m_{x_1x_2}^{x_4^2} & m_{x_1x_3}^{x_4^2} & -m_{x_1x_4}^{x_4^2} \\
        m_{x_1^2}^{x_2x_3} & -m_{x_1^2}^{x_2^2} & 0 \\
        m_{x_1^2}^{x_2x_4} & 0 & -m_{x_1^2}^{x_2^2} \\
        m_{x_1^2}^{x_3^2} & -m_{x_1^2}^{x_2x_3} & 0 \\
        m_{x_1^2}^{x_3x_4} & -m_{x_1^2}^{x_2x_4} & 0 \\
        m_{x_1^2}^{x_3x_4} & 0 & -m_{x_1^2}^{x_2x_3} \\
        m_{x_1^2}^{x_4^2} & 0 & -m_{x_1^2}^{x_2x_4} \\
        0 & m_{x_1^2}^{x_3x_4} & -m_{x_1^2}^{x_3^2} \\
        0 & m_{x_1^2}^{x_4^2} & -m_{x_1^2}^{x_3x_4} \\
    \end{bmatrix},
    \]
    \[
    \bb{{\mathbf{A}^1}'}_{Y_3}=\begin{bmatrix}
        m & 0 & 0 & 0 & 0 & 0 \\
        0 & m & 0 & 0 & 0 & 0 \\
        0 & 0 & m & 0 & 0 & 0 \\
        0 & 0 & 0 & m & 0 & 0 \\
        0 & 0 & 0 & 0 & m & 0 \\
        0 & 0 & 0 & 0 & 0 & m \\
        -m_{x_1x_2}^{x_2x_3} & m_{x_1x_3}^{x_2x_3}+m_{x_1x_2}^{x_2^2} & -m_{x_1x_4}^{x_2x_3} & -m_{x_1x_3}^{x_2^2} & m_{x_1x_4}^{x_2^2} & 0 \\
        -m_{x_1x_2}^{x_2x_4} & m_{x_1x_3}^{x_2x_4} & -m_{x_1x_4}^{x_2x_4}+m_{x_1x_2}^{x_2^2} & 0 & -m_{x_1x_3}^{x_2^2} & m_{x_1x_4}^{x_2^2} \\
        -m_{x_1x_2}^{x_3^2} & m_{x_1x_3}^{x_3^2}+m_{x_1x_2}^{x_2x_3} & -m_{x_1x_4}^{ x_3^2} & -m_{x_1x_3}^{x_2x_3} & m_{x_1x_4}^{x_2x_3} & 0 \\
        -m_{x_1x_2}^{x_3x_4} & m_{x_1x_3}^{x_3x_4}+m_{x_1x_2}^{x_2x_4} & -m_{x_1x_4}^{x_3x_4} & -m_{x_1x_3}^{x_2x_4} & m_{x_1x_4}^{x_2x_4} & 0 \\
        -m_{x_1x_2}^{ x_3x_4} & m_{x_1x_3}^{x_3x_4} & -m_{x_1x_4}^{x_3x_4}+m_{x_1x_2}^{ x_2x_3} & 0 & -m_{x_1x_3}^{x_2x_3} & m_{x_1x_4}^{x_2x_3} \\
        -m_{x_1x_2}^{x_4^2} & m_{x_1x_3}^{x_4^2} & -m_{x_1x_4}^{x_4^2}+m_{x_1x_2}^{x_2x_4} & 0 & -m_{x_1x_3}^{x_2x_4} & m_{x_1x_4}^{x_2x_4} \\
        0 & -m_{x_1x_2}^{ x_3x_4} & m_{x_1x_2}^{x_3^2} & m_{x_1x_3}^{x_3x_4} & -m_{x_1x_4}^{x_3x_4}-m_{x_1x_3}^{x_3^2} & m_{x_1x_4}^{x_3^2}\\
        0 & -m_{x_1x_2}^{x_4^2} & m_{x_1x_2}^{x_3x_4} & m_{x_1x_3}^{x_4^2} & -m_{x_1x_4}^{x_4^2}-m_{x_1x_3}^{x_3x_4} & m_{x_1x_4}^{ x_3x_4}\\
    \end{bmatrix}.
    \]
\end{example}

Now specify the submatrices
\begin{equation}\label{eq:N_matrices}
    \begin{aligned}
        \mathbf{N}_{1, 1}&:=\bb{{\mathbf{A}^1}'}_{E_{1, 2}', Y'_3},\\
    \mathbf{N}_{1, 2}&:=\bb{{\mathbf{A}^1}'}_{E_{1, 1}', Y'_5\cup Y'_4},\\
    \mathbf{N}_2&:=\bb{{\mathbf{A}^1}'}_{E'_{1, 2}, Y'_5\cup Y'_4},\\
    \mathbf{N}& :=\mathbf{N}_{1, 1}\mathbf{N}_{1, 2}-m\mathbf{N}_2\in \mathbb{C}^{|E_{1,2}'|\times |Y'_5\cup Y'_4|}.
    \end{aligned}
\end{equation}

We prove equivalences between the ranks of $\mathbf{A}^1, {\mathbf{A}^1}', \mathbf{N}$; the proof can be found in the appendix.
\begin{lemma}\label{lem:reductions}
    For $m=\operatorname{det}(\Z_B)\neq 0$ the following are equivalent:
    \begin{enumerate}
        \item $\mathbf{A}^1$ is full column rank,
        \item ${\mathbf{A}^1}'$ is full column rank,
        \item $\mathbf{N}$ is full column rank.
    \end{enumerate}
\end{lemma}

We now directly show that $(n, r), r\leq 2n+1$ are efficient formats. We start with a lemma counting $Y, E_1$.
\begin{lemma}\label{lem:c1_counts}
    Let $\phi\in S^4\bbC^{n+1}$ satisfy Assumption \ref{ass:concise}, Condition \ref{cond:low_rank} and Condition \ref{cond:first_r}, and let $\operatorname{rank}(\phi)=2n+1$.  Then
    \[|Y|=n+\binom{n}{2}+\binom{n+1}{3}=\binom{n+2}{3}\]
    and 
    \[|E_1|=n\binom{n}{2}.\]
    In particular, when $n\geq 4$, $|E_1|\geq |Y|$.
\end{lemma}
\begin{proof}
    Plug in $c=1$ into Theorem \ref{thm:counts}.  We then have $|E_1|-|Y|=(n-1)\binom{n}{2}-n-\binom{n+1}{3}$, which is
    \[
        \frac{n^3}{3}-n^2-\frac{n}{3},
    \]
    which has a root between $3$ and $4$.  As the leading coefficient is positive, for $n\geq 4$ $|E_1|-|Y|$ is positive.
\end{proof}

Our next lemma addresses the cases $n=2, 3$ -- not covered by the previous lemma; proof is in the appendix.  The proof of $n=2$ proceeds by showing that the vanishing of the maximal minors $\mathbf{A}$ is explicitly a nonzero polynomial in the entries of $\mathbf{Z}$ while the proof of $n=3$ proceeds by a specialization, i.e., choosing specific points.
\begin{lemma}\label{lem:n23_linear}
    $(2, 4)$ is an efficient format.  Additionally, $(3, 5), (3, 6), (3, 7)$ are efficient formats.
\end{lemma}

The proof of the next lemma also proceeds by specialization, with proof in the appendix.
\begin{lemma}\label{lem:c1_veryefficient}
    For $n\geq 4$, $(n, 2n+1)$ is a very efficient format.
\end{lemma}

We next show that efficient formats satisfy a sort of induction.
\begin{lemma}\label{lem:c1_efficient_reduction}
    Let $r\leq 2n+1$.  If $(n, r)$ is an efficient format then $(n+1, r+1)$ is an efficient format.
\end{lemma}
\begin{proof}
    We can proceed by an induction.  The format $(n, r)$ specifies 
    \[B^{(n, r)}=\{1, x_1, \dots, x_n, x_1^2, \dots, x_1x_{r-(n+1)}\}\]
    and the format $(n+1, r+1)$ specifies
    \[B^{(n+1, r+1)}=\{1, x_1, \dots, x_n, x_{n+1}, x_1^2, \dots, x_1x_{r-(n+1)}\},\]
    as $r+1\leq 2n+2< 2(n+1)+1$.  Similarly, we index the variables and equations for the former format by $Y^{(n, r)}, E^{(n, r)}$, respectively, and likewise for the latter format.  

    We show that $Y^{(n, r)}$ ``embeds'' in $Y^{(n+1, r+1)}$ in a suitable way, and similarly for $E^{(n, r)}$ in $E^{(n+1, r+1)}$.  Consider a variable in $Y^{(n, r)}$, which is indexed by $x^{\alpha+\alpha'}x_i$ for $x^\alpha, x^{\alpha'}\in B^{(n, r)}_2, i\leq n$.  Due to the specific form that $B^{(n+1, r+1)}$ takes, we have that $x^\alpha, x^{\alpha'}\in B^{(n+1, r+1)}_2$ and thus $x^{\alpha+\alpha'}x_i\in Y^{(n+1, r+1)}$.  It is also straightforward to compute $Y^{(n+1, r+1)}\setminus Y^{(n, r)}$.  This is precisely $x^{\alpha+\alpha'}x_{n+1}$ for $x^\alpha, x^{\alpha'}\in B^{(n+1, r+1)}_2=B^{(n, r)}_2$.  

    Consider an equation $(x_1 x_i x_j, x_k x_\ell)\in E^{(n, r)}_1$, so that $x_1x_i\in B^{(n, r)}_2, x_kx_\ell\notin B^{(n, r)}_2, x_kx_j\in B^{(n, r)}_2$.  All of these continue to be satisfied with $B^{(n+1, r+1)}$ as $B^{(n+1, r+1)}_2=B^{(n, r)}_2$.  This holds similarly for an equation $[(x_1x_ix_j, x_k x_\ell), (x_1x_i x_\ell, x_kx_j)]\in E^{(n, r)}_2$.

    To recap, we have
    \begin{align*}
        Y^{(n+1, r+1)}&=Y^{(n, r)}\cup (Y^{(n+1, r+1)}\setminus Y^{(n, r)}),\\
        E^{(n+1, r+1)}&=E^{(n, r)}\cup (E^{(n+1, r+1)}\setminus E^{(n, r)}).
    \end{align*}
    Next, we show that the equations in $E^{(n, r)}$ viewed as equations in $E^{(n+1, r+1)}$ are not supported on $Y^{(n+1, r+1)}\setminus Y^{(n, r)}$.  Indeed, if $(x_1 x_i x_j, x_k x_\ell)\in E^{(n, r)}_1$ we must have $i, j\leq n$.  This equation is supported on $x_1x_ix_jx^\alpha$, $\alpha\in B^{(n+1, r+1)}_2$.  But we know that $\alpha_{n+1}=0$ so $x_1x_ix_jx^\alpha\in Y^{(n, r)}$.  Similarly, an equation in $E^{(n, r)}_2$ is also supported only on $Y^{(n, r)}$.  

    Consider now each variable in $Y^{(n+1, r+1)}\setminus Y^{(n, r)}$.  We know that this is precisely the set $\{x^{\alpha+\alpha'}x_{n+1}\mid x^\alpha, x^{\alpha'}\in B^{(n, r)}_2\}$.  Because of the structure of $B^{(n, r)}_2$ this is equivalently the set of $x_1^2 x_ix_jx_{n+1}$ for $1\leq i\leq j\leq r-(n+1)$.  For each such ``new variable'' we assign an equation.  Let $x^\alpha=x_1x_i$, $x^\beta=x_1$, so that we have the equation $(x_1x_ix_j, x_1x_{n+1})\in E^{(n+1, r+1)}_1$; indeed, the ``swap'' of $x_1$ and $x_{n+1}$ gives $(x_1x_ix_{n+1}, x_1x_j)$, and $x_1x_j\in B^{(n, r)}_2$.  Clearly, the equations $(x_1x_ix_j, x_1x_{n+1})$ are supported on $x_1^2x_ix_jx_{n+1}$ with value $m$.  Moreover, on other variables in $Y^{(n+1, r+1)}\setminus Y^{(n, r)}$ this equation is zero: for $x^\alpha\in B^{(n, r)}_2$ we see that that $x_{n+1}$ cannot appear in $x_1x_ix_jx^{\alpha}$, essentially by assumption.  

    Define
    \[S:=\{(x_1x_ix_j, x_1x_{n+1})\mid 1\leq i\leq j\leq r-(n+1)\}.\]
    Take the set of equations $Y^{(n, r)}\cup S$.  These equations index the submatrix of $A^{(n+1, r+1)}$
    \[ \mathbf{A}^{(n+1, r+1)}_{Y^{(n, r)}\cup S}=
\begin{bNiceArray}{c|c}
   \Block{1-1}{\mathbf{A}^{(n, r)}} & \Block{1-1}{\mathbf{0}} \\ 
  \hline
  \Block{1-1}{*} &  \Block{1-1}{m\mathbf{I}}
\end{bNiceArray}.\]
Thus, $\mathbf{A}^{(n+1, r+1)}$ being full column rank reduces to $\mathbf{A}^{(n, r)}$ being full column rank (viewed as a embedding of the set of equations and variables).  $\mathbf{A}^{(n, r)}$ as a submatrix of $\mathbf{A}^{(n+1, r+1)}$ admits the same representation in terms of the variables $m_{x^\gamma}^{x^\delta}$.  Though, even with the same representation because the format is larger the specific value of $m_{x^\gamma}^{x^\delta}$ may be different, with the same specialization as Lemma \ref{lem:c1_veryefficient} we can see that $\mathbf{A}^{(n, r)}$ embedded as a submatrix of $\mathbf{A}^{(n+1, r+1)}$ is actually \textit{the same} submatrix as $\mathbf{A}^{(n, r)}$ under the format $(n, r)$.

Again, we let $\z_1=\mathbf{e}_0$.  For $k=2, \dots, n+1$ we write $\z_k=\mathbf{e}_0+\mathbf{e}_{k-1}$, and for $k=n+2, \dots r$ we have
\[\z_k=\begin{bmatrix}
    1\\
    -1\\
    \vdots \\
    -1\\
    2\\
    \vdots \\
    2
\end{bmatrix},\]
where the number of $-1$ is $k-(n+1)$ and the number of $2$ is $2n+1-k$.  We assert that 
\[\operatorname{det}\bb{\Z_{B^{(n+1, r+1)}\setminus \{x^\gamma\}\cup \{x^\delta\}}}=\operatorname{det}\bb{\Z_{B^{(n, r)}\setminus \{x^\gamma\}\cup \{x^\delta\}}}\]
if $\delta_{n+1}=0$.  By construction this matrix includes a $1$ in the $n+2^{\text{th}}$ row and $x_{n+1}$ column (which is the $n+2^{\text{th}}$ column), and by the pseudotriangular structure of the matrix the determinant must expand on this entry.  Removing this column and row we the remaining matrix is precisely $\operatorname{det}\bb{Z_{B^{(n, r)}\setminus{x^\gamma}\cup\{x^\delta\}}}$.  By assumption $\mathbf{A}^{(n, r)}$ is full column rank, so is also full column rank as a submatrix of $\mathbf{A}^{(n+1, r+1)}$; thus, $\mathbf{A}^{(n+1, r+1)}_{Y^{(n, r)}\cup S}$ is full column rank.
\end{proof}

We now present our main result of this section, giving a large family of tensors that can be efficiently decomposed via Algorithm \ref{alg:efficient_alg}.
\begin{theorem}\label{thm:c1_efficient}
Let $n$ and $r$ be non-negative integers. Then
    \begin{enumerate}
        \item $(2, r)$ is an efficient format for $r\leq 4$,
        \item $(n, r)$ is an efficient format for $n \geq 3$ and $r\leq 2n+1 $.
    \end{enumerate}
    Therefore, generic tensors of these formats are efficiently decomposable via Algorithm \ref{alg:efficient_alg} and $\mathbf{A}$ full column rank is an effective criterion for specific identifiability.
\end{theorem}
\begin{proof}
    For $(n, r)$ with $r\leq n+1$ this is given by simultaneous diagonalization.  The formats $(2, 4)$ and $(3, 5), (3, 6), (3, 7)$ are covered by Lemma \ref{lem:n23_linear}.  For $n\geq 4$, $(n, 2n+1)$ is covered by Lemma \ref{lem:c1_veryefficient}.  For $n\geq 4$, let $n+1<r<2n+1$.  If $(n-1, r-1)$ is an efficient format then by Lemma \ref{lem:c1_efficient_reduction} $(n, r)$ is an efficient format.  In particular, for fixed $r'\leq n+1$ if $(n, n+r')$ is an efficient format so is $(n+1, n+1+r')$.  We have a family of base cases: $(2, 4), (3, 6), (3, 7)$, and $(n, 2n+1)$ for $n\geq 4$, and by Lemma \ref{lem:c1_efficient_reduction} $(n, r)$ is an efficient format for $n+1<r<2n+1$.
\end{proof}

In the proof of Theorem \ref{thm:c1_efficient} we proceeded first by proving that $(n, 2n+1)$ is a very efficient format for $n\geq 4$, as in this regime the matrix $\mathbf{A}^1$ is tall, and then inducting on this in Lemma \ref{lem:c1_efficient_reduction}.  Moreover, we can easily see that for appropriate $r>2n+1$ the corresponding matrix $\mathbf{A}^1$ is still tall, so we would like to repeat this analysis for $r$ as large as possible.  We formalize this in our conjecture.  
\begin{conjecture}\label{conjecture}
    \begin{enumerate}[(1)]
        \item $\bb{n, \sum_{j=0}^c (n-j+1)}$ is a very efficient format if $c$ is such that $|E_1|\geq |Y|$.
        \item $(n, r)$ is an efficient format if there exists $c$ with $r\leq \sum_{j=0}^c (n-j+1)$ such that $\bb{n, \sum_{j=0}^{c} (n-j+1)}$ is a very efficient format.
    \end{enumerate}
\end{conjecture}
We note that (1) is equivalent to conjecturing that when $\mathbf{A}^1$ is tall it is full column rank.  Note that we are \textit{not} conjecturing that when $\mathbf{A}$ is tall it is full column rank; towards this point we give, in the appendix, an example of a set of linear equations that are always linearly dependent in Example \ref{ex:lin_dep_equations}.  Crucially this set includes equations in $E_2$.  We also give a example in our supplementary code of a format such that for any tensor of such a format the corresponding $\mathbf{A}$ tall but numerically drops rank.

The previous results prove the conjecture specifically for $c=1$; we now describe the implications of the conjecture for efficient tensor decomposition up to $r=O(n^2)$.
\begin{proposition}\label{prop:asymp_counts}
    There exists some fixed constant $t^*\in (0.632, 1)$ such that for all $t\in (0, t^*)$ there exists some $n_t$ such that for all $n\geq n_t$, $\left|E_1^{\bb{n, \sum_{j=0}^{tn}(n-j+1)}}\right|\geq \left|Y^{\bb{n, \sum_{j=0}^{tn}(n-j+1)}}\right|$ .
\end{proposition}
\begin{proof}
    Write $c:=tn$.  Expanding the expressions of Theorem \ref{thm:counts} for $|Y|$ and $|E_1|$ we see that these are both quintics in $n$.  The coefficient of $n^5$ in $|Y|$ is
    \[-\frac{7}{60}t^5+\frac{7}{24}t^4-\frac{1}{4}t^3+\frac{1}{12}t^2\]
    and the coefficient in front of $n^5$ in $|E_1|$ is 
    \[-\frac{1}{4}t^5+\frac{3}{4}t^4-\frac{3}{4}t^3+\frac{1}{4}t^2.\]
    Then the coefficient in front of $n^5$ in $|E_1|-|Y|$ is
    \begin{equation}\label{eq:poly_p}
    -\frac{2}{15}t^5+\frac{11}{24}t^4-\frac{1}{2}t^3+\frac{1}{6}t^2=t^2\bb{-\frac{2}{15}t^3+\frac{11}{24}t^2-\frac{1}{2}t+\frac{1}{6}}.
    \end{equation}
    This polynomial in $t$ has 2 distinct real roots: a double root at $0$ and 
    \[t^*=\frac{1}{48}\bb{55
    -\frac{29\cdot 5^{2/3}}{\bb{421-96\cdot 6^{1/2}}^{1/3}}
    -\bb{5\cdot \bb{421- 96\cdot 6^{1/2}}}^{1/3}
    }\approx 0.63297.\]
    For $t\in (0, t^*)$ the polynomial \eqref{eq:poly_p} is positive.  This implies that for such $t$ the polynomial $|E_1|-|Y|$ is positive for $n$ large enough.  
\end{proof}

\begin{table}
\centering
\begin{tabular}{l|l|l}
\rowcolor[HTML]{FFFFFF} 
$t$    & $n_t$    & $r$                          \\ \hline
\rowcolor[HTML]{FFFFFF} 
$0.4$   & $5$    & $0.32 n^2 + 1.2 n + 1$       \\
\rowcolor[HTML]{FFFFFF} 
$0.5$   & $8$    & $0.375 n^2 + 1.25 n + 1$     \\
\rowcolor[HTML]{FFFFFF} 
$0.6$   & $31$   & $0.42 n^2 + 1.3 n + 1$       \\
\rowcolor[HTML]{FFFFFF} 
$0.63$  & $331$  & $0.43155 n^2 + 1.315 n + 1$  \\
$0.632$ & $1011$ & $0.432288 n^2 + 1.316 n + 1$
\end{tabular}
\caption{$|E_1|-|Y|\geq 0$ for $n\geq n_t$ and asymptotic rank bound}
\label{tab:rank_bounds}
\end{table}

 Table \ref{tab:rank_bounds} gives bounds for some $t<t^*$.  For each $t$, $n_t$ is such that for all $n\geq n_t$, $|E_1|-|Y|\geq 0$.  Column $r$ writes the corresponding rank in terms of $n$ -- note that these are all $O(n^2)$ with improving leading coefficient as $t$ increases.    

\begin{theorem}\label{thm:conj_implication}
    If Conjecture \ref{conjecture} holds, then there exist constants $q, n_{q}$ such that for all $n\geq n_q$, $r\leq qn^2$, $(n, r)$ is an efficient format.  Thus, generic tensors of these formats are efficiently decomposable via Algorithm \ref{alg:efficient_alg} and $\mathbf{A}$ full column rank is an effective criterion for deciding the identifiability of specific tensors.  In particular, we can take $q> 0.432288$.
\end{theorem}

Theorem \ref{thm:conj_implication} matches the asymptotic $O(n^2)$ rate present in the literature.  Moreover, the leading coefficient improves existing work; we can take $q$ to be the leading coefficient of $\sum_{j=0}^{tn}(n-j+1)$ for any $t\in (0, t^*)$, with $t^*$ as in the proof of Proposition \ref{prop:asymp_counts}.  Concretely, looking at Table \ref{tab:rank_bounds} we can take $q=0.432288, n_q=1011$, for example.  In contrast, \cite{deLathauwer2007fourth, johnston2023computing} imply a leading coefficient of $1/6$ via an algorithm for finding rank-$1$ elements in a linear section of a variety, while the recent work of \cite{dastidar2025improvingthresholdfindingrank1} conjectures that this can be improved to $1/\sqrt{6}\approx 0.40825$, which seems optimal for their algorithm according to their numerical experiments.  Our conjecture implies a strict improvement through a different approach.

We give computer assisted proofs for $n=4, \dots, 17$.  For each $r>2n+1$ we choose a prime $p$ and draw $\z_1, \dots, \z_r\in \mathbb{F}_p^{n+1}$; we then construct the corresponding tensor $\phi=\sum_{i=1}^r \z_i^{\otimes 4}$ and construct the matrix $\mathbf{A}$ via the linear determinantal relations (Definition \ref{def:determinal_relations}).  We can then calculate the rank exactly via Gaussian elimination over the finite field; if the matrix $\mathbf{A}$ is full column rank over $\mathbb{F}_p$ then it is full column rank over $\mathbb{C}$.  To perform these calculations we use the \texttt{Julia} computer algebra package \texttt{Nemo.jl} \cite{nemo}; primes $p$ and explicit specializations $\z_1, \dots, \z_r$ for each $n$ are given in our supplementary code.

\begin{table}[ht]
\centering
\begin{tabular}{l|llllllll}
$n$ & 2           & 3           & 4           & 5           & 6           & 7            & 8            & 9            \\
$r_n$ & \textbf{4}  & \textbf{7}  & \textbf{11} & \textbf{15} & \textbf{21} & \textbf{28}  & \textbf{36}  & \textbf{44}  \\ \hline
$n$    & 10          & 11          & 12          & 13          & 14          & 15           & 16           & 17           \\
$r_n$    & \textbf{53} & \textbf{64} & \textbf{75} & \textbf{87} & \textbf{99} & \textbf{113} & \textbf{128} & \textbf{143}
\end{tabular}
\caption{Maximum $r$ for which $\mathbf{A}$ is full column rank, $n=2, \dots, 17$.}
\label{tab:n_r_bounds}
\end{table}

\begin{theorem}
    Let $2\leq n\leq 17$.  Then $(n, r)$ is an efficient format for all $r\leq r_n$, where $r_n$ is given in Table \ref{tab:n_r_bounds}.  Therefore, the generic tensor of such a format is efficiently decomposable via Algorithm \ref{alg:efficient_alg} and $\mathbf{A}$ full column rank is an effective criterion for specific identifiability.
\end{theorem}
\begin{proof}
    The cases $n=2, 3$ and $r\leq 2n+1$ for $n\geq 4$ are covered by Theorem \ref{thm:c1_efficient}.  For $2n+2\leq r\leq r_n$ and $n\geq 4$ we give an explicit prime $p$ and points $\z_1, \dots, \z_r\in \mathbb{F}_p^{n+1}$ (found in our supplementary code) for which the corresponding coefficient matrix $\mathbf{A}$ is full column rank.  As an explicit specialization of points, we confirm for the format $(n, r)$ that $\mathbf{A}$ is generically full column rank, that is, $(n, r)$ is an efficient format.
\end{proof}

Note that these $r_n$ are slightly better than the maximum ranks specified by Conjecture \ref{conjecture}.  To be precise, part (1) of the conjecture specifies a rank bound $r'_n:=\sum_{j=0}^{c_n}(n-j+1)$ where $c_n$ is the largest integer such that the corresponding $\mathbf{A}^1$ is full column rank.  Part (2) states that for all $r\leq r'_n$, $(n, r)$ is an efficient format.  Of course, we must have $r'_n\leq r_n$.  Qualitatively, the gap between the two is not large; see Table \ref{tab:n_r'_bounds}. 

\begin{table}[h]
\centering
\begin{tabular}{l|llllllll}
$n$    & 2           & 3           & 4           & 5           & 6           & 7            & 8            & 9            \\
$c_n$  & \textbf{0}  & \textbf{0}  & \textbf{1}  & \textbf{2}  & \textbf{2}  & \textbf{3}   & \textbf{4}   & \textbf{4}   \\
$r'_n$ & \textbf{3}  & \textbf{4}  & \textbf{9}  & \textbf{15} & \textbf{18} & \textbf{26}  & \textbf{35}  & \textbf{40}  \\ \hline
$n$       & 10          & 11          & 12          & 13          & 14          & 15           & 16           & 17           \\
 $c_n$      & \textbf{5}  & \textbf{5}  & \textbf{6}  & \textbf{7}  & \textbf{7}  & \textbf{8}   & \textbf{9}   & \textbf{9}   \\
 $r_n'$      & \textbf{51} & \textbf{57} & \textbf{70} & \textbf{84} & \textbf{92} & \textbf{108} & \textbf{125} & \textbf{135}
\end{tabular}
\caption{Maximum $c_n, r_n'$ for which $\mathbf{A}^1$ is full column rank, $n=2, \dots, 17$.}
\label{tab:n_r'_bounds}
\end{table}

\subsubsection{Further application: three collinear points}\label{sec:collinear}
One might ask if there is a non-generic rank $r$ tensor $ \phi \in S^4 \bbC^{n+1}$ where $(n, r)$ is an efficient format but $\mathbf{A}$ drops rank.  In this subsection we show that this is indeed the case for a specific class of nonidentifiable tensors -- decompositions with three collinear points. Although these tensors are non-generic, linear relations can still suffice for extension and decomposition.

Fix the format $(n, n+2)$.  We know from Theorem \ref{thm:c1_efficient} that this is an efficient format.  From Lemma \ref{lem:equation_deg} the quadratic relations are of the form
\[\operatorname{det}\bb{\mathcal{H}_{B\cup \{x^\alpha x_i\}, B\cup\{x^\beta x_j\}}}-\operatorname{det}\bb{\mathcal{H}_{B\cup \{x^\alpha x_j\}, B\cup\{x^\beta x_i\}}}=0\]
for $\alpha, \beta\in B_2, i\neq j$.  Now if $\phi$ satisfies Assumption \ref{ass:dehomogenization} and Assumption \ref{ass:concise} and is of this format, then there are no quadratic moment relations: $B_2$ is simply one element and so the quadratic relations are the constant zero polynomials.  Tensors of this format are thus completely specified by linear equations -- including when there are multiple decompositions.  

\begin{condition}\label{cond:collinear}
    $\phi\in S^4\bbC^{n+1}$ is such that $\operatorname{rank}(\phi)=n+2$ and there exists a decomposition $\Z=\{\z_1, \dots, \z_r\}$ where $\z_1, \dots \z_{r-1}$ are generic and $\z_1, \z_2, \z_r$ are collinear.
\end{condition}

A set of such points can be written
\[\Z=\begin{bmatrix}
    \z_1^\top \\
    \z_2^\top \\
    \vdots \\
    \z_{n+1}^\top \\
    t \z_1^\top +(1-t)\z_2^\top
\end{bmatrix}\]
where $t\in (0, 1)$ (because we dehomogenize) and $\z_{i, 0}=1$ for all $i$.  

\begin{proposition}\label{prop:collinear_eqs}
    Let $\Z$ be a set of points satisfying Condition \ref{cond:collinear} and let 
    \[B=\{1, x_1, \dots, x_n, x_1^2\}.\]
    Then for all $i,j,k,\ell \in \{1,\ldots, n\}$, the points $\Z$ satisfies the relations
    \[\operatorname{det}\bb{\Z_{B\setminus \{x_1^2\}\cup \{x_ix_j\} }}\operatorname{det}\bb{\Z_{B\setminus \{x_1^2\}\cup \{x_kx_\ell\} }}-\operatorname{det}\bb{\Z_{B\setminus \{x_1^2\}\cup \{x_ix_k\} }}\operatorname{det}\bb{\Z_{B\setminus \{x_1^2\}\cup \{x_jx_\ell\} }}=0.\]
\end{proposition}
\begin{proof}\belowdisplayskip=-12pt
    We have 
    \begin{align*}
        m_{x_1^2}^{x_ix_j}
        &=\sum_{k=1}^r (-1)^{r+k}\z_{k, i}\z_{k, j}\operatorname{det}\bb{\Z^{\hat{k}}} \\
        &=\sum_{k\in \{1, 2, r\}}(-1)^{r+k}\z_{k, i}\z_{k, j}\operatorname{det}\bb{\Z^{\hat{k}}},\\
    \end{align*}
    as $\operatorname{det}\bb{\Z^{\hat{k}}}=0$ for $k\neq 1, 2, r$.  We can also see $\operatorname{det}\bb{\Z^{\hat{1}}}=t\operatorname{det}\bb{\Z^{\hat{r}}}, \operatorname{det}\bb{\Z^{\hat{2}}}=-(1-t)\operatorname{det}\bb{\Z^{\hat{r}}}$,
    so we can deduce
    \begin{gather*}
    m_{x_1^2}^{x_ix_j}/\operatorname{det}\bb{\Z^{\hat{r}}}=-t(1-t)\bb{\z_{1, i}-\z_{2, i}}\bb{\z_{1, j}-\z_{2, j}},\\
    m_{x_1^2}^{x_kx_\ell}/\operatorname{det}\bb{\Z^{\hat{r}}}=-t(1-t)\bb{\z_{1, k}-\z_{2, k}}\bb{\z_{1, \ell}-\z_{2, \ell}},\\
    m_{x_1^2}^{x_ix_k}/\operatorname{det}\bb{\Z^{\hat{r}}}=-t(1-t)\bb{\z_{1, i}-\z_{2, i}}\bb{\z_{1, k}-\z_{2, k}},\\
    m_{x_1^2}^{x_jx_\ell}/\operatorname{det}\bb{\Z^{\hat{r}}}=-t(1-t)\bb{\z_{1, j}-\z_{2, j}}\bb{\z_{1, \ell}-\z_{2, \ell}}.
\end{gather*}
\qedhere
\end{proof}
This actually holds for $B=\{1, x_1, \dots, x_n, x_ix_j\}$ for any $1 \leq i, j \leq n$, and replacing $x_1^2$ with $x_ix_j$ in the statement and proof.  Using this proposition we give Example \ref{ex:collinear}, in the appendix, of a tensor of format $(n, r) = (2, 4)$ such that $\z_1, \z_2, \z_3$ are generic and $\z_4$ lies on the line passing through $\z_1, \z_2$.  We explicitly demonstrate that though $(2, 4)$ is an efficient format, when there are three collinear points the matrix $\mathbf{A}$ drops rank; we then show that this corresponds to a one-dimensional parameterized space of decompositions.  This example recovers case (4b)(ii) of Theorem 3.7 of \cite{MOURRAIN2020347} via an algorithmic method.

We generalize this in the following theorem, also generalizing some cases from Theorem 4.2 of \cite{MOURRAIN2020347} to larger $n$.
\begin{theorem}
    Under Condition \ref{cond:collinear} the tensor $\phi$ has a one-dimensional space of decompositions: the points $\z_3, \dots, \z_{n+1}$ and $s_1\z_1+(1-s_1)\z_2, s_2\z_1+(1-s_2)\z_2, s_3\z_1+(1-s_3)\z_2$, where $s_1, s_2, s_3$ are the three distinct solutions to a cubic with coefficients parameterized by a single moment variable.  Each decomposition can be found efficiently via the parameterization.
\end{theorem}
\begin{proof}
    With a change of variables we can always write $\Z=\{\z_1, \dots, \z_{n+2}\}$ such that $\z_{1}=\mathbf{e}_0$, $\z_k=\mathbf{e}_0+\mathbf{e}_{k-1}$ for $k=2, \dots, n+1$, and $\z_{n+2}=\mathbf{e}_0+(1-t)\mathbf{e}_1$.  It can be easily seen that $\operatorname{det}\bb{\Z_B}=-t(1-t)$, and for all $x^\alpha\neq x_1^2$ that $\operatorname{det}\bb{\Z_{B\setminus \{x_1^2\}\cup \{x^\alpha\} }}=0$.  In particular, the column $y_{x_1^5}$ of $\mathbf{A}$ is zero.  Each other column $y_{x_1^4x_j}, j\geq 2$ contains $\operatorname{det}\bb{\Z_B}$ in exactly one entry -- the equation $(x_1^3, x_1x_j)\in E_1$.  We can treat $y_{x_1^5}$ as a free parameter, and as there are no quadratic relations, the dimension of the space of decompositions is one dimensional.

    From the same analysis as in Example \ref{ex:collinear}, a generic choice for $y_{x_1^5}$ results in multiplication matrices that are nondefective, with eigenvectors $\z_3, \dots, \z_{n+1}$ fixed regardless of choice of parameter.  The remaining three eigenvectors are of the form $s_i\z_1+(1-s_i)\z_2, i=1, 2, 3$ where $s_i$ are the three solutions to a cubic in the choice of $y_{x_1^5}$.
\end{proof}

\begin{remark}
    For this specific format, with $r=n+2$ and $\operatorname{det}\bb{\Z_B}\neq 0$ for $B=\{1, x_1, \dots, x_n, x_ix_j\}$, three points collinear seems to be a necessary condition for $\mathbf{A}$ to drop rank.  
    For example, if $\z_{n+2}$ is a nontrivial linear combination of $\z_1, \z_2, \z_3$, so we have four points coplanar, $\mathbf{A}$ is full column rank (verified by a specialization).
    
    We conjecture that when 
    Condition \ref{cond:first_r} and Condition \ref{cond:collinear} are satisfied, the relations
    \[\operatorname{det}\bb{\Z_{B\setminus \{x_1^2\}\cup \{x_ix_j\} }}\operatorname{det}\bb{\Z_{B\setminus \{x_1^2\}\cup \{x_kx_\ell\} }}-\operatorname{det}\bb{\Z_{B\setminus \{x_1^2\}\cup \{x_ix_k\} }}\operatorname{det}\bb{\Z_{B\setminus \{x_1^2\}\cup \{x_jx_\ell\} }}=0\]
    generate the radical of the vanishing of the $|Y|\times |Y|$ minors of $\mathbf{A}$. 
\end{remark}

\begin{remark}
    When $r>n+2$ there are nontrivial quadratic determinantal relations.  However, similar behavior as above appears when $(n, r)$ is still an efficient format, $\z_1, \dots, \z_{r-1}$ are generic, and $\z_r$ is collinear with $\z_1, \z_2$.  We observe experimentally that $\mathbf{A}$ drops rank by one. Fixing one variable and writing all other variables as a linear expression in the fixed variable, we can substitute these expressions into the quadratic relations to produce a quadratic system in one variable, which is guaranteed to have a solution.  However, instead of this quadratic system specifying a finite number of solutions for this fixed variable we observe that these quadratic relations become the constant zero -- again indicating that linear equations suffice and the space of decompositions is one-dimensional.  An example is included in our attached code.  This merits further consideration and we leave it open for future work.
\end{remark} 

\subsection{Higher order tensors}
Our analysis is not restricted to order-$4$ tensors.  Indeed, under Condition \ref{cond:low_rank} we still can obtain a significant number of linear equations.  However, counting the number of equations and variables remains difficult unless we impose something similar to Condition \ref{cond:first_r}, which for larger $d$ becomes unrealistic and too strong.

If we do choose to proceed with a similar condition, counting equations and variables becomes possible.  We expect for $r=O(n^{d/2})$ the matrix $\mathbf{A}$ will be tall, and similar to our conjecture we also expect this matrix to be generically full column rank.  This condition is one on the expected growth of the Hilbert function $h_{\phi}$, i.e., ranks of catalecticant.

On the other hand, in our supplementary code we give an explicit example of an order-$6$ tensor that does not achieve the expected growth of the Hilbert function yet decomposition still reduces to a linear system.  This certainly motivates future study.

\section{Monomials}\label{sec:monomials}
In this section we discuss applying Algorithm \ref{alg:size_s_decomp} on symmetric tensors that are \textit{monomials}.  It holds that monomials given in a canonical form satisfy Assumption \ref{ass:dehomogenization}, which allows us to determine algorithmically the rank of monomials and provide a parameterization of the space of decompositions efficiently.

\begin{condition}\label{cond:monomial}
    $\phi$ is a monomial, i.e., $\phi_{x_0^{d_0}\dots x_n^{d_n}}=1$ and all other entries are zero, with $d_0\leq d_1\leq \dots \leq d_n$.  $\phi$ is an order-$d:=\sum_{j=0}^n d_j$ symmetric tensor.
\end{condition}

The following is known:
\begin{proposition}[\cite{BUCZYNSKA201345waring}, Corollary 19]\label{prop:monomial_dehomogenization}
    For $\phi$ satisfying Condition \ref{cond:monomial}, if $\z_1, \dots, \z_r$ gives a minimal rank decomposition of $\phi$ then $\z_{i, 0}\neq 0$ for all $i$.
\end{proposition}

Proposition \ref{prop:monomial_dehomogenization} ensures that Assumption \ref{ass:dehomogenization} holds for decompositions of monomials.  Form the Hankel matrix $\mathcal{H}$ corresponding to $\phi$, dehomogenizing with respect to $x_0$; the first step is to find a monomial set $B$ such that $\mathcal{H}_{B, B}$ is invertible.  Consider
\begin{equation}\label{eq:monomial_B}
    B:=\{x^\alpha\mid 0\leq \alpha_i\leq d_i, 1\leq i\leq n\}.
\end{equation}
Then $|B|=\prod_{j=1}^n (d_j+1)$ and
\[\mathcal{H}_{B, B}=\begin{bmatrix}
    0 & &  &  & 1\\
      & &  & 1 &   \\
      & & \iddots & & \\
    & 1 & & &     \\
    1 & &  & & *
\end{bmatrix},\]
that is, the antidiagonal of $\mathcal{H}_{B, B}$ is all ones and above it lies all zeros.  To see this, write $\overline{d}:=(d_1, \dots, d_n)$.  An antidiagonal element of $\mathcal{H}_{B, B}$ is precisely indexed by $\bb{x^\alpha, x^{\overline{d}-\alpha}}$ and by Condition \ref{cond:monomial}, $\phi_{x^\alpha x^{\overline{d}-\alpha}}=\phi_{x^{\overline{d}}}=1$.  Above this antidiagonal are indices of degree less than $d$, so must be zero.  Below the antidiagonal there can be moment variables in some of the entries.

Using Proposition \ref{prop:monomial_dehomogenization} and the above observation, we can prove a lower bound on the rank of $\phi$.
\begin{proposition}\label{prop:monnomial_r_lower}
    $\operatorname{rank}(\phi)\geq \prod_{j=0}^n (d_j+1)$.
\end{proposition}
\begin{proof}
     Proposition \ref{prop:monomial_dehomogenization} and  Theorem \ref{thm:main} show that all decompositions of $\phi$ can be realized via extensions using the Hankel matrix $\mathcal{H}$.  Therefore, numerical ranks of submatrices of $\mathcal{H}$ lower bound the rank of $\phi$.  Given
    \[B=\{x^\alpha\mid 0\leq \alpha_i\leq d_i, 1\leq i\leq n\}\]
    and realizing $\mathcal{H}_{B, B}$ is lower antidiagonal with $1$'s on the antidiagonal, we have $\operatorname{rank}(\phi)\geq \operatorname{rank}(\mathcal{H}_{B, B})= \prod_{j=1}^n (d_j+1)$ regardless of what happens below the antidiagonal.
\end{proof}
Henceforth, fix the $B$ given in Equation \ref{eq:monomial_B}.  As a brief road map to our discussion on monomials, the next two subsections we describe the moment variables and determinantal equations and reveal a certain grading in the equations.  We then show how certain moment variables act as free parameters, reducing solving the determinantal equations to \textit{polynomial evaluation}, implying an efficient algorithm for decomposition of monomials.

\subsection{Variables}
We describe the variables present in the determinantal relations.
\begin{equation}\label{eq:monomial_vars}
\begin{aligned}
    Y&=\{x^\alpha x^{\alpha'}x_i\mid x^\alpha, x^{\alpha'}\in B, 1\leq i\leq n\},\\
    Y_k&:=\{x^\gamma\in Y\mid (k-1)(d_0+1)+1\leq |\gamma|-d\leq k(d_0+1)\}.
\end{aligned}
\end{equation}
$Y_k$ stratifies $Y$. Indeed let $m$ be the smallest integer such that $|\gamma|-d\leq m(d_0+1)$ for all $x^\gamma\in Y$.  Then $Y=\bigcup_{k=1}^m Y_k$.  Furthermore, define
\begin{equation}Y_k^\ell:=\{x^\gamma\in Y_k\mid \exists T\subseteq [n], |T|=\ell \text{ s.t. } \forall i\in T, \gamma_i>d_i, \forall j\notin T, \gamma_j\leq d_j\}.\end{equation}
In other words, $Y_k^\ell$ is the set of $x^\gamma\in Y_k$ such that $\gamma_i>d_i$ on exactly $\ell$ coordinates $i$.  A useful fact is that if $x^\gamma \in Y$ then $\gamma_i$ can be equal to $2d_i+1$ on at most one coordinate $i$.

\subsection{Equations}\label{sec:monomialeqs}
We now describe the determinantal relations.  Recall that equations are of the form
\[\operatorname{det}\bb{\mathcal{H}_{B\cup \{x^\alpha x_i\}, B\cup\{x^\beta x_j\}}}-\operatorname{det}\bb{\mathcal{H}_{B\cup \{x^\alpha x_j\}, B\cup\{x^\beta x_i\}}}=0\]
for $x^\alpha, x^\beta\in B, i\neq j$.  We abuse some notation from the previous section and define
\begin{equation*}
    \begin{aligned}
        E_1&:=\{(x^\alpha x_i, x^\beta x_j)\mid x^\alpha\in B, x^\beta\in B, i\neq j, x^\alpha x_i\notin B\wedge x^\beta x_j\notin B, x^\alpha x_j\in B\vee  x^\beta x_i\in B\},\\
        E_2&:=\{((x^\alpha x_i, x^\beta x_j), (x^\alpha x_j, x^\beta x_i))\mid x^\alpha\in B, x^\beta\in B, i\neq j, x^\alpha x_i, x^\alpha x_j, x^\beta x_i, x^\beta x_j\notin B\},\\
        E&:=E_1 \cup E_2.
    \end{aligned}
\end{equation*}
As before, we consider $E_1,E_2$ and $E$ as sets of equations defined by the given indices.  However, in contrast with Section \ref{sec:lowrankeven}, these equations are not necessarily linear.

For $\alpha, \beta, i, j$ such that the corresponding equation is in $E_2$ we need $x^\alpha\in B, x^\beta\in B, i\neq j, x^\alpha x_i, x^\alpha x_j, x^\beta x_i, x^\beta x_j\notin B$.  In particular, this means that $\alpha_i=d_i, \alpha_j=d_j, \beta_i=d_i, \beta_j=d_j$ and so, $x^{\alpha+\beta}x_ix_j\notin Y$.  By contrapositive every $x^\gamma\in Y$ can be realized as $x^\gamma=x^{\alpha+\beta}x_ix_j$ with $(x^\alpha x_i, x^\beta x_j)\in E_1$.  Henceforth it suffices to focus on $E_1$.

It will be useful to define the sets $S_k\subseteq B\times B$ for $k\geq 0$:
\begin{equation} \label{eq:upper-n-lower}
S_k:=\{(x^\gamma, x^{\gamma'})\in B\times B\mid k(d_0+1)\leq |\gamma'|-|\gamma|\leq (k+1)(d_0+1)-1\}.
\end{equation}
Implicitly we are assuming $|\gamma'|\geq |\gamma|$.  We can show that $(x^\gamma, x^{\gamma'})\in S_k\implies x^{\overline{d}-\gamma+\gamma'}\in Y_k$:
\begin{align*}
    |\overline{d}-\gamma+\gamma'|
    &=|\gamma'|+|\overline{d}-\gamma|\\
    &=|\gamma'|+d-d_0-|\gamma|,
\end{align*}
the second equality being justified since $x^\gamma$ divides $x^{\overline{d}}$.  Then 
\begin{align*}
    |\gamma'|+d-d_0-|\gamma|
    & \leq d-d_0+(k+1)(d_0+1)-1\\
    &=d+k(d_0+1).
\end{align*}
The lower bound in Equation \ref{eq:upper-n-lower} holds similarly.  Technically we need to check that $x^{\overline{d}-\gamma+\gamma'}\in Y$ but this is clear.

The next two lemmas we present give recursive expressions for $\mathcal{G}:=\mathcal{H}_{B, B}^{-1}$ and  the determinantal equation $\operatorname{det}\mathcal{H}_{B\cup\{x^{\alpha}x_i\}, B\cup\{x^\beta x_j\}}$.  In particular, we show that these expressions possess a certain graded structure.  The proofs are in Appendix \ref{app:b}.

\begin{lemma}\label{lem:monomial_G}
    Let $(x^\gamma, x^{\gamma'})\in S_k$.  Then 
    \[
        \mathcal{G}_{x^\gamma, x^{\overline{d}-\gamma'}}=-y_{x^{\overline{d}-\gamma+\gamma'}}-\sum_{\ell=1}^{k-1}\sum_{(x^\delta, x^{\gamma'})\in S_\ell} \mathcal{G}_{x^\gamma, x^{\overline{d}-\delta}}y_{x^{\overline{d}-\delta+\gamma'}}.
    \]
    In particular, $\operatorname{deg}\bb{\mathcal{G}_{x^\gamma, x^{\overline{d}-\gamma'}}}=k$  and the only moment variable in $Y_k$ is $y_{x^{\overline{d}-\gamma+\gamma'}}$; all other moment variables present are in $\bigcup{\ell<k} Y_\ell$.
\end{lemma}

\begin{lemma}\label{lem:monomial_eq}
    Let $x^{\alpha+\beta}x_ix_j\in Y_k$, with $i\neq j$.  Then $\operatorname{det}\mathcal{H}_{B\cup\{x^{\alpha}x_i\}, B\cup\{x^\beta x_j\}}$ equals
    \begin{align*}
        -y_{x^{\alpha+\beta}x_i x_j}+\sum_{d-|\alpha|\leq |\gamma|\leq |\beta|-d_0}y_{x^{\alpha+\gamma}x_i}y_{x^{\overline{d}-\gamma+\beta} x_j}+\sum_{\ell=1}^{k-2}\sum_{\substack{|\gamma|\geq d-|\alpha|,\\ |\gamma'|\leq |\beta|-d_0,\\(x^\gamma, x^{\gamma'})\in S_\ell}}y_{x^{\alpha+\gamma}x_i}\mathcal{G}_{x^\gamma, x^{\overline{d}-\gamma'}}y_{x^{\overline{d}-\gamma'+\beta} x_j}.
    \end{align*}
    In particular, $\operatorname{deg}\bb{\operatorname{det}\mathcal{H}_{B\cup\{x^{\alpha}x_i\}, B\cup\{x^\beta x_j\}}}=k$ and the only moment variable in equation $Y_k$ is $y_{x^{\alpha+\beta}x_i x_j}$; all other variables are in $\bigcup{\ell<k} Y_\ell$.
\end{lemma}

Importantly, note that there may be distinct equations corresponding to different $\alpha, \beta, i, j$ and $\alpha', \beta', i', j'$ even though $x^{\alpha+\beta}x_ix_j=x^{\alpha'+\beta'}x_{i'}x_{j'}$.

\begin{example}
    Let $\phi=x_0x_1x_2^2$.  Then
    \[B=\{1, x_1, x_2, x_1x_2, x_2^2, x_1x_2^2\}.\]
    The variables are
    \[Y=\{y_{x_1^3x_2^2}, y_{x_1^2x_2^3}, y_{x_1x_2^4}, y_{x_2^5}, y_{x_1^3x_2^3}, y_{x_1^2x_2^4}, y_{x_1x_2^5}, y_{x_1^3x_2^4}, y_{x_1^2x_2^5}\}\]
    and $E_1$
    \begin{align*}
        &y_{x_1^2x_2^3}=0, y_{x_1^3x_2^3}=0, y_{x_1^2x_2^4}=0,\\
        &y_{x_1^3x_2^4}-y_{x_1^3x_2^2}y_{x_1x_2^4}-y_{x_1^2x_2^3}^2=0,\\
        &y_{x_1^2x_2^5}-y_{x_1^3x_2^2}y_{x_2^5}-y_{x_1^2x_2^3}y_{x_1x_2^4}=0.
    \end{align*}
    $E_2$ is empty.  If we substitute we end up with the two equations
    \[y_{x_1^3x_2^4}=y_{x_1^3x_2^2}y_{x_1x_2^4}, \quad y_{x_1^2x_2^5}=y_{x_1^3x_2^2}y_{x_2^5}.\]
    As a lead-in to the next section, we see that we can treat $y_{x_1^3x_2^2}, y_{x_1x_2^4}, y_{x_2^5}, y_{x_1^2x_2^5}$ as free parameters.
    
\end{example}

\subsection{Parameters}
Now that we have described the variables and equations in the determinantal equations, and have shown that these equations possess a graded structure. We draw attention to a particular subset of variables that we claim serve as ``parameters for extension.''

\begin{definition}\label{def:monomial_parameters}
    Define the set
    \[Y_P:=\bigcup_{k} Y_k^1\]
    to be the parameters for extension.
\end{definition}

The important characteristic of the set $Y_P$ is that variables in this set do not appear as a linear term in any equation.  To be precise, consider equations in $E_1$, we require $i\neq j$, and for $x^\alpha x_i, x^\beta x_j\notin B$ we need $\alpha_i=d_i, \beta_j=d_j$.  But then $x^\gamma=x^{\alpha+\beta}x_i x_j$ is such that $\gamma_i\geq d_i, \gamma_j\geq d_j$, so $x^\gamma\notin Y_P$.  On the other hand, for $x^\gamma\in Y\setminus Y_P$ we can always find $\alpha, \beta, i\neq j$ such that $x^\gamma=x^{\alpha+\beta}x_ix_j$ with corresponding equation in $E_1$.  

\begin{proposition}\label{prop:determined_by_param}
    In the determinantal relations the value of a variable $x^\gamma\in Y_k\setminus Y_{k}^1$ is determined by the values of $\cup_{\ell<k} Y_\ell^1$.
\end{proposition}
\begin{proof}
    This is a straightforward induction on $k$.  We know from Lemma \ref{lem:monomial_eq} that a variable $x^\gamma$ with $\gamma\in Y_{k}$ is determined by variables in $Y_\ell, \ell<k$ as it appears linearly in the corresponding equations (which may be distinct), whereas there are no other variables outside $\cup_{\ell=1}^{k-1}Y_{\ell}^1$ except for $x^{\gamma}$ in these equations.  But a variable $x^\delta\in Y_\ell$ with $\ell<k$ is either in $Y_\ell^1$ or is determined by $\cup_{\ell'=1}^{\ell-1}Y_{\ell'}^1$ by induction.  We simply need to prove the base case $Y_1\setminus Y_1^1$.  From  Equation \ref{eq:monomial_vars}, if $x^{\gamma}\in Y_1$ then $|\gamma|\leq d+d_0+1$.  If there exists $\alpha, \beta, i\neq j$ such that $x^\gamma=x^{\alpha+\beta}x_ix_j$, then $|\alpha|+|\beta|\leq d+d_0-1$.  Then from  Lemma \ref{lem:monomial_eq} we see that both of the summations are empty, giving $y_{x^\gamma}=0$.  
\end{proof}

\subsubsection{An explicit decomposition via a specific choice of values for \texorpdfstring{$Y_P$}{YP}}\label{sec:canonical_decomposition_monomial}
We demonstrate a choice of values for $Y_P$ that corresponds to a decomposition of $\phi$, thereby verifying that $r:=\operatorname{rank}(\phi)=\prod_{j=1}^n (d_j+1)$, reproving the result of Section 2 of \cite{BUCZYNSKA201345waring}.  Consider the size-$n$ subset of $Y_P$ given by 
\[\{x^\gamma\mid \forall i=1, \dots, n, \gamma_i=2d_i+1, \gamma_j=d_j, j\neq i\}.\]
We assign the value $1$ to these $n$ parameters and the value $0$ to all other parameters in $Y_P$.  If there is a decomposition of size $r$, then by Proposition \ref{prop:determined_by_param} there is a solution to the determinantal relations, where the values the other moment variables take are uniquely determined by the choice of $Y_P$.  We claim that this unique determination sets all other moment variables to take value $0$. It is possible to show this via a direct analysis of the determinantal relations.  Equivalently we construct the corresponding multiplication matrices $\mathbf{M}_i^B(Y^*)$, with $Y^*$ our claimed solution, and show that these commute and are nondefective; by Lemma \ref{lem:determinant_equals_commuting} these are the same.

No variable $x^\gamma$ with some $i$ such that $\gamma_i=2d_i+1$ can appear in $\mathcal{H}_{B, B}$. Thus, $\mathcal{H}_{B, B}$ has ones on the antidiagonal and zeros everywhere else.  Therefore, $\mathcal{H}_{B, B}^{-1}=\mathcal{H}_{B, B}$.  We now consider $\mathcal{H}_{B, B_i}$.  An entry $(x^\alpha, x^{\beta}), x^\alpha, x^\beta\in B$ of this matrix is $1$ if $\alpha+\beta=(d_1, \dots, d_{i-1}, d_i-1, d_{i+1}, \dots, d_n)$ or if $\alpha+\beta=(d_1, \dots, d_{i-1}, 2d_i, d_{i+1}, \dots, d_n)$, and is $0$ otherwise.  $\mathcal{H}_{B, B}^{-1}$ is order reversing (i.e., the permutation matrix corresponding to the cycle $(r\,\, r-1\,\,\dots \, \,2\,\,1)$), so we have
\[
\bb{\mathbf{M}_i^B(Y^*)}_{x^\alpha, x^\beta x_i}=\begin{cases}
    1,\quad \text{if $\beta-\alpha=\mathbf{e}_i$},\\
    1,\quad \text{if $\alpha-\beta=d_i\mathbf{e}_i$},\\
    0,\quad \text{otherwise}.
\end{cases}
\]
Consider row $x^\alpha$. If $\alpha_i=d_i$ then there is a $1$ where $x^\beta$ is such that $\beta_i=0, \beta_j=\alpha_j, j\neq i$.  If $\alpha_i<d_i$, then there is a $1$ where $x^\beta$ is such that $\beta_i=\alpha_i+1, \beta_j=\alpha_j, j\neq i$.  This means that each row has exactly one nonzero entry.  

Let $\mathbf{t}$ be the vector
\[\mathbf{t}=\begin{bmatrix}
    1\\
    \zeta_1\\
    \vdots \\
    \zeta_n
\end{bmatrix},\]
where $\zeta_i$ is some primitive $d_i+1$ root of unity.  The $x^\beta$ entry of $\mathbf{t}^B$ is $\zeta_1^{\beta_1}\dots \zeta_n^{\beta_n}$.  Then $\bb{\mathbf{M}_i^B(Y^*)}_{x^\alpha}^\top \mathbf{t}^B$ has value $\zeta_1^{\alpha_1}\dots \zeta_i^{\alpha_i+1}\dots \zeta_n^{\alpha_n}$ if $\alpha_i<d_i$ and value $\zeta_1^{\alpha_i} \dots \zeta_{i-1}^{\alpha_{d-1}} \zeta_{i+1}^{\alpha_{i+1}}\dots \zeta_n^{d_n} $.  This is precisely $\zeta_i\mathbf{t}^B$, i.e., $\mathbf{t}^B$ is an eigenvector of $\bb{\mathbf{M}_i^B(Y^*)}_{x^\alpha, x^\beta x_i}$ with corresponding eigenvalue $\zeta_i$ for all $i$.  Letting $\mathbf{t}$ range over all permutations of the primitive roots of unity we see that the $\bb{\mathbf{M}_i^B(Y^*)}_{x^\alpha, x^\beta x_i}$ share $r$ linearly independent eigenvectors.  By  Theorem \ref{thm:main} this is a decomposition of $\phi$, and so $\operatorname{rank}(\phi)\leq r$.  With Proposition \ref{prop:monnomial_r_lower}, $\operatorname{rank}(\phi)=r$.

\subsubsection{The space of decompositions} 
It is immediate by Proposition \ref{prop:determined_by_param} that $|Y_P|$ upper bounds the dimension of the space of minimal rank decompositions of $\phi$; if there are no extra relations imposed among variables in $Y_P$ then these are equal.  Moreover, if we can treat $Y_P$ truly as free parameters then $\phi$ can be decomposed efficiently since we can choose generic values for $Y_P$ and by Proposition \ref{prop:determined_by_param}, decomposition reduces to polynomial evaluation.  We first need a result on the dimension of the space of decompositions of $\phi$.

\begin{proposition}\label{prop:dim_VSP_monomial}[\cite{BUCZYNSKA201345waring}, Theorem 2]
    Let $\phi$ satisfy Condition \ref{cond:monomial} and consider the ideal $I:=(a_1^{d_1+1}, \dots, a_n^{d_n+1})$ in $\bbC[a_0, \dots, a_n]$ with associated Hilbert function $h_I$.  Then the dimension of the space of decompositions of $\phi$ is $\sum_{j=1}^n h_I(d_j-d_0)$.
\end{proposition}

\begin{theorem}\label{thm:vsp}
    Let $\phi$ satisfy Condition \ref{cond:monomial}.  The set of associated parameters for extension $Y_P$ is such that $|Y_P|$ is equal to the dimension of the space of decompositions of $\phi$.
\end{theorem}
\begin{proof}
    Define 
    \[Y_P^j=\{x^\gamma\in Y_P\mid \gamma_j>d_j\}.\]
    By definition, if $x^\gamma\in Y_P^j$ then for all $k\neq j, \gamma_k\leq d_k$.  $Y_P$ is the disjoint union of these $Y_P^j$.  We show that $|Y_P^j|=h_I(d_j-d_0)$ where $I$ is the ideal defined in Proposition \ref{prop:dim_VSP_monomial}.  

    We know that $h_I(d_j-d_0)=\operatorname{dim}\bb{\bbC[a_0, \dots, a_n]/I}_{d_j-d_0}$.  The quotient space \\$\bb{\bbC[a_0, \dots, a_n]/I}_{d_j-d_0}$
    is generated by the monomials of degree $d_j-d_0$ not divided by any $a_j^{d_j+1}$.  Equivalently, these are the monomials of degree at most $d_j-d_0$ in $\bbC[a_1, \dots, a_n]$ not divided by $a_j^{d_j+1}$.  We construct a one-to-one mapping between these monomials and elements of $Y_P^j$.

    Explicitly,
    \[Y_P^j=\{x^\gamma\mid d_j+d_0+1\leq \gamma_j\leq 2d_j+1, 0\leq \gamma_k\leq d_k, k\neq j\}.\] 
    We can grade this set by defining
    \[Y_P^{j, k}=\{x^\gamma\in Y_P^{j}\mid \gamma_j=d_j+d_0+1+k\}.\]
    Then $Y_P^j=\cup_{k=0}^{d_j-d_0} Y_P^{j, k}$.  Because of the condition that $|\gamma|\geq d+1$, we must have for all $x^\gamma\in Y_P^{j, k}$ that $\sum_{i\leq j} d_i-\gamma_i\leq k$.  As such, an element of $Y_P^{j, k}$ can be written 
    \[x_1^{d_1-\ell_1}\dots x_j^{d_j+d_0+1+k}\dots x_n^{d_n-\ell_n}\]
    with $\ell_1+\dots+ \ell_{j-1}+\ell_{j+1}+\dots +\ell_n\leq k$ and $\ell_i\leq d_i$ as required.  We define the mapping
    \[x_1^{d_1-\ell_1}\dots x_j^{d_j+d_0+1+k}\dots x_n^{d_n-\ell_n}\mapsto a_1^{\ell_1}\dots a_j^{k-\sum_{i\neq j}\ell_i}\dots a_n^{\ell_n}.\]
    We claim that 
    \[a_1^{\ell_1}\dots a_j^{k-\sum_{i\neq j}\ell_i}\dots a_n^{\ell_n}\in \bb{\bbC[a_1, \dots, a_n]/I}_{k}.\]
    It is clear that this monomial is degree $k$.  Moreover, as $\ell_i\leq d_i$ for $i\neq j$, $a_i^{d_i+1}$ cannot divide it.  Moreover, as $k\leq d_j-d_0$ and $0\leq \sum_{i\leq j}\ell_i\leq k$, $0\leq k-\sum_{i\leq j}\ell_i\leq k\leq d_j-d_0$.  $a_j^{d_j+1}$ cannot divide this monomial because $d_0\geq 1$.  Therefore, this claim is true.

    It is obvious this map is injective.  We further claim it is onto $\bb{\bbC[a_1, \dots, a_n]/I}_{k}$ for $k\leq d_j-d_0$.  An element of this space is $a_1^{m_1}\dots a_j^{m_j}\dots a_n^{m_n}$ for $\sum_i m_i=k$ and $m_i\leq d_i$ for all $i$.  Take $\ell_i=d_i-m_i$ for $i\neq j$.  Then the resulting $x_1^{d_1-\ell_1}\dots x_j^{d_j+d_0+1+k}\dots x_n^{d_n-\ell_n}$ satisfies $\sum_{i\neq j}\ell_i\leq k $ and $d_j+d_0+1+k\leq 2d_j+1$ as $k\leq d_j-d_0$.  

    Then $|Y_P^{j, k}|=\operatorname{dim}\bb{\bbC[a_1, \dots, a_n]/I}_{k}$, and so 
    \[|Y_P^{j}|=\operatorname{dim}\bb{\bbC[a_1, \dots, a_n]/I}_{\leq d_j-d_0}=\operatorname{dim}\bb{\bbC[a_0, a_1, \dots, a_n]/I}_{d_j-d_0}.\]
    Then $Y_P=\sum_{j=1}^n h_I(d_j-d_0)$ and we are done by Proposition \ref{prop:dim_VSP_monomial}.
\end{proof}

\begin{remark}
    For the discussion in this section to be fully self-contained we would need standalone proofs for Proposition \ref{prop:monomial_dehomogenization} and Proposition \ref{prop:dim_VSP_monomial}. Towards the latter, we present the following idea. 

    Let $x^\gamma\in Y_P$.  In the full set of determinantal relations there may be different $\alpha, \beta, i\neq j$ and $\alpha', \beta', i'\neq j'$ such that $x^\gamma=x^{\alpha+\beta}x_ix_j=x^{\alpha'+\beta'}x_{i'}x_{j'}$; these are in fact different polynomials in $Y$.  However, by Proposition \ref{prop:determined_by_param} we can rewrite $x^\gamma$ in terms of only variables in $Y_P$.  We simply need to show that all equations corresponding to $x^\gamma$ become equivalent when written in terms of $Y_P$ -- in other words, an equation $(x^\alpha x_i, x^\beta x_j)\in E_1$ has a description in terms of only $x^{\alpha+\beta}x_ix_j$.  We then need to check that given these values for $Y$, $E_2$ is satisfied.  This is combinatorial and we leave it to future study; computations with small monomials indicate that this works.
\end{remark}

To obtain decompositions of monomials, we take the valid $\mathcal{B}_r=\{B\}$, with $B$ as in Equation \ref{eq:monomial_B}.  We then choose generic values for $Y_P$ (so that the multiplication matrices are not defective) and directly solve for the other moment variables via the grading described.  This allows us to produce \textit{any} decomposition of $\phi$ efficiently via Algorithm \ref{alg:monomial}.

\begin{algorithm}[!ht]
\hspace*{\algorithmicindent} \textbf{Input: }$\phi=x_0^{d_0}x_1^{d_1}\dots x_n^{d_n}, d_0\leq d_1\leq \dots \leq d_n$\\
\hspace*{\algorithmicindent} \textbf{Output: }A minimal rank decomposition of $\phi$
\begin{algorithmic}[1]
\STATE Form $\mathcal{H}_{B, B}$ with $B$ as in Equation \ref{eq:monomial_B}
\STATE Choose values for $Y_P$ as defined in Definition \ref{def:monomial_parameters}
\FOR{$k=1, \dots, m$, where $m$ is the smallest integer s.t. $|\gamma|-d\leq m(d_0+1), \forall x^\gamma\in Y$} 
\STATE Solve for $Y_k$ via the linear equations described in Lemma \ref{lem:monomial_eq}
\ENDFOR
\STATE  Obtain a decomposition $\z_1, \dots, \z_s$ via an eigendecomposition and solve a linear system for $\lambda_1, \dots, \lambda_s$

\end{algorithmic}
\caption{Monomial decomposition}
\label{alg:monomial}
\end{algorithm}

\begin{lemma}\label{lem:monomial_choice}
    Let $\phi$ satisfy Condition \ref{cond:monomial}.  Generic values of $Y_P$ correspond to decompositions of $\phi$.
\end{lemma}
\begin{proof}
    Let $Y_P$ be the space of free parameters and $\mathcal{Z} \subset \bbC^{r \times (n+1)}$ the space of decompositions of $\phi$ in the sense that $\Z \in \mathcal{Z}$ if $\phi = \sum^r_{i=1} \z_i^{\otimes d}$.  Both $\bbC^{Y_P}$ and $\mathcal{Z}$  are varieties, and we know they have the same dimension by Theorem \ref{thm:vsp}; let this dimension be denoted $N$.  
    Let $\psi : \bbC^{Y_P} \rightarrow (\bbC^{r \times r})^{\times n}$ denote the polynomial map from free parameters to tuples of claimed multiplication matrices, where we fixed an order of equations to make $\psi$ a definite map.  Let $\xi: \mathcal{Z} \rightarrow (\bbC^{r \times r})^{\times n}$ denote the polynomial map sending a decomposition $\Z \in \mathcal{Z}$ to the corresponding tuple of multiplication matrices (i.e., with the right eigenvectors and eigenvalues as specified by $\Z$).  Let $\mathcal{V} \subset \bbC^{Y_P}  \times \mathcal{Z}$ be the incidence variety defined by $\mathcal{V} = \{(Y^*, Z) \in \bbC^{Y_P} \times \mathcal{Z} : \psi(Y^*) = \xi(\Z)\}$. 
    Consider the projections  onto the first and second factors, $\pi_1: \mathcal{V} \rightarrow \bbC^{Y_P}$ and $\pi_2 : \mathcal{V} \rightarrow \mathcal{Z}$.  Then $\pi_2$ is surjective, because every decomposition of $\phi$ has a corresponding choice of free parameters.  This implies $\operatorname{dim}(\mathcal{V}) \geq N$.  On the other hand, notice that the generic fiber of $\pi_1$ is finite, since for every choice of free parameters mapping to commuting and nondefective multiplication matrices the decomposition is then uniquely determined (up to permutation of factors).  This implies $\dim(\mathcal{V}) \leq N$.  
    We conclude that $\operatorname{dim}(\mathcal{V})=N$ and that $\pi_1$ is dominant.  Hence, there exists a nonempty Zariski open subset of $\bbC^{Y_P}$ contained in $\pi_1(\mathcal{V})$ by Chevalley’s theorem, and the lemma follows.
\end{proof}

\begin{theorem}
    A parameterization of all minimal decompositions of a monomial $\phi$ can be produced efficiently via Algorithm \ref{alg:monomial}.
\end{theorem}
\begin{proof}
    By Proposition \ref{prop:monnomial_r_lower}, $\mathcal{H}_{B, B}$ is invertible.  By Theorem \ref{thm:vsp}, $Y_P$ is precisely the set of parameters for the extension problem.  Choosing values for these, we can then solve for the remaining variables in $Y_k$ for increasing $Y_k$ efficiently by Lemma \ref{lem:monomial_eq} and Proposition \ref{prop:determined_by_param} via polynomial evaluation.  By Lemma \ref{lem:monomial_choice}, if these values are generic we obtain a decomposition of $\phi$ and all minimal decompositions arise in this way.
\end{proof}

We now briefly discuss the action of the $n$-dimensional torus on the space of decompositions (see \cite{BUCZYNSKA201345waring}) that acts via scaling the variables, i.e., if $\mathbf{z}_1, \dots, \mathbf{z}_r$ is a decomposition for $\phi$ with coefficients $\lambda_1, \dots, \lambda_r$ then $\mathbf{z}'_1, \dots, \mathbf{z}'_r$ is also a decomposition for $\mathbf{z}'_{i, j}=\tau_j \mathbf{z}_{i, j}$, $i=1, \dots, r$ for some nonzero scalars $\tau_1, \dots, \tau_n$; the new coefficients are $\lambda_i':=\lambda_i/\boldsymbol{\tau}^{\overline{d}}$.  It is not hard to see that for two choices of parameter values $Y_{P, 1}^*$ and $Y_{P, 2}^*$, the corresponding decompositions are related via scaling the variables if and only if there exist nonzero scalars $\tau_1, \dots, \tau_n$ such that $\boldsymbol{\tau}^\gamma y_{x^\gamma, 1}^* -\boldsymbol{\tau}^{\overline{d}}y_{x^\gamma, 2}^*=0$ for all $x^\gamma \in Y_P$.  When $d_0, d_1, \dots, d_n$ are not all the same this action cannot be transitive as $|Y_P|>n$.  On the other hand, for $d_0=d_1=\dots=d_n$, $|Y_P|=n$ is of the form
\[Y_P=\{x^\gamma\mid \exists i, \gamma_i=2d_i+1, \forall j\neq i, \gamma_j=d_j\}.\]
Then for fixed $x^\gamma$ and two values $y_{x^\gamma, 1}^*, y_{x^\gamma, 2}^*$ both nonzero, with $\gamma_i=2d_i+1$, take $\tau_i^{d_i+1}=y_{x^\gamma, 2}^*/y_{x^\gamma, 1}^*$.  It is straightforward that these $\tau_1, \dots, \tau_n$ satisfy $\boldsymbol{\tau}^\gamma y_{x^\gamma, 1}^* -\boldsymbol{\tau}^{\overline{d}}y_{x^\gamma, 2}^*=0$ for all $x^\gamma \in Y_P$, that is, the action of the torus is transitive.  This is an algorithmic proof of Theorem 4 of \cite{BUCZYNSKA201345waring}.

\begin{remark}
    The discussion in this section crucially relies on $\phi$ being given in a canonical form.  However, in a different basis the resulting tensor may not be easily decomposed. Determining if a tensor is a monomial is decidable and can be determined using Gr\"obner bases and symbolic algebraic methods, however, these methods are a priori not efficient. This begs the question: suppose an arbitrary tensor is given; is there an efficient method that can distinguish if the tensor is a monomial or not, and if it is a monomial, provide the change of basis to convert it to its canonical form?  If so, then Condition \ref{cond:monomial} can be relaxed and our discussion applies to a monomial in any presentation. 
\end{remark}

\section{Discussion and Future Work}
We refined and performed theoretical analysis of the tensor decomposition algorithm of \cite{brachat2010symmetric}, connecting its complexity to the regularity of a target decomposition. We showed that for even-order tensors with decompositions of low enough regularity, moment matrix extension and therefore tensor decomposition can reduce to solving a linear system.  For order-$4$ tensors we proved that this linear system suffices for decomposition of generic tensors of rank at most $2n+1$, implying that the extension algorithm is efficient and gives an effective criterion for identifiability of specific tensors.  We formulated a conjecture that the linear system suffices for rank up to $O(n^2)$, with improved leading coefficient compared to existing literature, and gave a computer-assisted proof verifying the correctness of the conjecture up to $n=17$.  We also demonstrated that certain classes of nonidentifiable tensors, including monomials, can be decomposed efficiently.  To our knowledge, handling nonidentifiability is a first in work on efficient tensor decomposition. In these cases, we algorithmically reproved existing \nolinebreak theory \nolinebreak as \nolinebreak well.

There are various opportunities for future work.  Of course, Conjecture \ref{conjecture} remains open. Relatedly, we could further understand the algebra of the vanishing of the maximal minors of the coefficient matrix $\mathbf{A}$ and implications for identifiability.   
More broadly, we could ask how far we can push the rank bounds for efficient decomposition of generic tensors; in terms of our meta-algorithm, this can include using both new core algorithms (e.g., new normal form methods) and new extension algorithms.  One idea for the extension algorithms is to use equations arising from Young flattenings (rather than the Hankel matrix) 
as defining equations. 
One could use \cite{kothari2024overcompletetensordecompositionkoszulyoung}, which develops a core algorithm based on these flattenings.

\section*{Acknowledgments}
The authors are grateful to Bernard Mourrain, Enrica Barrilli, Tait Weicht, and Luke Oeding for helpful discussions at different stages of this research. 
BS is supported in part by HDR TRIPODS Phase II grant 2217058.  
JL is supported in part by NSF DMS 2318837.  
JK is supported in part by NSF DMS 2309782, NSF DMS 2436499, NSF CISE-IIS 2312746, and DE SC0025312, as well as a J. Tinsley Oden Faculty Fellowship from UT Austin.  

\printbibliography

\newpage
\begin{appendices}
\label{appendix}

    \section{Additional Proofs and Examples for Section \ref{sec:lowrankeven}}\label{app:a}
    \subsection{Section \ref{sec:simplifications}}
    \begin{proof}[Proof of Proposition \ref{prop:counts_inductions}]
    We have
    \begin{align*}
        B^{(n, r)}&=\{1, x_1, \dots, x_n, x_1^2, \dots, x_1x_n, \dots, x_c^2, \dots x_cx_s\},\\
        B^{(n+1, r+c)}&=\{1, x_1, \dots, x_n, x_{n+1}, x_1^2, \dots, x_1x_n, x_1x_{n+1}\dots, x_c^2, \dots x_cx_s,\}
    \end{align*}
    for $c\leq s\leq n$, and $|B^{(n, r)}|=r$.  $Y^{(n, r)}$ and $E^{(n, r)}$ ``embed'' inside $Y^{(n+1, r+c)}$ and $E^{(n+1, r+c)}$, respectively, in a natural way.  Indeed, if $x^{\alpha+\alpha'}x_i\in Y^{(n, r)}$ then $x^\alpha, x^{\alpha'}\in B^{(n+1, r+c)}$, and so $x^{\alpha+\alpha'}x_i\in Y^{(n+1, r+c)}$.  Similarly, if $x^\beta x_i\notin B^{(n, r)}$ with $|\beta|=1$ then $x^\beta x_i\notin B^{(n+1, r+c)}$.  This is enough to ensure that $E^{(n, r)}\subset E^{(n+1, r+c)}$ via an embedding.

    We now describe $Y^{(n+1, r+c)}\setminus Y^{(n, r)}$.  It is clear that all such variables must involve $x_{n+1}$.  The three cases are if the variable involves $x_{n+1}^3, x_{n+1}^2$, or $x_{n+1}$ just once.  We demonstrate that we can choose one new equation for each of these variables.
    \begin{enumerate}[(1)]
        \item If $x^\gamma=x_ix_jx_{n+1}^3$:  In this case we know that $i, j\leq c-1$, as $x_cx_{n+1}^2, x_{n+1}^2\not\in B^{(n+1, r+c)}$.  Choose $x^\alpha=x_jx_{n+1}, x^\beta=x_{n+1}$.  Then as $x_ix_{n+1}\in B^{(n+1, r+c)}$ we have $(x_ix_jx_{n+1}, x_{n+1}^2)\in E_1^{(n+1, r+c)}$.
        \item If $x^\gamma=x_ix_jx_kx_{n+1}^2$:  As $x^\gamma=x^{\delta+\delta'}x_a$ either $x_a=x_{n+1}$ or $\delta'_{n+1}\neq 1$; we must have $\delta_{n+1}\neq 0$ regardless.  In the first case we must have $x_ix_{n+1}, x_jx_k\in B$ so we can take $x^\alpha=x_jx_k, x^\beta=x_{n+1}$, and $(x_ix_jx_k, x_{n+1}^2)\in E_1^{(n+1, r+c)}$, and in the second case as $j\leq c-1$ we again have $x_jx_k\in B^{(n+1, r+c)}$ and we can take the same equation.
        \item There are some subcases.
        \begin{enumerate}
            \item Suppose $\ell\geq c$ without loss of generality.  If $x^\gamma=x^{\delta+\delta'}x_{n+1}$ with $x^\delta, x^{\delta'}\in B_{(n+1, r+c)}$ then we know $x_ix_j, x_kx_\ell\in B^{(n+1, r+c)}$, so take $x^\alpha=x_ix_j, x^\beta=x_\ell$, with the equation $(x_ix_jx_k, x_\ell x_{n+1})$.  On the other hand, if $\delta_{n+1}\neq 0$, without loss of generality, we can see that we always have the equation $(x_ix_jx_k, x_\ell x_{n+1})$ as well. 
            \item If $i, j, k, \ell\leq c-1$ we have that any $2$-subset of these and $n+1$ is in $B^{(n+1, r+c)}$
        \end{enumerate}
    \end{enumerate}
    From this we see that for all ``new variables'' except $x_ix_jx_kx_\ell x_{n+1}$ with $i, j, k, \ell\leq c-1$ there is a corresponding equation in $E_1^{(n+1, r+c)}$ such that $x^{\alpha+\beta}x_ix_j$ is the new variable.  Thus, we choose extra equations for these variables $x_ix_jx_kx_\ell x_{n+1}$.  For each of these variables assign the equation $(x_ix_jx_{n+1}, x_c x_{n+1})$, which arises from choosing $x^\alpha=x_jx_{n+1}, x^\beta=x_{n+1}$.  These equations were not chosen in case (2) above.  Then as $x_kx_\ell\in B^{(n+1, r+c)}$, the variable $x_ix_jx_kx_\ell x_{n+1}$ appears in this equation.          
    \end{proof}
    \subsection{Section \ref{sec:conjecture}}

\begin{proof}[Proof of Lemma \ref{lem:reductions}]
     First, take $x^\gamma=x_ix_jx_k^3$ for $i, j\leq c, k>c$.  It is clear that $y_{x^\gamma}$ only appears in one equation in $E_1$: the equation $(x_ix_jx_k, x_k^2)$.  This equation will always be linearly independent from all other equations in $E_1$, so it suffices to remove all equations of the form $(x_ix_jx_k, x_k^2)$ and variables of the form $x_ix_jx_k^3$ from $Y$.  Now consider variables of the form $x^\delta x_k^2x_\ell$ for $|\delta|=2, \delta_{\mathcal{C}}=2, k\neq \ell$.  These variables only appear in two equations: $(x^\delta x_k, x_kx_\ell)$ and $(x^\delta x_\ell, x_k^2)$.  It suffices to replace one of these equations by the difference of the two, and the resulting two equations will be linearly independent if and only if the original two are.  We can thus take the equation $[(x^\delta x_k, x_kx_\ell), (x^\delta x_\ell, x_k^2)]$ and remove the variable $x^\delta x_k^2x_\ell$.  Similarly, for a variable of the form $x^\delta x_j x_k x_\ell$ for for $|\delta|=2, \delta_{\mathcal{C}}=2$, $j, k, \ell$ all distinct and $j, k, \ell>c$.  By a similar argument, this variable only appears in three equations: $(x^\delta x_j, x_kx_\ell)$, $(x^\delta x_k, x_jx_\ell)$, and $(x^\delta x_\ell, x_jx_k)$.  There are three differences we can take for these equations, but any one of them is dependent on the other two.  Taking any two of them, however, with one of the original equations means that the resulting equations will be linearly independent if the original three are.  We then take the two new equations $[(x^\delta x_j, x_kx_\ell), (x^\delta x_k, x_jx_\ell)]$ and $[(x^\delta x_j, x_kx_\ell), (x^\delta x_\ell, x_jx_k)]$ and remove the variable $x^\delta x_j x_k x_\ell$.  This process exactly constructs $Y'$ and $E_1'$, so that ${\mathbf{A}^1}$ is full column rank if and only if ${\mathbf{A}^1}'$ \nolinebreak is.  

     Up to a change in ordering of the rows and columns of ${\mathbf{A}^1}'$ we can write
     \[{\mathbf{A}^1}'=\begin{bNiceArray}{c|c}
   \Block{1-1}{\mathbf{N}_{1, 2}} & \Block{1-1}{m\mathbf{I}} \\ 
  \hline
  \Block{1-1}{\mathbf{N}_2} &  \Block{1-1}{\mathbf{N}_{1, 1}}
\end{bNiceArray}.\]
Then it is straightforward that ${A^1}'$ is full column rank if and only if $\mathbf{N}=\mathbf{N}_{1, 1}\mathbf{N}_{1, 2}-m\mathbf{N}_2$ is full column rank.
\end{proof}

\begin{proof}[Proof of Lemma \ref{lem:n23_linear}]
    First let $n=2$.  Let $\phi\in S^4\bbC^3$ satisfy Assumptions \ref{ass:dehomogenization}, \ref{ass:concise} and Conditions \ref{cond:low_rank}, \ref{cond:first_r}.
    \begin{enumerate}
        \item  Let $r=4$ with some decomposition $\Z$.  Then 
        \[B=\{1, x_1, x_2, x_1^2\}.\]
        Then
        \[Y=\{x_1^5, x_1^4x_2\}\]
        and 
        \[E_1=\{(x_1^3, x_1x_2)\}, E_2=\{[(x_1^3, x_2^2), (x_1^2x_2, x_1x_2)]\}.\]
        Therefore,
        \[\mathbf{A}=\begin{bmatrix}
            -m_{x_1^2}^{x_1x_2} & m \\ 
            -m_{x_1^2}^{x_2^2} & m_{x_1^2}^{x_1x_2}
        \end{bmatrix}.\]
        If $m\cdot m_{x_1^2}^{x_2^2}-\bb{m_{x_1^2}^{x_1x_2}}^2\neq 0$ as a polynomial in the entries of $\Z$ then $A$ is generically full column rank.  Indeed, as all of the $m_\gamma^\delta$ represent minors of
        \[\Z_2=\begin{bmatrix}
            1 & \z_{1, 1} & \z_{1, 2} & \z_{1, 1}^2 & \z_{1, 1}\z_{1, 2} & \z_{1, 2}^2\\
            1 & \z_{2, 1} & \z_{2, 2} & \z_{2, 1}^2 & \z_{2, 1}\z_{2, 2} & \z_{2, 2}^2\\
            1 & \z_{3, 1} & \z_{3, 2} & \z_{3, 1}^2 & \z_{3, 1}\z_{3, 2} & \z_{3, 2}^2\\
            1 & \z_{4, 1} & \z_{4, 2} & \z_{4, 1}^2 & \z_{4, 1}\z_{4, 2} & \z_{4, 2}^2
        \end{bmatrix},\]
        we can expand both $m\cdot m_{x_1^2}^{x_2^2}$ and $\bb{m_{x_1^2}^{x_1x_2}}^2$ as polynomials in the $\z_{i, j}$ and find monomials that appear in one but not the other.  For example, one such monomial that appears in $m\cdot m_{x_1^2}^{x_2^2}$ but not $\bb{m_{x_1^2}^{x_1x_2}}^2$ (up to sign) is 
        \[\z_{1, 1}^3\z_{2, 2}^3\z_{3, 1}\z_{4, 2}.\]
        \item Let $r=5$ with some decomposition $\Z$.  Then
        \[B=\{1, x_1, x_2, x_1^2, x_1x_2\}.\]
        Then
        \[Y=\{x_1^5, x_1^4x_2, x_1^3x_2^2, x_1^2x_2^3\}\]
        and
        \[E_1=\{(x_1^3, x_2^2), (x_1^2x_2, x_2^2)\}, E_2=\varnothing.\]
        There are four moment variables but only two linear equations, so linear equations do not suffice.
    \end{enumerate}    
    For $n=3$, it suffices to exhibit sets of points such that the corresponding $\mathbf{A}$ is full column rank.  We give related specializations for $r=5, 6, 7$.
    \begin{enumerate}
        \item Let 
        \[\Z=\begin{bmatrix}
            1 & 0 & 0 & 0 \\
            1 & 1 & 0 & 0 \\
            1 & 0 & 1 & 0 \\ 
            1 & 0 & 0 & 1 \\
            1 & -1 & 2 & 2
        \end{bmatrix},\]
        so
        \[\Z_2=\begin{bmatrix}
            1 & 0 & 0 & 0 & 0 & 0 & 0 & 0 & 0 & 0 \\ 
            1 & 1 & 0 & 0 & 1 & 0 & 0 & 0 & 0 & 0 \\
            1 & 0 & 1 & 0 & 0 & 0 & 0 & 1 & 0 & 0 \\ 
            1 & 0 & 0 & 1 & 0 & 0 & 0 & 0 & 0 & 1 \\
            1 & -1 & 2 & 2 & 1 & -2 & -2 & 4 & 4 & 4
        \end{bmatrix}.\]
        There are three moment variables
        \[Y=\{x_1^5, x_1^4x_2, x_1^4x_3\}.\]
        There are many equations, but we simply need to exhibit three linearly independent equations.  We will take these three equations to be
        \begin{align*}
            & (x_1^3, x_1x_2),\\
            &(x_1^3, x_1x_3),\\
            &[(x_1^3, x_2x_3), (x_1^2x_2, x_1x_3))].
        \end{align*}
        The corresponding submatrix of $\mathbf{A}$ is
        \[\begin{bmatrix}
            2 & 2 & 0\\
            2 & 0 & 2\\
            -4 & -2 & 0
        \end{bmatrix}\]
        from computing the appropriate minors of $\Z_2$.  This is obviously rank three.
        \item Let
            \[\Z=\begin{bmatrix}
            1 & 0 & 0 & 0 \\
            1 & 1 & 0 & 0 \\
            1 & 0 & 1 & 0 \\ 
            1 & 0 & 0 & 1 \\
            1 & -1 & 2 & 2 \\
            1 & -1 & -1 & 2
        \end{bmatrix}.\]
        There are seven variables, and a $7\times 7$ full-rank submatrix of $\mathbf{A}$ is
        \[\begin{bmatrix}
            6 & 0 & 6 & 0 & 0 & 0 & 0 \\
            -6 & 0 & 0 & 6 & 0 & 0 & 0 \\ 
            0 & 12 & 0 & 0 & 6 & 0 & 0 \\
            0 & 6 & 0 & 0 & 6 & 0 & 0 \\
            0 & -6 & 0 & 0 & 0 & 6 & 0 \\
            0 & 0 & 0 & 12 & 0 & 0 & 6\\
            0 & 0 & 0 & 6 & 0 & 0 & 6
        \end{bmatrix}.\]
        \item Let
            \[\Z=\begin{bmatrix}
            1 & 0 & 0 & 0 \\
            1 & 1 & 0 & 0 \\
            1 & 0 & 1 & 0 \\ 
            1 & 0 & 0 & 1 \\
            1 & -1 & 2 & 2 \\
            1 & -1 & -1 & 2 \\
            1 & -1 & -1 & -1
        \end{bmatrix}.\]
        There are ten variables, and a $10\times 10$ full-rank submatrix of $\mathbf{A}$ is
        \[\begin{bmatrix}
             -18  &  0 &   0 & 18  &  0  &  0 &  0  &  0 &  0  & 0 \\
             -18  & 36 & -18 &  0  & 18  &  0 &  0  &  0 &  0  & 0 \\
             -18  &  0 &   0 &  0  &  0  & 18 &  0  &  0 &  0  & 0 \\
               0  &-18 &   0 &  0  &  0  &  0 & 18  &  0 &  0  & 0 \\
               0  &-18 &   0 & 36  &-18  &  0 &  0  & 18 &  0  & 0 \\
               0  &-18 &   0 &  0  &  0  &  0 &  0  &  0 & 18  & 0 \\
               0  &  0 & -18 &  0  &  0  &  0 &  0  & 18 &  0  & 0 \\
               0  &  0 & -18 &  0  & 36  &-18 &  0  &  0 & 18  & 0 \\
               0  &  0 & -18 &  0  &  0  &  0 &  0  &  0 &  0  &18 \\
               0  &  0 &   0 &-18  & 18  &  0 &  0  &-36 & 18  & 0
        \end{bmatrix}.\]
        Note that here $r=2n+1$, but $|E_1|=9<|Y|$.  Indeed, the first nine rows in the given submatrix are indexed by $E_1$; we are still short one equation so we must look for an equation in $E_2$.  In this case the tenth row is given by
        \[[(x_1x_2^2, x_3^2), (x_1x_2x_3, x_2x_3)]\in E_2.\]
        \item For $r\geq 8$ there are simply not enough linear equations by a naive count.\qedhere
    \end{enumerate}
\end{proof}

\begin{proof}[Proof of Lemma \ref{lem:c1_veryefficient}]
        By Lemma \ref{lem:c1_counts}, for $n\geq 4$ $|E_1|\geq |Y|$.  We give a family of specializations.  Construct
    \begin{equation}\label{eq:c1_construction}\z_1=\mathbf{e}_0, \quad \z_k=\mathbf{e}_0+\mathbf{e}_{k-1}, k=2, \dots, n+1, \quad \z_k=\begin{bmatrix}
        1 \\
        -\mathbf{1}_{(k-(n+1))\times 1} \\ 
        \mathbf{2}_{(2n+1-k)\times 1}
    \end{bmatrix}, k=n+2, \dots, 2n+1.\end{equation}
    Also,
    \[B=\{1, x_1, \dots, x_n, x_1^2, \dots, x_1x_n\}.\]

    We want to show $\mathbf{A}^1$ is full column rank.  By Lemma \ref{lem:reductions} this is equivalent to showing ${\mathbf{A}^1}'$ is full column rank.  We will give a set $S\subseteq E_1'$, $|S|=|Y'|$ such that the corresponding square submatrix of ${\mathbf{A}^1}'$ is invertible.  When $c=1$ $Y'$ and $E_{1, 1}', E_{1, 2}'$ have particularly simple forms:
    \begin{align*}
        Y'&=\{x_1^3x_jx_k\mid 1\leq j\leq k\leq n\},\\
        E_{1, 1}'&=\{(x_1^3, x_jx_k)\mid 2\leq j\leq k\leq n\},\\
        E_{1, 2}'&= \{[(x_1^2x_j, x_kx_\ell), (x_1^2x_k, x_jx_\ell)]\mid 2\leq j, k, \ell\leq n; j\leq k, k\neq \ell\}.
    \end{align*}
    Noting that $|Y'_3|=E_{1, 1}'$ we include $E_{1, 1}'$ in the set $S$.  Now as $|Y'_5\cup Y'_4|=n$, we need to select $n$ equations from $E'_{1, 2}$.  We select the following:
    \begin{align*}
        & [(x_1^2x_2, x_2x_3), (x_1^2x_3, x_2^2)],\\
        & [(x_1^2x_2, x_3^2), (x_1^2x_3, x_2x_3)],\\
        & [(x_1^2x_2, x_3x_4), (x_1^2x_3, x_2x_4)],\\
        & [(x_1^2x_2, x_3x_k), (x_1^2x_k, x_2x_3)],\quad k\geq 4,
    \end{align*}
    and call this set $S_2$, so that $S=E'_{1, 1}\cup S_2$.  Write 
    \begin{align*}
        \mathbf{P}_{1, 1}&:=\bb{{\mathbf{A}^1}'}_{S, Y'_3},\\
        \mathbf{P}_{1, 2}&:=\bb{{\mathbf{A}^1}'}_{E'_{1, 1}, Y'_5\cup Y'_4},\\
        \mathbf{P}_2&:=\bb{{\mathbf{A}^1}'}_{S, Y'_5\cup Y'_4},
    \end{align*}
    so that
    \[\bb{{\mathbf{A}^1}'}_S=\begin{bNiceArray}{c|c}
   \Block{1-1}{\mathbf{P}_{1, 2}} & \Block{1-1}{m\mathbf{I}} \\ 
  \hline
  \Block{1-1}{\mathbf{P}_2} &  \Block{1-1}{\mathbf{P}_{1, 1}}
\end{bNiceArray};\]
    these $\mathbf{P}$ are submatrices of the $\mathbf{N}$ defined in Equation \ref{eq:N_matrices}.  Using the formula for the determinant of a $2\times 2$ block matrix we have that $\bb{{\mathbf{A}^1}'}_S$ is invertible if and only if \\$\operatorname{det}\bb{\mathbf{P}_{1, 1}\mathbf{P}_{1, 2}-m\mathbf{P}_2}\neq 0$.  

    \begin{claim}\label{claim:c1_veryefficient}
        Construct $\Z$ as in Equation \ref{eq:c1_construction}.  Then $m:=\det\bb{\Z_B}\neq 0$, and
        \begin{align*}
        & m_{x_1^2}^{x_jx_k}=(-1)^{n+1}m,&& 2\leq j\leq k\leq n,\\
        & m_{x_1x_j}^{x_jx_k}=(-1)^{n+j+1}\cdot 2m, && 2\leq j< k\leq n,\\
        & m_{x_1x_k}^{x_jx_k}=(-1)^{n+k}m,&& 2\leq j< k\leq n,
    \end{align*}
    and all other $m_{x_1x_i}^{x_jx_k}=0$.
    \end{claim}

    Using this claim we compute $\operatorname{det}\bb{\mathbf{P}_{1, 1}\mathbf{P}_{1, 2}-m\mathbf{P}_2}$.  We want to first determine the entries of $\mathbf{P}_{1, 1}\mathbf{P}_{1, 2}$.  Consider a row of $\mathbf{P}_{1, 1}$ corresponding to $[(x_1^2x_j, x_kx_\ell), (x_1^2x_k, x_jx_\ell)]$.  As a difference of the two vectors corresponding to $(x_1^2x_j, x_kx_\ell)$ and $(x_1^2x_k, x_jx_\ell)$, we first fix $(x_1^2x_j, x_kx_\ell)$.  An entry $\gamma\in Y'_3$ of the coefficient vector is nonzero for $x^\gamma=x_1^3x_jx_b, b\geq 2$ with coefficient $(-1)^{r+1+\operatorname{pos}(x^1x_b)}m_{x_1x_b}^{x_kx_\ell}$. 

    Similarly, take a column of $\mathbf{P}_{1, 2}$ corresponding to $x^\gamma\in Y'_5\cup Y'_4$.  In the case $c=1$ such a $x^\gamma$ corresponds to $x^\gamma=x_1^4x_i, i\geq 1$.  Take the dot product of this column with the row $(x_1^2x_j, x_kx_\ell)$.  This is the sum 
    \[\sum_{b=2}^n (-1)^{r+1+\operatorname{pos}(x_1x_b)}(-1)^{r+1+\operatorname{pos}(x_1x_i)}m_{x_1x_b}^{x_kx_\ell}m_{x_1x_i}^{x_jx_b}=\sum_{b=2}^n (-1)^{\operatorname{pos}(x_ix_b)+\operatorname{pos}(x_1x_i)}m_{x_1x_b}^{x_kx_\ell}m_{x_1x_i}^{x_jx_b}.\]
    Taking now the row $(x_1^2x_k, x_jx_\ell)$ dot the product with $\gamma=x_1^4x_i$ is
    \[\sum_{b=2}^n (-1)^{\operatorname{pos}(x_1x_b)+\operatorname{pos}(x_1x_i)}m_{x_1x_b}^{x_jx_\ell}m_{x_1x_i}^{x_kx_b}.\]
    It is straightforward that $\operatorname{pos}(x_1x_i)=1+n+i$.  Then taking the difference of the two products we have
    \[\sum_{b=2}^n (-1)^{b+i}\bb{m_{x_1x_b}^{x_kx_\ell}m_{x_1x_i}^{x_jx_b}-m_{x_1x_b}^{x_jx_\ell}m_{x_1x_i}^{x_kx_b}}.\]
    We compute the entries of $\mathbf{P}_{1, 1}\mathbf{P}_{1, 2}$ corresponding to row $(j, k, \ell)=(2, 3, \ell)$ and column $i$ for $\ell=2, 3, 4, i\geq 1$.  Proceed with casework using Claim \ref{claim:c1_veryefficient}.
    \begin{itemize}
        \item \underline{$\ell=2, i=1$:} 
        \begin{align*}
            \sum_{b=2}^n (-1)^{b+1}\bb{m_{x_1x_b}^{x_2x_3}m_{x_1^2}^{x_2x_b}-m_{x_1x_b}^{x_2^2}m_{x_1^2}^{x_3x_b}}
            &=\sum_{b=2}^n (-1)^{b+1}m_{x_1x_b}^{x_2x_3}m_{x_1^2}^{x_2x_b}\\
            &=\sum_{b=2}^n (-1)^{b+n}m_{x_1x_b}^{x_2x_3}\\
            &=(-1)^{2+n}m_{x_1x_2}^{x_2x_3}+(-1)^{3+n}m_{x_1x_3}^{x_2x_3}\\
            &=(-1)^{2+n}\cdot (-1)^{n+2+1}\cdot 2+(-1)^{3+n}\cdot (-1)^{n+3}\\
            &=-1.
        \end{align*}
        \item \underline{$\ell=2, i=2$:}
        \begin{align*}
            \sum_{b=2}^n (-1)^{b+2}\bb{m_{x_1x_b}^{x_2x_3}m_{x_1x_2}^{x_2x_b}-m_{x_1x_b}^{x_2^2}m_{x_1x_2}^{x_3x_b}}
            &=\sum_{b=2}^n (-1)^{b}m_{x_1x_b}^{x_2x_3}m_{x_1x_2}^{x_2x_b}\\
            &=(-1)^{3}m_{x_1x_3}^{x_2x_3}m_{x_1x_2}^{x_2x_3}\\
            &=(-1)^{3}\cdot (-1)^{n+3}\cdot  (-1)^{n+2+1}\cdot 2\\
            &=-2.
        \end{align*}
        \item \underline{$\ell=2, i=3$:}
        \begin{align*}
            \sum_{b=2}^n (-1)^{b+3}\bb{m_{x_1x_b}^{x_2x_3}m_{x_1x_3}^{x_2x_b}-m_{x_1x_b}^{x_2^2}m_{x_1x_3}^{x_3x_b}}
            &=\sum_{b=2}^n (-1)^{b+1 }m_{x_1x_b}^{x_2x_3}m_{x_1x_3}^{x_2x_b}\\
            &=(-1)^{3+1}m_{x_1x_3}^{x_2x_3}m_{x_1x_3}^{x_2x_3}\\
            &=1.
        \end{align*}
        \item \underline{$\ell=2, i\geq 4$:}
        \[\sum_{b=2}^n (-1)^{b+i}\bb{m_{x_1x_b}^{x_2x_3}m_{x_1x_i}^{x_2x_b}-m_{x_1x_b}^{x_2^2}m_{x_1x_i}^{x_3x_b}}=0,\]
        as we need $b=i$, but then as $i\geq 0$, $m_{x_1x_i}^{x_2x_3}=0$.
        \item \underline{$\ell=3, i=1$:} 
        \begin{align*}
            \sum_{b=2}^n (-1)^{b+1}\bb{m_{x_1x_b}^{x_3^2}m_{x_1^2}^{x_2x_b}-m_{x_1x_b}^{x_2x_3}m_{x_1^2}^{x_3x_b}}
            &=\sum_{b=2}^n (-1)^{b}m_{x_1x_b}^{x_2x_3}m_{x_1^2}^{x_3x_b}\\
            &=(-1)^{2}m_{x_1x_2}^{x_2x_3}m_{x_1^2}^{x_2x_3}+(-1)^{3}m_{x_1x_3}^{x_2x_3}m_{x_1^2}^{x_3^3}\\
            &=(-1)^{n+2+1}\cdot 2 \cdot (-1)^{n+1}-(-1)^{n+3}\cdot (-1)^{n+1}\\
            &=1.
        \end{align*}
        \item \underline{$\ell=3, i=2$:} 
        \begin{align*}
            \sum_{b=2}^n (-1)^{b+2}\bb{m_{x_1x_b}^{x_3^2}m_{x_1x_2}^{x_2x_b}-m_{x_1x_b}^{x_2x_3}m_{x_1x_2}^{x_3x_b}}
            &=\sum_{b=2}^n (-1)^{b+1}m_{x_1x_b}^{x_2x_3}m_{x_1x_2}^{x_3x_b}\\
            &=(-1)^{3}m_{x_1x_2}^{x_2x_3}m_{x_1x_2}^{x_2x_3}\\
            &=-4.
        \end{align*}
        \item \underline{$\ell=3, i=3$:} 
        \begin{align*}
            \sum_{b=2}^n (-1)^{b+3}\bb{m_{x_1x_b}^{x_3^2}m_{x_1x_3}^{x_2x_b}-m_{x_1x_b}^{x_2x_3}m_{x_1x_3}^{x_3x_b}}
            &=\sum_{b=2}^n (-1)^{b}m_{x_1x_b}^{x_2x_3}m_{x_1x_3}^{x_3x_b}\\
            &=(-1)^{2}m_{x_1x_2}^{x_2x_3}m_{x_1x_3}^{x_2x_3}\\
            &=2.
        \end{align*}
        \item \underline{$\ell=3, i\geq 4$:}
        \[\sum_{b=2}^n (-1)^{b+i}\bb{m_{x_1x_b}^{x_3^2}m_{x_1x_i}^{x_2x_b}-m_{x_1x_b}^{x_2x_3}m_{x_1x_i}^{x_3x_b}}=0.\]
        \item \underline{$\ell=4, i=1$:}
        \begin{align*}
            \sum_{b=2}^n (-1)^{b+1}\bb{m_{x_1x_b}^{x_3x_4}m_{x_1^2}^{x_2x_b}-m_{x_1x_b}^{x_2x_4}m_{x_1^2}^{x_3x_b}}
            &=(-1)^{2+1}\bb{-m_{x_1x_2}^{x_2x_4}m_{x_1^2}^{x_2x_3}}\\
            &\quad + (-1)^{3+1}m_{x_1x_3}^{x_3x_4}m_{x_1^2}^{x_2x_3}\\
            &\quad + (-1)^{4+1}\bb{m_{x_1x_4}^{x_3x_4}m_{x_1^2}^{x_2x_4}-m_{x_1x_4}^{x_2x_4}m_{x_1^2}^{x_3x_4}}\\
            &=(-1)^{n+2+1}\cdot 2\cdot (-1)^{n+1}\\
            &\quad + (-1)^{n+3+1}\cdot 2\cdot (-1)^{n+1}\\
            &\quad -\bb{(-1)^{n+4}\cdot (-1)^{n+1}-(-1)^{n+4}\cdot (-1)^{n+1}}\\
            &=0.
        \end{align*}
        \item \underline{$\ell=4, i=2$:} 
        \begin{align*}
            \sum_{b=2}^n (-1)^{b+2}\bb{m_{x_1x_b}^{x_3x_4}m_{x_1x_2}^{x_2x_b}-m_{x_1x_b}^{x_2x_4}m_{x_1x_2}^{x_3x_b}}
            &=(-1)^{2+2}\bb{-m_{x_1x_2}^{x_2x_4}m_{x_1x_2}^{x_2x_3}}\\
            &\quad +(-1)^{3+2}m_{x_1x_3}^{x_3x_4}m_{x_1x_2}^{x_2x_3}\\
            &\quad +(-1)^{4+2}m_{x_1x_4}^{x_3x_4}m_{x_1x_2}^{x_2x_4}\\
            &=-(-1)^{n+2+1}\cdot 2 \cdot (-1)^{n+2+1}\cdot 2\\
            &\quad -(-1)^{n+3+1}\cdot 2 \cdot (-1)^{n+2+1}\cdot 2\\
            &\quad +(-1)^{n+4}\cdot (-1)^{n+2+1}\cdot 2\\
            &=-2.
        \end{align*}
        \item \underline{$\ell=4, i=3$:} 
        \begin{align*}
            \sum_{b=2}^n (-1)^{b+3}\bb{m_{x_1x_b}^{x_3x_4}m_{x_1x_3}^{x_2x_b}-m_{x_1x_b}^{x_2x_4}m_{x_1x_3}^{x_3x_b}}
            &=(-1)^{2+3}\bb{-m_{x_1x_2}^{x_2x_4}m_{x_1x_3}^{x_2x_3}}\\
            &\quad + (-1)^{3+3}m_{x_1x_3}^{x_3x_4}m_{x_1x_3}^{x_2x_3}\\
            &\quad + (-1)^{4+3}\bb{-m_{x_1x_4}^{x_2x_4}m_{x_1x_3}^{x_3x_4}}\\
            &=(-1)^{n+2+1}\cdot 2 \cdot (-1)^{n+3}\\
            &\quad + (-1)^{n+3+1}\cdot 2\cdot (-1)^{n+3}\\
            &\quad + (-1)^{n+4}\cdot (-1)^{n+3+1}\cdot 2\\
            &=2.
        \end{align*}
        \item \underline{$\ell=4, i\geq 4$:}
        \[\sum_{b=2}^n (-1)^{b+i}\bb{m_{x_1x_b}^{x_3x_4}m_{x_1x_i}^{x_2x_b}-m_{x_1x_b}^{x_2x_4}m_{x_1x_i}^{x_3x_b}}=0.\]
        This is because for $i\geq 5$, we need $b=i$, but then  $m_{x_1x_i}^{x_3x_4}, m_{x_1x_i}^{x_2x_4}=0$.  For $i=4$, again we need $b=i=4$, and then the corresponding term is 
        \[(-1)^{b+4}\bb{m_{x_1x_4}^{x_3x_4}m_{x_1x_4}^{x_2x_4}-m_{x_1x_4}^{x_2x_4}m_{x_1x_4}^{x_3x_4}}=0.\]
    \end{itemize}
    Now we let $(j, k, \ell)=(2, k, 3)$ for $k\geq 4$.  The entry of $\mathbf{P}_{1, 1}\mathbf{P}_{1, 2}$ with this row and column $i$ is
    \[\sum_{b=2}^n (-1)^{b+i}\bb{m_{x_1x_b}^{x_3x_k}m_{x_1x_i}^{x_2x_b}-m_{x_1x_b}^{x_2x_3}m_{x_1x_i}^{x_kx_b}}.\]
    There are two cases we consider.
    \begin{itemize}
        \item \underline{$i=k$:}
        \begin{align*}
            \sum_{b=2}^n (-1)^{b+k}\bb{m_{x_1x_b}^{x_3x_k}m_{x_1x_k}^{x_2x_b}-m_{x_1x_b}^{x_2x_3}m_{x_1x_k}^{x_kx_b}}
            &=(-1)^{2+k}\bb{-m_{x_1x_2}^{x_2x_3}m_{x_1x_k}^{x_2x_k}}\\
            &\quad + (-1)^{3+k}\bb{-m_{x_1x_3}^{x_2x_3}m_{x_1x_k}^{x_3x_k}}\\
            &\quad + (-1)^{2k}m_{x_1x_k}^{x_3x_k}m_{x_1x_k}^{x_2x_k}\\
            &=(-1)^{k+1}\cdot (-1)^{n+2+1}\cdot 2\cdot (-1)^{n+k}\\
            &\quad + (-1)^k \cdot (-1)^{n+3}\cdot (-1)^{n+k}\\
            &\quad + (-1)^{2k}\cdot (-1)^{n+k}\cdot (-1)^{n+k}\\
            &=2.
        \end{align*}
        \item \underline{$i\geq 4, i\neq k$:}
        \[\sum_{b=2}^n (-1)^{b+i}\bb{m_{x_1x_b}^{x_3x_k}m_{x_1x_i}^{x_2x_b}-m_{x_1x_b}^{x_2x_3}m_{x_1x_i}^{x_kx_b}}=0.\]
        This is because we need $b=i$ as $i\neq k$.  But then we need $b\in \{3, k\}$ or $b\in \{2, 3\}$, which is not possible. 
    \end{itemize}
    Therefore, 
    \[\frac{1}{m}\mathbf{P}_{1, 1}\mathbf{P}_{1, 2}=
\begin{pNiceArray}{ccc|c}
  -1 & -2 & 1 & \Block{3-1}{\mathbf{0}} \\
  1 & -4 & 2 \\
  0 & -2 & 2 \\ 
  \hline
  \Block{1-3}{*} & & & \Block{1-1}{2\mathbf{I}}
\end{pNiceArray}
.\]
Computing the entries for $M_2$ is more straightforward.  This computation reveals
\[\frac{1}{m}\mathbf{P}_2=
\begin{pNiceArray}{ccc|c}
  0 & -1 & 1 & \Block{3-1}{\mathbf{0}} \\
  0 & -1 & 1 \\
  0 & -1 & 1 \\ 
  \hline
  \Block{1-3}{*} & & & \Block{1-1}{\mathbf{I}}
\end{pNiceArray}
,\]
so that
\[\frac{1}{m}\bb{{\mathbf{A}^1}'}_S =
\begin{pNiceArray}{ccc|c}
  -1 & -1 & 0 & \Block{3-1}{\mathbf{0}} \\
  1 & -3 & 1 \\
  0 & -1 & 1 \\ 
  \hline
  \Block{1-3}{*} & & & \Block{1-1}{\mathbf{I}}
\end{pNiceArray}.\]
Therefore, the determinant of $\frac{1}{m}\bb{{\mathbf{A}^1}'}_S$ is exactly the determinant of the principal $3\times 3$ block, which is easily computed to be $3$.  As such, we have found an invertible maximal submatrix of ${\mathbf{A}^1}'$, meaning that it is full column rank, and therefore so is $\mathbf{A}^1$.
\end{proof}

\begin{proof}[Proof of Claim \ref{claim:c1_veryefficient}]
    First, we can write
    \[\Z_B=\begin{bmatrix}
        \mathbf{1} & \mathbf{G}_1 & \mathbf{G}_2 \\
        \mathbf{1} & \mathbf{G}_3 & \mathbf{G}_4
    \end{bmatrix}\]
    and 
    \[\Z_{B\setminus \{x_1x_i\}\cup \{x_jx_k\}}=\begin{bmatrix}
        \mathbf{1} & \mathbf{G}_1 & \mathbf{G}^{(i, j, k)}_2\\
        \mathbf{1} & \mathbf{G}_3 & \mathbf{G}^{(i, j, k)}_4
    \end{bmatrix},\]
    where 
    \[\mathbf{G}_1=\begin{bmatrix}
        \mathbf{0}_{1\times n}\\
        \mathbf{I}_{n\times n}
    \end{bmatrix}\in\mathbb{C}^{(n+1)\times n}, \quad \mathbf{G}_3=\begin{bmatrix}
        -1 & 2 & \dots & 2\\
        -1 & -1 & \dots & 2\\
        \vdots & & &  \vdots \\
        -1 & -1 & \dots & 2\\
        -1 & -1 & \dots & -1
    \end{bmatrix}\in\bbC^{n\times n}.\]
    We deal with the structure of $\mathbf{G}^{(i, j, k)}_2, \mathbf{G}^{(i, j, k)}_4$ in a case-by-case basis.  We index the columns of these matrices by monomials, and write $\operatorname{col}(x^\alpha)$ to refer to the columns themselves.

    The general proof idea for all of the cases will be as follows: fix the $n\times n$ matrix
    \[\mathbf{G}=\begin{bmatrix}
        1 & -2 & \dots & -2\\
        1 & 1 & \dots & -2\\
        \vdots & & &  \vdots \\
        1 & 1 & \dots & -2\\
        1 & 1 & \dots & 1
    \end{bmatrix}.\]
    Because $\mathbf{G}^{(i, j, k)}_2$ is very sparse, with ones in only at most a few entries, the matrices we are considering are almost block triangular.  For each of these cases we will specify a few column updates so that they become block triangular, and then specify the determinants based on the relationship of the lower-right block with $\mathbf{G}$.

    Begin with $\det\bb{\Z_B}$.  Note that $\mathbf{G}_2$ as only one nonzero entry -- a one -- in the entry corresponding to column $x_1^2$ and row $2$, and $\mathbf{G}_4=\mathbf{G}$.  Taking $\operatorname{col}(x_1^2)\leftarrow \operatorname{col}(x_1^2)-\operatorname{col}(x_1)$ we obtain the matrix
    \[\begin{bmatrix}
        \mathbf{1} & \mathbf{G}_1 & \mathbf{0}_{(n+1)\times n} \\
        \mathbf{1} & \mathbf{G}_3 & \mathbf{G}'
    \end{bmatrix},\]
    where 
    \[\mathbf{G}'=\begin{bmatrix}
        2 & -2 & \dots & -2\\
        2 & 1 & \dots & -2\\
        \vdots & & &  \vdots \\
        2 & 1 & \dots & -2\\
        2 & 1 & \dots & 1
    \end{bmatrix}.\]
    Then $\det\bb{Z_B}=\operatorname{det}\bb{\mathbf{G}'}=2\det\bb{\mathbf{G}}=2\cdot 3^{n-1}$.  

To compute $m_{x_1^2}^{x_k^2}$ for $2 \leq k \leq n$, we first see that $\mathbf{G}^{(1, k, k)}_4$ is the last $n-1$ columns of $\mathbf{G}$ appended with the column $[\mathbf{4}_{1\times (k-1)} \mathbf{1}_{1\times (n-k+1)}]^\top $.  With the same idea as above we replace $\operatorname{col}(x_k^2)\leftarrow \operatorname{col}(x_k^2)-\operatorname{col}(x_k^2)$ to obtain essentially the same matrix $\mathbf{G}'$ except the all $2$ column is cycled to the end.  Then $m_{x_1^2}^{x_k^2}=(-1)^{n+1}\cdot 2\det(G)=(-1)^{n+1}m$, where the sign comes from considering the sign of the cycle $(2\,3\dots n \, 1)$.

To compute $m_{x_1^2}^{x_jx_k}, 2\leq j<k\leq n$ we first see that 
\[ \text{col}(x_j x_k) = \begin{bmatrix}
    \mathbf{0}_{(n+1)\times 1} \\ \mathbf{4}_{(j-1)-1} \\ -\mathbf{2}_{(k-j)\times 1} \\ \mathbf{1}_{(n-k+1)\times 1}
\end{bmatrix}.\]
Denote $\mathbf{v}_{jk}$ denote the last $n$ entries of this vector.  We additionally see that $\mathbf{G}_2^{(1, j, k)}$ is all zero.  Then
\begin{multline*}
    m_{x_1^2}^{x_jx_k}
    =\det \begin{bmatrix}
    \mathbf{G}_{x_1x_2} & \cdots & \mathbf{G}_{x_1x_n} & \mathbf{v}_{jk}
\end{bmatrix}
    =(-1)^{n+1}\det \begin{bmatrix}
   \mathbf{v}_{jk} & \mathbf{G}_{x_1x_2} & \cdots & \mathbf{G}_{x_1x_n}
\end{bmatrix}\\
= \det \bb{\mathbf{G} + (\mathbf{G}_{x_1^2} - \mathbf{v}_{jk})\mathbf{e}_1^\top}
=\det(\mathbf{G})\bb{1 + \mathbf{e}_1^\top \mathbf{G}^{-1}(\mathbf{G}_{x_1^2} - \mathbf{v}_{jk})}.
\end{multline*}
One can directly verify that 
\[ \mathbf{G}^{-1} = \frac{1}{3} \cdot \begin{bmatrix}
    1 & 0 & 0 & \cdots & 0 & 2 \\ 
    -1 & 1 & 0 & \cdots & 0 & 0 \\
    0 & -1 & 1 & \cdots & 0 & 0 \\
    \vdots & & \ddots & & \vdots \\
    0 & 0 & 0 & \cdots & -1 & 1
\end{bmatrix}. \]
Then
\[ \mathbf{e}_1^\top \mathbf{G}^{-1} (\mathbf{G}_{x_1^2} - \mathbf{v}_{jk}) = \frac{1}{3} \begin{bmatrix}
     1 & 0 & \cdots & 0 & 2 
\end{bmatrix} \begin{bmatrix}
   \mathbf{3}_{j-1} \\ -\mathbf{3}_{k-j} \\ \mathbf{0}_{n-k+1}
\end{bmatrix} = 1,\]
so $m_{x_1^2}^{x_jx_k} = (-1)^{n+1}\det(\mathbf{G})(1 + 1) = (-1)^{n+1}\cdot 2 \det(\mathbf{G}) = (-1)^{n+1}m$.

Now we compute $m_{x_1x_i}^{x_j x_k}$ where $i > 1$.  We first notice that $m_{x_1x_i}^{x_j^2}$ from the fact that taking $\operatorname{col}\bb{x_1^2}\leftarrow \operatorname{col}\bb{x_1^2}-\operatorname{col}\bb{x_1}$ and $\operatorname{col}\bb{x_j^2}\leftarrow \operatorname{col}\bb{x_j^2}-\operatorname{col}\bb{x_j}$ leads to a repeated column, so the determinant must be zero.  Now assume $j<k$.  Take $\operatorname{col}\bb{x_1^2}\leftarrow \operatorname{col}\bb{x_1^2}-\operatorname{col}\bb{x_1}$, so we are left with determining the determinant of the matrix
\[\begin{bmatrix}
    2\mathbf{G}_{x_1}^2 & \dots & \mathbf{G}_{x_1x_{i-1}} & \mathbf{G}_{x_1x_{i+1}} & \dots & \mathbf{G}_{x_1x_n} & \mathbf{v}_{jk}
\end{bmatrix}.\]
We instead cycle $\mathbf{v}_{jk}$ into $x_1x_i$ index and first calculate
\[\mathbf{G}'=\begin{bmatrix}
    2\mathbf{G}_{x_1}^2 & \dots & \mathbf{G}_{x_1x_{i-1}} &  \mathbf{v}_{jk}& \mathbf{G}_{x_1x_{i+1}} & \dots & \mathbf{G}_{x_1x_n}
\end{bmatrix},\]
accounting for sign later.  Then
\[\det\bb{\mathbf{G}'}=2\det\bb{\mathbf{G}+(\mathbf{v}_{jk}-\mathbf{G}_{x_1x_i})\mathbf{e}_i^\top}=2\det\bb{\mathbf{G}}\bb{1+\mathbf{e}_i^\top \mathbf{G}^{-1}(\mathbf{v}_{jk}-\mathbf{G}_{x_1x_i})}.\]
Observe that 
\[
\mathbf{e}_i^\top \mathbf{G}^{-1}=\frac{1}{3}\begin{bmatrix}
    \mathbf{0}_{(i-2)\times 1}\\
    -1 \\
    1\\
    \mathbf{0}_{(n-i)\times 1}
\end{bmatrix}^\top.
\]
 Let $u_{i-1}$ and $u_{i}$ be the $(i-1)^{\text{th}}$ and $i^{\text{th}}$ entries of $\mathbf{v}_{jk} - \mathbf{G}_{x_1x_i}$. We then have that $\mathbf{e}_i^\top \mathbf{G}^{-1} (\mathbf{v}_{jk} - \mathbf{G}_{x_1x_i}) = \frac{1}{3}(-u_{i-1} + u_i)$. Observe that the $(i-1)^{\text{th}}$ and $i^{\text{th}}$ entries of $\mathbf{G}_{x_1x_i}$ are $-2$ and $1$ respectively. The possible $(i-1)^{\text{th}}$ and $i^{\text{th}}$ entries of $\mathbf{v}_{jk}$ are $\{(4,4), (4, -2), (-2, -2), (-2, 1), (1,1) \}$. This gives that $(u_{i-1},u_i) \in \{(6,3),(6,-3),(0,-3),(0,0),(3,0) \}$. In particular, $\mathbf{e}_i^\top \mathbf{G}^{-1}(\mathbf{v}_{jk} - \mathbf{G}_{x_1x_i}) \in \{-1,-3,-1,0,-1\}$ so $\det\bb{\mathbf{G}'} \in \{0, -4 \det(\mathbf{G}), 0, 2 \det(\mathbf{G}), 0\}$. This shows that $\operatorname{det}\bb{\mathbf{G}'} = 0$ except in two cases: when $(u_{i-1}, u_i) = (6,-3)$ and when $(u_{i-1}, u_i) = (0,0)$. The former case happens when $j = i$ and the latter when $k = i$.  Therefore, when $i=j$ then  $m_{x_1x_j}^{x_jx_k}=(-1)^{n+j+1}\cdot 2m$ and when $i=k$ then $m_{x_1x_k}^{x_jx_k}=(-1)^{n+k}m$ as we account for sign.
\end{proof}

\begin{example}\label{ex:lin_dep_equations}
    Let $n=3, r=7$.  Let
    \[B=\{1, x_1, x_2, x_3, x^2, x_1x_2, x_1x_3\}.\]
    Let $x^\alpha =x^2, x^\beta=x_2, i=2, j=3$.  This corresponds to 
    \[[(x_1^2x_2, x_2x_3), (x_1^2x_3, x_2^2)]\in E_2.\]
    Let $x^\alpha=x_1x_2, x^\beta =x_3, i=1, j=2$.  This corresponds to 
    \[(x_1^2x_2, x_2x_3)\in E_1,\]
    as $x^\beta x_i\in B_2$.  Similarly, let $x^\alpha=x_1x_3, x^\beta=x_2, i=1, j=2$.  This corresponds to 
    \[(x_1^2x_3, x_2^2)\in E_1.\]
    But the difference of the latter two equations is exactly the first equation.
    
    A more subtle dependency is as follows.  Consider the following six equations in $E_1$:
    \[(x_1^3, x_2^2),(x_1^3, x_3^2),(x_1^2x_2, x_2^2),(x_1^2x_2, x_3^2),(x_1^2x_3, x_2^2), (x_1^2x_3, x_3^2).\]
    It can be seen that the sum of $[(x_1x_2^2, x_3^2), (x_1x_2x_3, x_2x_3)]$ and $[(x_1x_2x_3, x_2x_3), (x_1x_3^2, x_2^2)]$ in $E_2$ is a linear combination of these 6 equations.  The corresponding submatrix of $\mathbf{A}$ has the following form:
\[
\left[
\begin{array}{cccccccccc}
 - m_{ x_1^2}^{ x_2^2} & m_{x_1 x_2}^{ x_2^2} &  - m_{x_1 x_3}^{ x_2^2} & m & 0 & 0 & 0 & 0 & 0 & 0 \\
 - m_{ x_1^2}^{ x_3^2} & m_{x_1 x_2}^{ x_3^2} &  - m_{x_1 x_3}^{ x_3^2} & 0 & 0 & m & 0 & 0 & 0 & 0 \\
0 &  - m_{ x_1^2}^{ x_2^2} & 0 & m_{x_1 x_2}^{ x_2^2} &  - m_{x_1 x_3}^{ x_2^2} & 0 & m & 0 & 0 & 0 \\
0 &  - m_{ x_1^2}^{ x_3^2} & 0 & m_{x_1 x_2}^{ x_3^2} &  - m_{x_1 x_3}^{ x_3^2} & 0 & 0 & 0 & m & 0 \\
0 & 0 &  - m_{ x_1^2}^{ x_2^2} & 0 & m_{x_1 x_2}^{ x_2^2} &  - m_{x_1 x_3}^{ x_2^2} & 0 & m & 0 & 0 \\
0 & 0 &  - m_{ x_1^2}^{ x_3^2}  & 0 & m_{x_1 x_2}^{ x_3^2} &  - m_{x_1 x_3}^{ x_3^2} & 0 & 0 & 0 & m \\
0 & 0 & 0 & 0 &  - m_{ x_1^2}^{x_2 x_3} & m_{ x_1^2}^{ x_2^2} & 0 & m_{x_1 x_2}^{x_2 x_3} &  - m_{x_1 x_2}^{ x_2^2} - m_{x_1 x_3}^{x_2 x_3} & m_{x_1 x_3}^{ x_2^2} \\
0 & 0 & 0 &  -m_{ x_1^2}^{ x_3^2} & m_{ x_1^2}^{x_2 x_3} & 0 & m_{x_1 x_2}^{ x_3^2} &  - m_{x_1 x_2}^{x_2 x_3} - m_{x_1 x_3}^{ x_3^2} & m_{x_1 x_3}^{x_2 x_3} & 0
\end{array}
\right].
\]
The columns correspond to $Y$:
\[x_1^5,x_1^4x_2,x_1^4x_3,x_1^3x_2^2,x_1^3x_2x_3,x_1^3x_3^2,x_1^2x_2^3, x_1^2x_2^2 x_3,x_1^2x_2x_3^2, x_1^2x_3^3.\]
This matrix can be verified to have rank at most $7$ (not full row rank) regardless of the values for the $m_{x^\gamma}^{x^\delta}$.
\end{example}

\subsection{Section \ref{sec:collinear}}

\begin{example}\label{ex:collinear}
    Let $n=2, r=4$, with $\phi$ satisfying Condition \ref{cond:collinear}.  Recall from the proof of Lemma \ref{lem:n23_linear} that the single polynomial $m\cdot m_{x_1^2}^{x_2^2}-\bb{m_{x_1^2}^{x_1x_2}}^2$ determines whether $\mathbf{A}$ drops rank.  But indeed this polynomial is equal to zero by Proposition \ref{prop:collinear_eqs}, and as we can show that $m\neq 0$, $\mathbf{A}$ is rank $1$.  As the two moment variables are $y_{x_1^5}, y_{x_1^4x_2}$, we can treat $y_{x_1^5}$ as a parameter, thus determining a value for $y_{x_1^4x_2}$ via the single linearly independent relation.

For simplicity we assume $\lambda_i=1, i=1, \dots, 4$.  We have 
\[
\mathcal{H}_{B, B_1}=\begin{bmatrix}
    \phi_{x_1} & \phi_{x_1^2} & \phi_{x_1x_2} & \phi_{x_1^3}\\
    \phi_{x_1^2} & \phi_{x_1^3} & \phi_{x_1^2x_2} & \phi_{x_1^4}\\
    \phi_{x_1x_2} & \phi_{x_1^2x_2} & \phi_{x_1x_2^2} & \phi_{x_1^3 x_2}\\
    \phi_{x_1^3} & \phi_{x_1^4} & \phi_{x_1^3x_2} & y_{x_1^5}
\end{bmatrix},\quad \mathcal{H}_{B, B_2}=\begin{bmatrix}
    \phi_{x_2} & \phi_{x_1x_2} & \phi_{x_2^2} & \phi_{x_1^2x_2}\\
    \phi_{x_1x_2} & \phi_{x_1^2x_2} & \phi_{x_1x_2^2} & \phi_{x_1^3x_2}\\
    \phi_{x_2^2} & \phi_{x_1x_2^2} & \phi_{x_2^3} & \phi_{x_1^2 x_2^2}\\
    \phi_{x_1^2x_2} & \phi_{x_1^3x_2} & \phi_{x_1^2x_2^2} & y_{x_1^4x_2}
\end{bmatrix}.
\]
Now letting $\Z_B^{\hat{i}, \hat{x^\alpha}}$ denote the $3\times 3$ submatrix of $\Z_B$ obtained by removing row $i$ and column $x^\alpha$, 
\[
\Z_B^{-1}=\frac{1}{\operatorname{det}\bb{\Z_B}}\begin{bmatrix}
    \operatorname{det}\bb{\Z_{B}^{\hat{1}, \hat{1}}} & -\operatorname{det}\bb{\Z_{B}^{\hat{2}, \hat{1}}} & \operatorname{det}\bb{\Z_{B}^{\hat{3}, \hat{1}}} & -\operatorname{det}\bb{\Z_{B}^{\hat{4}, \hat{1}}}\\
    -\operatorname{det}\bb{\Z_{B}^{\hat{1}, \hat{x_1}}} & \operatorname{det}\bb{\Z_{B}^{\hat{2}, \hat{x_1}}} & -\operatorname{det}\bb{\Z_{B}^{\hat{3}, \hat{x_1}}} & \operatorname{det}\bb{\Z_{B}^{\hat{4}, \hat{x_1}}}\\
    \operatorname{det}\bb{\Z_{B}^{\hat{1}, \hat{x_2}}} & -\operatorname{det}\bb{\Z_{B}^{\hat{2}, \hat{x_2}}} & \operatorname{det}\bb{\Z_{B}^{\hat{3}, \hat{x_2}}} & -\operatorname{det}\bb{\Z_{B}^{\hat{4}, \hat{x_2}}}\\
    -\operatorname{det}\bb{\Z_{B}^{\hat{1}, \hat{x_1^2}}} & \operatorname{det}\bb{\Z_{B}^{\hat{2}, \hat{x_1^2}}} & -\operatorname{det}\bb{\Z_{B}^{\hat{3}, \hat{x_1^2}}} &\operatorname{det}\bb{\Z_{B}^{\hat{4}, \hat{x_1^2}}}
\end{bmatrix},
\]
and in particular, $\operatorname{det}\bb{\Z_{B}^{\hat{3}, \hat{x_1^2}}}=0$.  Therefore, $\mathbf{M}_1^B=\mathcal{H}_{B, B_1}\mathcal{H}_{B, B}^{-1}=\mathcal{H}_{B, B_1}\Z_B^{-1}\Z_B^{-\top}$ is such that
\[\mathcal{H}_{B, B_1}\Z_B^{-1}=\begin{bmatrix}
    \z_{1, 1} & \z_{2, 1} & \z_{3, 1} & t\z_{1, 1}+(1-t)\z_{2, 1} \\ 
    \z_{1, 1}^2 & \z_{2, 1}^2 & \z_{3, 1}^2 & \bb{t\z_{1, 1}+(1-t)\z_{2, 1}}^2 \\ 
    \z_{1, 1}\z_{1, 2} & \z_{2, 1}\z_{2, 2} & \z_{3, 1}\z_{3, 2} & \bb{t\z_{1, 1}+(1-t)\z_{2, 1}}\bb{t\z_{1, 2}+(1-t)\z_{2, 2}} \\ 
    f_{x_1^3, 1}(y_{x_1^5}) & f_{x_1^3, 2}(y_{x_1^5}) & f_{x_1^3, 3}(y_{x_1^5}) & f_{x_1^3, 4}(y_{x_1^5})
\end{bmatrix},\]
where $f_{x_1^3, j}$ are linear functions in $y_{x_1^5}$.  In particular, we can calculate explicitly
\begin{align*}
    f_{x_1^3, 3}(y_{x_1^5})
    &=\phi_{x_1^3}\operatorname{det}\bb{\Z_B^{\hat{3}, \hat{1}}}-\phi_{x_1^4}\operatorname{det}\bb{\Z_B^{\hat{3}, \hat{x_1^2}}}+\phi_{x_1^3x_3}\operatorname{det}\bb{\Z_B^{\hat{3}, \hat{x_2}}}-y_{x_1^5}\operatorname{det}\bb{\Z_B^{\hat{3}, \hat{x_1^2}}}\\
    &=\phi_{x_1^3}\operatorname{det}\bb{\Z_B^{\hat{3}, \hat{1}}}-\phi_{x_1^4}\operatorname{det}\bb{\Z_B^{\hat{3}, \hat{x_1^2}}}+\phi_{x_1^3x_3}\operatorname{det}\bb{\Z_B^{\hat{3}, \hat{x_2}}}\\
    &=\z_{3, 1}^3.
\end{align*}
Then we immediately have $\mathbf{M}_1^B(y_{x_1^5})\z_3^B=\z_{3, 1}\z_3^B$, so that $\z_3^B$ is always an eigenvector of the multiplication matrices regardless of the value of $y_{x_1^5}$.

On the other hand, consider a vector $\z'=s\z_1+(1-s)\z_2$ for any $s\in (0,1 )$.  We have ${\z'}^B=s\z_1^B+(1-s)\z_2^B-s(1-s)\begin{bmatrix}
    0\\0\\0\\(\z_{1, 1}-\z_{2, 1})^2
\end{bmatrix}$; then a calculation shows $\mathbf{M}_1^B(y_{x_1^5}){\z'}^B$ is equal to
\[
 \begin{bmatrix}
    s\z_{1, 1}+(1-s)\z_{2, 1}\\
    (s\z_{1, 1}+(1-s)\z_{2, 1})^2\\
    (s\z_{1, 1}+(1-s)\z_{2, 1})(s\z_{1, 2}+(1-s)\z_{2, 2})\\
    g_{3, 0}(y_{x_1^5})+g_{3, 1}(y_{x_1^5})(s\z_{1, 1}+(1-s)\z_{2, 1})+g_{3, 2}(y_{x_1^5})(s\z_{1, 2}+(1-s)\z_{2, 2})+g_{3, 3}(y_{x_1^5})(s\z_{1, 1}+(1-s)\z_{2, 1})^2
\end{bmatrix},
\]
for linear functions $g$ in $y_{x_1^5}$, so we consider
\begin{multline*}
     g_{3, 0}(y_{x_1^5})+g_{3, 1}(y_{x_1^5})(s\z_{1, 1}+(1-s)\z_{2, 1})+g_{3, 2}(y_{x_1^5})(s\z_{1, 2}\\+(1-s)\z_{2, 2})+g_{3, 3}(y_{x_1^5})(s\z_{1, 1}+(1-s)\z_{2, 1})^2\\
    =  (s\z_{1, 1}+(1-s)\z_{2, 1})^3.
\end{multline*}
This is a cubic polynomial in $s$, with coefficients the entries of $\phi$, $\z_1, \z_2$, and the parameter $y_{x_1^5}$.  There are three distinct eigenvectors when there are three distinct solutions to this cubic, which is a condition on the discriminant.  For generic choices of value for $y_{x_1^5}$ the discriminant is nonzero, as checked by specialization.

Because there are no quadratic relations, all decompositions are specified by the linear relations: the point $\z_3$ and $s_1\z_1+(1-s_1)\z_2, s_2\z_1+(1-s_2)\z_2, s_3\z_1+(1-s_3)\z_2$, where $s_1, s_2, s_3$ are the three distinct solutions corresponding to the above cubic for a fixed value for the variable $y_{x_1^5}$.  That is, as there is one free parameter, the dimension of the space of decompositions is one-dimensional -- geometrically, $\z_3$ along with the line passing through $\z_{1}$ and $\z_2$.  This recovers case (4b)(ii) of Theorem 3.7 of \cite{MOURRAIN2020347}.
\end{example}

\section{Additional Proofs for Section \ref{sec:monomials}}\label{app:b}
\subsection{Section \ref{sec:monomialeqs}}

\begin{proof}[Proof of Lemma \ref{lem:monomial_G}]
    As $\mathcal{H}_{B, B}$ is symmetric lower antitriangular, $\mathcal{G}$ is symmetric upper antitriangular.  By definition, for $x^\gamma, x^{\gamma'}\in B$, $\mathcal{G}_{x^\gamma}^\top \mathcal{H}_{B, x^{\gamma'}}=1$ if $\gamma=\gamma'$ and $0$ otherwise.

    Explicitly, we have
    \[\mathcal{G}_{x^\gamma}^\top \mathcal{H}_{B, x^{\gamma'}}=\sum_{\delta}\mathcal{G}_{x^\gamma, x^{\overline{d}-\delta}}\mathcal{H}_{x^{\overline{d}-\delta}, x^{\gamma'}}.\]
    $\mathcal{H}_{\overline{d}-\delta, \gamma'}$ is nonzero when either $\delta=\gamma'$, so that the corresponding entry is $1$, or $(x^\delta, x^{\gamma'})\in S_k$ for $k\geq 1$, in which case the corresponding entry is $y_{x^{\overline{d}-\delta+\gamma'}}$ (note that it suffices to consider only those $\delta$ such that $|\delta|\leq |\gamma'|$; for $|\delta|>|\gamma|$ the size of $\overline{d}-\delta+\gamma'$ is too small so the entry of $\mathcal{H}_{B, B}$ is zero).  Then
    \begin{multline*}
        \mathcal{G}_{x^\gamma}^\top \mathcal{H}_{B, x^{\gamma'}}=\mathcal{G}_{x^\gamma, x^{\overline{d}-\gamma'}}+\sum_{(x^\delta, x^{\gamma'})\in S_k, k\geq 1}\mathcal{G}_{x^\gamma, x^{\overline{d}-\delta}}\mathcal{H}_{x^{\overline{d}-\delta}, x^{\gamma'}}\\=\mathcal{G}_{x^\gamma, x^{\overline{d}-\gamma'}}+\sum_{(x^\delta, x^{\gamma'})\in S_k, k\geq 1}\mathcal{G}_{x^\gamma, x^{\overline{d}-\delta}}y_{x^{\overline{d}-\delta+\gamma'}}.
    \end{multline*}

    In particular, for $\gamma=\gamma'$ 
    \[\mathcal{G}_{x^\gamma, x^{\overline{d}-\gamma}}=1-\sum_{(x^\delta, x^\gamma)\in S_k, k\geq 1}\mathcal{G}_{x^\gamma, x^{\overline{d}-\delta}}y_{x^{\overline{d}-\delta+\gamma}}\]
    and when $\gamma\neq \gamma'$
    \[\mathcal{G}_{x^\gamma, x^{\overline{d}-\gamma'}}=-\sum_{(x^\delta, x^{\gamma'})\in S_k, k\geq 1}\mathcal{G}_{x^\gamma, x^{\overline{d}-\delta}}y_{x^{\overline{d}-\delta+\gamma'}}.\]
    
    First, note that for $|\gamma|>|\gamma'|$ $\mathcal{G}_{x^\gamma, x^{\overline{d}-\gamma'}}=0$; this can be seen easily from upper antitriangularity of $\mathcal{G}$.  Assuming now $|\gamma|\leq |\gamma'|$ we compute $\mathcal{G}_{x^\gamma, x^{\overline{d}-\gamma'}}$ for $(x^\gamma, x^{\gamma'})\in S_k$ for increasing $S_k$.  Let $k=0$; then $0\leq |\gamma'|-|\gamma|\leq d_0$.  But the summation is vacuous: suppose there exists $\delta$ such that $(x^\delta, x^{\gamma'})\in S_k, k\geq 1$.  For $|\gamma|>|\delta|$ we know $\mathcal{G}_{x^\gamma, x^{\overline{d}-\delta}}=0$, so we can also assume $|\gamma|\leq |\delta|$.  But then we have $|\gamma'|-|\delta|\geq d_0+1$ and $-\gamma\geq -\delta$, so $|\gamma'|-|\gamma|\geq d_0+1$, a contradiction.  Therefore, for $(x^\gamma, x^{\gamma'})\in S_0$, $\mathcal{G}_{x^\gamma, x^{\overline{d}-\gamma'}}=1$ is $\gamma=\gamma'$ and is $0$ otherwise.
    
    Now suppose $(x^\gamma, x^{\gamma'})\in S_k$ for some $k\geq 1$.  In this case the entry is always
    \[\mathcal{G}_{x^\gamma, x^{\overline{d}-\gamma'}}=-\sum_{(x^\delta, x^{\gamma'})\in S_k, k\geq 1}\mathcal{G}_{x^\gamma, x^{\overline{d}-\delta}}y_{x^{\overline{d}-\delta+\gamma'}}.\]
    In fact, we can always upper bound $k$ in the summation.  Suppose $(x^\delta, x^{\gamma'})\in S_\ell$ for some $\ell> k$.  Then by the same argument as before, as $|\gamma|\leq |\delta|$ we have
    \[|\gamma'|-|\gamma|\geq |\gamma'|-|\delta|\geq \ell(d_0+1),\]
    a contradiction.  On the other hand, for $\ell\leq k$ there will always exist some $(x^\delta, x^{\gamma'})\in S_\ell$ (e.g., for $\ell=k$ just take $\delta=\gamma$).  Thus, we can write
    \[\mathcal{G}_{x^\gamma, x^{\overline{d}-\gamma'}}=-\sum_{\ell=1}^k \sum_{(x^\delta, x^{\gamma'})\in S_\ell} \mathcal{G}_{x^\gamma, x^{\overline{d}-\delta}}y_{x^{\overline{d}-\delta+\gamma'}}.\]
    Let us be more precise.  We can see that if $(x^\delta, x^{\gamma'})\in S_\ell$, that we immediately have the upper bound $|\delta|-|\gamma|\leq (k-\ell+1)(d_0+1)-1$, i.e., if $(x^\delta, x^{\gamma'})\in S_\ell$ and $(x^\gamma, x^\delta)\in S_{\ell'}$, then $\ell'\leq d-\ell$.  In particular, if $(x^\delta, x^{\gamma'})\in S_k$ then $(x^\gamma, \delta)\in S_0$; as we know that for $(\gamma, x^\delta)\in S_0$, $\mathcal{G}_{x^\gamma, x^{\overline{d}-\delta}}\neq 0$ if and only if $\gamma=\delta$, we finally write
    \[\mathcal{G}_{x^\gamma, x^{\overline{d}-\gamma'}}=-y_{x^{\overline{d}-\gamma+\gamma'}}-\sum_{\ell=1}^{k-1}\sum_{(x^\delta, x^{\gamma'})\in S_\ell}\mathcal{G}_{x^\gamma, x^{\overline{d}-\delta}}y_{x^{\overline{d}-\delta+\gamma'}}.\]
    It is easy to induct via this expression and see that the degree of $\mathcal{G}_{x^\gamma, x^{\overline{d}-\gamma'}}$ is $k$ and the only moment variable in $Y_k$ is $y_{x^{\overline{d}-\gamma+\gamma'}}$.
\end{proof}

\begin{proof}[Proof of Lemma \ref{lem:monomial_eq}]
    We can see that
    \[\operatorname{det}\mathcal{H}_{B\cup\{x^{\alpha}x_i\}, B\cup\{x^\beta x_j\}}=-y_{x^{\alpha+\beta}x_i x_j}+\sum_{\gamma, \gamma'}\mathcal{H}_{x^\alpha x_i, x^\gamma}\mathcal{G}_{x^\gamma, x^{\overline{d}-\gamma'}}\mathcal{H}_{x^{\overline{d}-\gamma'}, x^\beta x_j}.\]
    Then the expression simply comes from the fact that $\mathcal{G}_{x^\gamma, x^{\overline{d}-\gamma}}=1$.  

    By induction we see that for fixed $\ell$,
    \[\operatorname{deg}\bb{\sum_{\substack{|\gamma|\geq d-|\alpha|,\\ |\gamma'|\leq |\beta|-d_0,\\(x^\gamma, x^{\gamma'})\in S_\ell}}y_{x^{\alpha+\gamma}x_i}\mathcal{G}_{x^\gamma, x^{\overline{d}-\gamma'}}y_{x^{\overline{d}-\gamma'+\beta} x_j}}=\ell+2,\]
    so that if $x^{\alpha+\beta}x_i x_j\in Y_k$ then $\operatorname{deg}\bb{\mathcal{H}_{B\cup \{x^\alpha x_i\}, B\cup \{x^\beta x_j\} }}=k$ with degree $\ell$ part graded in the above manner.  For some fixed $\ell$ we calculate the sizes of the variables present.  First, from $(x^\gamma, x^{\gamma'})\in S_\ell$ we can see $|\gamma|\leq |\gamma'|-\ell(d_0+1)$.  Then
    \begin{align*}
        |\alpha|+|\gamma|+1
        &\leq |\alpha|+|\gamma'|-\ell(d_0+1)+1\\
        &\leq |\alpha|+|\beta|+2-(\ell+1)(d_0+1)\\
        &\leq d+(k-(\ell+1))(d_0+1).
    \end{align*}
    Thus, if $x^{\alpha+\gamma }x_i\in Y_{\ell'}$ then we must have $\ell'\leq k-(\ell+1)$.  The same idea works for $x^{\overline{d}-\gamma'+\beta} x_j$.  We also now that the maximum size of any variable that appears in $\mathcal{G}_{x^\gamma, x^{\overline{d}-\gamma'}}$ is in $Y_\ell$.  Therefore, we have again that the only variable in $Y_k$ in the equation is $x^{\alpha+\beta}x_ix_j$.
\end{proof}

\end{appendices}

\end{document}